\documentclass[11pt]{article}
\usepackage[a4paper, total={5.8in, 8in}]{geometry}
\usepackage{tikz}

\usepackage{biblatex}
\usepackage[nottoc,numbib]{tocbibind}
\usepackage{filecontents}

\usepackage{microtype}
\usepackage{lipsum}
\usepackage{graphicx}
\usepackage{float}
\usepackage[utf8]{inputenc}
\usepackage{biblatex}
\usepackage[hidelinks]{hyperref}
\usepackage{xurl}
\hypersetup{breaklinks=true}
\usepackage{mathtools}
\usepackage{amssymb}
\usepackage{setspace}
\usepackage[english]{babel}
\usepackage{amsmath}
\usepackage{tikz}
\usepackage{framed}
\usepackage{caption}
\usepackage[hang,flushmargin]{footmisc}
\hypersetup{colorlinks,linkcolor={blue},citecolor={green},urlcolor={red}}  
\usepackage{amsthm}
\usepackage{xcolor}
\usepackage{musicography}
\hypersetup{
  urlcolor=black}

\usepackage{csquotes}

\newtheorem{theorem}{Theorem}

\newtheorem{corollary}{Corollary}

\newtheorem{definition}{Definition}

\newtheorem{example}{Example}

\newtheorem{lemma}{Lemma}

\newtheorem*{question}{Question}

\newtheorem{proposition}{Proposition}
\newtheorem{remark}{Remark}

\makeatletter

\newcommand\ackname{Acknowledgements}
\if@titlepage
\newenvironment{acknowledgements}{%
	\titlepage
	\null\vfil
	\@beginparpenalty\@lowpenalty
	\begin{center}%
		\bfseries \ackname
		\@endparpenalty\@M
\end{center}}%
{\par\vfil\null\endtitlepage}
\else

\fi

\newcommand{\ie}{\emph{i.e.} }
\newcommand{\cf}{\emph{cf.} }

\newcommand{\R}{\mathbb{R}} \newcommand{\C}{\mathbb{C}}
\newcommand{\Z}{\mathbb{Z}}

\newcommand{\norm}[1]{\left\lVert#1\right\rVert}
\newcommand{\Lap}{\Delta}

\title{Fueter sections and $\Z_2$-harmonic 1-forms}
\author{Saman Habibi Esfahani, Yang Li}
\date{\today}

\begin{document}
	\maketitle
	
	\begin{abstract}
		Motivated by a conjecture of Donaldson and Segal on the counts of monopoles and special Lagrangians in Calabi-Yau 3-folds, we prove a compactness theorem for Fueter sections of charge 2 monopole bundles over 3-manifolds: Let $u_k$ be a sequence of Fueter sections of the charge 2 monopole bundle over a closed oriented Riemannian 3-manifold $(M,g)$, with $L^\infty$-norm diverging to infinity. Then a renormalized sequence derived from $u_k$ subsequentially converges to a non-zero $\Z_2$-harmonic 1-form $\mathcal{V}$ on $M$ in the $W^{1,2}$-topology. 
	\end{abstract}

    \tableofcontents
	
	\section{Introduction}
	
	The Fueter operator is a non-linear generalization of the Dirac operator defined over oriented Riemannian 3- and 4-manifolds, where the spinor bundle is replaced with a hyperk\"ahler target manifold bundle. In this paper, we focus on the 3-dimensional case.
	
	\begin{definition}
		Let $(M^3,g)$ be a closed oriented Riemannian 3-manifold and $X$ a hyperk\"ahler manifold. Let $ \mathfrak{X} \rightarrow M$ be a fiber bundle with fibers isometric to $X$, equipped with a connection $\nabla$ and an isometric isomorphism $I$ from the unit tangent bundle of $M$ to the 2-sphere bundle of complex structures on the fiber of  $\mathfrak{X}$. A section $u \in \Gamma (M,\mathfrak{X})$ is called a \emph{Fueter section} if
		\[\mathfrak{F}(u) := I(e_1) \nabla_{e_1} u + I(e_2) \nabla_{e_2} u + I(e_3) \nabla_{e_3} u = 0,
		\]
		where $(e_1, e_2, e_3)$ is any choice of local orthonormal frame on $M$. The operator $\mathfrak{F}$, which is independent of the chosen local frame, is called the Fueter operator.
	\end{definition}

	The Fueter sections appear in many problems in gauge theory and calibrated geometry. For instance:
	\begin{itemize}
		
		\item On a Calabi-Yau 3-fold, a sequence of Calabi-Yau monopoles \((A_i, \Phi_i)\) with increasingly large mass parameters can have curvature concentrated near a special Lagrangian 3-fold \(L\), such that the behavior of \((A_i, \Phi_i)\) restricted to the normal directions of \(L\) is modeled on monopoles on \(\R^3\). In the adiabatic limit, the variation of these monopoles along \(L\) is encoded into a Fueter section of a hyperk\"ahler bundle over \(L\), where the fibers are given by a moduli space of monopoles on \(\R^3\). This observation led Donaldson and Segal \cite{MR2893675} to conjecture that a count of Calabi-Yau monopoles agrees with a weighted count of special Lagrangian 3-folds, where the weights for the special Lagrangians are defined by the count of the Fueter sections of the monopole hyperk\"ahler bundles. This relation is an example of gauge theory/calibrated geometry duality and resembles Taubes' $SW=Gr$ theorem on symplectic 4-manifolds \cite{MR1798140}.

		\item Taubes \cite{MR1708781} proposed studying Fueter sections and generalized Seiberg-Witten equations on 3-manifolds with the motivation of defining new 3-manifold invariants. Hohloch-Noetzel-Salamon \cite{MR2529942} proposed a Floer theory for 3-manifolds with a `hypercontact structure', where the critical points are Fueter maps from the 3-manifold to a hyperk\"ahler target space $X$, and the gradient flowlines are given by counting the 4-dimensional analogues of the Fueter maps.

		\item Compactness problems for the generalized Seiberg-Witten equation typically lead to Fueter sections on the domain 3- or 4-manifolds, where the hyperk\"ahler target manifold $X$ arises from a hyperk\"ahler quotient construction as studied by Haydys \cite{MR2980921} and Doan-Walpuski \cite{MR4122245, MR4255067, MR4085667}.

		\item Doan-Rezchikov proposed a conjectural 2-category associated to a hyperkähler manifold $M$, defined using a count of Fueter maps \cite{doan2022holomorphic}.
		
	\end{itemize}

In all these problems, the central issue is the compactness problem. A partial compactness result is established by Walpuski \cite{MR3718486}; however, there are two essential assumptions there that fail in our setup: (i) the target hyperk\"ahler fibers are compact, and (ii) there exists a uniform bound on the Dirichlet energy of the Fueter sections. Notably, this energy assumption is not automatic, as the energy of a Fueter section is rarely a topological invariant. 

In this setup, Walpuski shows smooth convergence away from a singular locus with codimension at least two. The non-compactness has two sources: the bubbling-off of holomorphic spheres transverse to the singular locus and the formation of non-removable singularities in a set of $\mathcal{H}^1$-measure zero. Similarly, in the work of Bellettini-Tian \cite{MR3941492}, when addressing compactness issues for a weakly converging sequence of triholomorphic maps, it is assumed that there is a uniform bound on the Dirichlet energy.

\subsubsection*{Fueter sections of monopole bundles} 

Let $X_{AH}=\text{Mon}_2^\circ(\R^3)$ be the \emph{Atiyah-Hitchin manifold}, i.e., the non-compact 4-dimensional hyperk\"ahler manifold with $\pi_1(X) = \Z_2$, defined as the moduli space of centered $SU(2)$-monopoles on $\R^3$ with charge 2. Its double cover $\tilde{X}_{AH}= \widetilde{\text{Mon}}_2^\circ(\R^3)$ is a simply connected hyperk\"ahler manifold. There is a rotating $SO(3)$-action on $X_{AH}$ induced from the natural $SO(3)$-action on $\R^3$, and an $SO(3)$-equivariant projection map $X_{AH}\to \R^3/\Z_2$, which asymptotically exhibits the Gibbons-Hawking coordinates near the infinity of $X_{AH}$.

Let $P_{SO(3)} \to M$ denote the frame bundle of the 3-manifold $(M,g)$. We define the Atiyah-Hitchin monopole bundle, also known as the charge 2 monopole bundle, by 
\[\mathfrak{X}_{AH}:= P_{SO(3)} \times_{SO(3)} X_{AH} \to M,\] 
equipped with the induced Levi-Civita connection $\nabla$ and the covariantly constant isomorphism $I$ from the unit tangent bundle of $M$ to the 2-sphere bundle of complex structures on the Atiyah-Hitchin fibers. 

The associated bundle construction gives rise to a fiberwise projection map
	\[
	\pi: \mathfrak{X}_{AH} \to P_{SO(3)}\times_{SO(3)} \R^3/\Z_2\simeq T^*M/\Z_2.
	\]
	In particular, any section $u$ of $\mathfrak{X}_{AH}$ projects to a section $\pi(u)$ of $T^*M/\Z_2$, called a $\Z_2$-valued 1-form.  We shall write $\norm{u}_{L^\infty}$ as the $L^\infty$-norm of the $\Z_2$-valued 1-form $\pi(u)$. The bulk of this paper will study the compactness problem for the space of Fueter sections $u: M\to \mathfrak{X}_{AH}$.

	The main previous result in \cite{esfahani2023towards, MR4495257} is that if a sequence of such Fueter sections remains \emph{uniformly bounded} inside the Atiyah-Hitchin monopole bundle, then after passing to a subsequence, they converge in $C^\infty$-topology to a limiting Fueter section. This is based on two main ingredients:
	\begin{itemize}
		\item Using Walpuski's compactness theorem, if codimension two bubbling phenomenon occurs, we could extract an (anti-)holomorphic sphere bubble for some choice of the complex structure of the Atiyah-Hitchin manifold. However, the Atiyah-Hitchin manifold is exact with respect to all the symplectic forms compatible with the hyperk\"ahler structure, so these holomorphic curves do not exist. Thus, the compactness problem does not involve codimension two bubbling, and no energy is lost in the limit.
		
		\item  The formation of an essential singularity is modeled on a Fueter map from $\R^3 \setminus \{0\}$ to $X_{AH}$, whose tangent map at infinity is given by the cone over a contractible `triholomorphic map' $S^2\to X_{AH}$. A maximum principle argument shows that the only triholomorphic map into the Atiyah-Hitchin manifold comes from the minimal $\mathbb{RP}^2$ in the Atiyah-Hitchin manifold, which is not contractible, so this model Fueter map cannot exists as a tangent map.
	\end{itemize}

	The upshot is that the only possible source of non-compactness comes from a totally different mechanism: \emph{the Fueter sections diverge to infinity}. In particular, there is no longer any topological energy formula argument that guarantees an a priori bound on the $L^2$-energy, and the known compactness theorems do not apply.

	\begin{definition}
		We say $\mathcal{V}$ is a $\Z_2$-harmonic 1-form on a closed Riemannian 3-manifold $(M,g)$, if $\mathcal{V}$ is a H\"older continuous section of $T^*M/\Z_2$, such that on the complement of the closed set $Z=\mathcal{V}^{-1}(0)$, the section $\mathcal{V}$ locally admits a smooth lift to $T^*M$ satisfying $d\mathcal{V}=d^*\mathcal{V}=0$, and the total energy $\int_{M\setminus Z} |\nabla \mathcal{V}|^2<+\infty$.  
	\end{definition}

	Our main compactness theorem is the following:

	\begin{theorem}\label{mainthm}
		Suppose $u_k$ is a sequence of Fueter sections of $ \mathfrak{X}_{AH}$ over $(M,g)$ with $L^\infty$-norm diverging to infinity. Then, the renormalized sequence of $\Z_2$-valued 1-forms defined by $\pi(u_k)/ \norm{u_k}_{L^\infty}$ subsequentially converges to a non-zero $\Z_2$-harmonic 1-form $\mathcal{V}$ on $M$ in the $W^{1,2}$-topology.
	\end{theorem}

	\begin{remark} \emph{
			The work of B. Zhang \cite{MR4596625} implies that the zero set of any non-zero $\Z_2$-harmonic 1-form on a closed 3-manifold has finite Hausdorff $\mathcal{H}^1$-measure and finite Minkowski content, and is 1-rectifiable (\cf Section \ref{sect:Z2harmonic1form}). }
	\end{remark}

	\begin{remark}\emph{
			The $W^{1,2}$-topology on coordinate charts is defined on the Sobolev space of multi-valued functions, as in De Lellis-Spadaro \cite{MR2663735} or Almgren \cite{MR1777737}. Since our sequence $\pi(u_k)/\norm{u_k}_{L^\infty}$ is normalized to be uniformly bounded in $L^\infty$, the $W^{1,2}$ convergence implies $L^p$ convergence for any $p>0$, but it does not imply $C^0$ convergence.} 
	\end{remark}

	\begin{remark}\emph{
			In the case where the 3-manifold is given by the product of a Riemann surface and $S^1$, we discover some delicate analytic issues that seem unexpected. For instance, if we replace $W^{1,2}$ with $C^0$ in Theorem \ref{mainthm}, the convergence statement $\pi(u_k)/\norm{u_k}_{L^\infty} \to \mathcal{V}$ would be false. The sequence of renormalized $\Z_2$-valued 1-forms $\pi(u_k)/\norm{u_k}_{L^\infty}$ fails to have any uniform modulus of continuity. }

		\emph{This situation may be compared with the generalized Seiberg-Witten equation studied by Walpuski-Zhang \cite{MR4340726}, which involves a gauge field $A$ and a coupled spinor field $\Phi$ on a 3-manifold. The main question is to study the compactness problem for a sequence of solutions $(A_k, \Phi_k)$, and the main difficulty is when $\norm{\Phi_k}_{L^2}\to +\infty$ along the sequence. It was proved, in particular, that the magnitude of the rescaled spinor field $\Phi_k /\norm{\Phi_k}_{L^2}$ satisfies uniform H\"older continuity estimates. This contrast is an indication that the Fueter equation with Atiyah-Hitchin target hyperk\"ahler manifold is more singular than the generalized Seiberg-Witten equation from the PDE perspective.} 
	\end{remark}

\subsubsection*{Plan of the paper}

\noindent
In \textbf{Section \ref{sec:AH-manifold}}, we review the geometry of the Atiyah-Hitchin manifold, such as the Gibbons-Hawking ansatz for the asymptotic hyperk\"ahler metric, the equivariant projection map $X_{AH}\to \R^3/\Z_2$ (Subsection \ref{sec:AH-metric}), Donaldson's rational map description of the monopole moduli spaces (Subsection \ref{sect:affinevariety}), and the Hessian formula of \(SO(3)\)-symmetric functions on the Atiyah-Hitchin space (Subsection \ref{sec:Hessian}). We then define the Atiyah-Hitchin bundle over a Riemannian 3-manifold, set up the Fueter equation (Subsection \ref{sec:AH-bundle}), and prove a formula for the Dirichlet energy of Fueter sections (Subsection \ref{sec:energy}).

In \textbf{Section \ref{sec:examples}}, we use Donaldson's rational map description of the Atiyah-Hitchin manifold to identify all $S^1$-invariant Fueter sections on product 3-manifolds $S^1\times \Sigma$, where $\Sigma$ is a Riemann surface. These examples reveal several analytic subtleties that demonstrate the optimality of Theorem \ref{mainthm}.

In \textbf{Section \ref{sec:4}}, we begin the a priori PDE estimates on the Fueter sections $u$. Using the Hessian formula for rotation-invariant functions on the Atiyah-Hitchin space, and by a delicate curvature computation, we derive a Laplacian formula for zeroth-order rotation-invariant functions of $u$ (Lemma \ref{lem:C0Lap}), which implies a maximum principle for Fueter sections. 

We use a combination of energy identities, the coarea formula, and ODE comparison arguments to prove a power law decay estimate for the Dirichlet energy of sublevel sets of $|u|$. This, combined with a Weitzenb\"ock formula argument, implies that when $\norm{u}_{L^\infty}$ is large, the $\Z_2$-valued 1-form $\pi(u)$  is close to being harmonic in some $L^2$-integral sense, up to lower order errors, and the circle component contribution to the Dirichlet energy is a lower order quantity (Prop. \ref{prop:almostharmonic1form}).

Using the Laplacian computation for higher order derivatives of $u$, and an argument inspired by the $\epsilon$-regularity of harmonic maps, one can show higher order regularity of $u$, provided that $|u|$ does not oscillate too much (Prop. \ref{prop:localregularity}). On the other hand, the examples in Section \ref{sec:examples} show that higher order regularity cannot hold unconditionally.

Almgren's theory of multi-valued harmonic functions provides H\"older regularity and a good compactness theory for harmonic $\Z_2$-valued 1-forms. In our case, the $\Z_2$-valued 1-form $\pi(u)$ is not harmonic on the nose, but only approximately harmonic. Our strategy is then to take some of the key ingredients of Almgren's theory, such as energy competitor constructions, Lipschitz approximations, and Campanato-Morrey decay estimates, and establish analogous effective estimates for $\pi(u)$. 

These effective estimates include certain error terms not present in the Almgren's theory; such error terms are necessary because the examples in Section \ref{sec:examples} forbid us to improve the estimates to uniform H\"older estimates on the normalized $\Z_2$-valued 1-forms $\pi(u)/\norm{u}_{L^\infty}$ directly. Instead, we prove the weaker statement that there are sufficiently good Lipschitz approximations of $\pi(u)/\norm{u}_{L^\infty} $ which enjoy uniform H\"older estimates (Cor. \ref{cor:macroscopicHolder}).

The \textbf{Section \ref{sect:convergence}} concludes the proof of Theorem \ref{mainthm}. The point is that to develop the compactness theory of Section \ref{sect:convergence}, we can essentially replace $\pi(u)/\norm{u}_{L^\infty}$ by its Lipschitz approximation. Then we gain the technical advantage of convergence in H\"older topology, so we can find definite size neighborhood to locally lift the sequence of $\Z_2$-valued 1-forms to 1-forms, so the convergence theory for the $\Z_2$-valued 1-forms reduces to the standard theory for single valued 1-forms. In particular, we can show that the weak limit is harmonic, and the convergence is strong in the $W^{1,2}$-topology, which completes the proof of Theorem \ref{mainthm}.

In \textbf{Section \ref{sec:open-questions}}, we discuss several open questions related to the compactness problems of Fueter sections, motivated by  defining enumerative and Floer-theoretic invariants for 3-manifolds. This includes discussions on the \(C^{\infty}_{\text{loc}}\)-convergence problem and the case of higher charge monopole bundles.

Finally in the \textbf{Appendix \ref{appendix}},  we introduce a version of the frequency function and prove a few effective estimates on its monotonicity and boundedness. Improvement on these estimates is likely useful for addressing the open problems in Section \ref{sec:open-questions}.

\vspace{15pt}

\noindent
\textbf{Acknowledgement.}
We thank  Professor Simon Donaldson for helpful suggestions. S.E. thanks Lorenzo Foscolo and Mark Stern for discussions on related topics.	Y.L. is a Royal Society University Research Fellow at Cambridge University, and this work was partly done while Y.L. served as Clay Research Fellow at MIT.

\section{The Atiyah-Hitchin manifold}\label{sec:AH-manifold}

In this section, we briefly review the geometry of the Atiyah-Hitchin manifold $X_{AH} = {\text{Mon}}^{\circ}_2(\mathbb{R}^3)$, defined as the moduli space of centered $SU(2)$ monopoles on $\mathbb{R}^3$ with charge 2. For a more detailed exposition, refer to, see \cite{MR934202, foscolo2013moduli, MR4271388, MR1780369, MR2869404}.

\subsection{The Atiyah-Hitchin metric}\label{sec:AH-metric}

In this section, we recall the hyperk\"ahler metric on the Atiyah-Hitchin manifold and the Gibbons-Hawking model for its asymptotic region. 
    
The Atiyah-Hitchin manifold $X_{AH}$ is a 4-dimensional non-compact hyperk\"ahler manifold with an isometric $SO(3)$-action, and is double covered by a simply connected hyperk\"ahler manifold $\tilde{X}_{AH}=\widetilde{\text{Mon}}^{\circ}_2(\R^3)$ with the lifted $SO(3)$-action. Topologically, $\tilde{X}_{AH}$ is the total space of a degree $-4$ complex line bundle over $S^2$. The principal orbit of the $SO(3)$-action on $X_{AH}$ (resp. $\tilde{X}_{AH}$) is $SO(3)/\Z_2\times \Z_2$ (resp. $SO(3)/\Z_2$), where $\Z_2\times \Z_2$ embeds diagonally in $SO(3)$. Furthermore, $X_{AH}$ (resp. $\tilde{X}_{AH}$) has a minimal $\mathbb{RP}^2$ (resp. $S^2$) as its core. 
 
Let $\sigma_1, \sigma_2, \sigma_3$ be the invariant 1-forms on $SO(3)$ dual to the standard basis of its Lie algebra, so that
\[
	d\sigma_1= \sigma_2\wedge \sigma_3,\quad d\sigma_2= \sigma_3\wedge \sigma_1,\quad d\sigma_3= \sigma_1\wedge \sigma_2.
\]
In Euler angle coordinates $(\psi, \theta, \phi)$, they are the Maurer-Cartan 1-forms,
	\begin{align*}
		\sigma_1 = -\sin \psi d\theta+\sin\theta \cos\psi d\varphi, \quad \sigma_2 = \cos \psi d\theta + \sin\theta \sin\psi d\varphi, \quad \sigma_3 = d\psi+\cos\theta d\varphi.
	\end{align*}
As in \cite{MR535151}, the metric on $X_{AH}$ (resp. $\tilde{X}_{AH}$) is given by
	\begin{equation*}
		g_{AH} =
		d\xi^2 + a^2 \sigma_1^2 + b^2 \sigma_2^2 + c^2 \sigma_3^2,
	\end{equation*}
	where $\xi: X_{AH} \to [0, \infty)$ is an $SO(3)$-invariant real-valued function that measures the geodesic distance from the minimal $SO(3)$-orbit $\{ \xi=0\}$, which is a minimal surface diffeomorphic to  $\mathbb{RP}^2$ (resp. $S^2$). The coefficients $a, b$, and $c$ are functions of $\xi$ and determined by the following system of ODEs: 
	\begin{align} \label{ODE1}
		\frac{da}{d\xi} = \frac{a^2 - (b-c)^2}{2bc}, \quad 
		\frac{db}{d\xi} = \frac{b^2 - (c-a)^2}{2ca}, \quad 
		\frac{dc}{d\xi} = \frac{c^2 - (a-b)^2}{2ab}.
    \end{align}
	with the initial conditions $a(0)=0, -b(0) =c(0) = m$ for a positive constant $m$. Here we use the Gibbons-Manton convention $m = \pi$.

    The Atiyah-Hitchin manifold, along with its double cover, admits a hyperk\"ahler structure compatible with its metric, consisting of a triple of symplectic forms, denoted by $\omega_1$, $\omega_2$, and $\omega_3$.

\subsubsection*{The Gibbons-Hawking asymptote} 

It is well known that the Atiyah-Hitchin manifold admits an asymptotic description in terms of the Gibbons-Hawking ansatz. Let $H^k$ be the total space of the principal $U(1)$-bundle associated with the line bundle $\mathcal{O}(k)$ over $S^2$ radially extended to $\R^3 \setminus B_R$ for some $R > 0$. There is a $\Z_2$-action on $H^k$, which is a simultaneous involution of $\R^3$ and the fiber $S^1$. The Gibbons-Hawking ansatz produces the model hyperk\"ahler structure on $H^k$:
    \begin{align}\label{Gibbons-Hawking}
    \begin{cases}
	g^{GH}= V \sum_{i=1}^3 du_i^2 + V^{-1} \vartheta_{GH}^2,
	\\
	\omega_1^{GH}= \vartheta_{GH}\wedge du_1+ Vdu_2\wedge du_3,
	\\
	\omega_2^{GH}= \vartheta_{GH}\wedge du_2+ Vdu_3\wedge du_1,
	\\
	\omega_3^{GH}= \vartheta_{GH}\wedge du_3+ Vdu_1\wedge du_2,
    \end{cases}
    \end{align}
    where $u_1, u_2, u_3$ are Euclidean coordinates on $\R^3$, 
    \[V = 1 + \frac{k}{2\sqrt{u_1^2+u_2^2+u_3^2}},\] 
    is the Gibbons-Hawking potential, and $\vartheta_{GH}$ is an $S^1$-connection satisfying $*_3 d\vartheta_{GH}= -dV$ on $\R^3\setminus \{  0 \}$.

\begin{proposition}
[see eg. {\cite[Proposition 2.10.]{MR4271388}}] The complement of a compact region in the Atiyah-Hitchin manifold is diffeomorphic to $H^{-4}/\Z_2$, such that the hyperk\"ahler structure $(g_{AH},\omega_i)$ satisfies
\[
g_{AH}= g^{GH} + O(\text{exp. small}),\quad \omega_i= \omega_i^{GH} +O(\text{exp. small}),
\]
where the $O(\text{exp. small})$ denotes an error term which decays exponentially in distance, to all orders of derivatives.  
\end{proposition}

We briefly review how to see this at the level of the metric $g_{AH}$. Let $t$ be a new variable, defined by $(-abc)dt = d\xi$, and let
	\begin{align*}
		w_1 = bc, \quad w_2 = ca, \quad w_3 = ab.
	\end{align*}
	The ODE system (\ref{ODE1}) is equivalent to the following Darboux-Halphen system,
	\begin{align*} 
		\frac{dw_i}{dt} + \frac{dw_j}{dt} = -2 w_i w_j, \quad \text{when} \quad i \neq j.
	\end{align*}
	The Darboux-Halphen system can be solved for $w_1, w_2, w_3$ in terms of theta functions. In the asymptotic region, where $t \to 0^+$, we have
	\begin{align*}
		\begin{split}
			w_1(t) &=  - \frac{1}{t} + \text{exp. small},\\
			w_2(t) &= - \frac{1}{t} + \text{exp. small},\\
			w_3(t) &=  \frac{1}{4t^2} - \frac{1}{t} + \text{exp. small}.
		\end{split}
	\end{align*}
We define $\tilde{r}$ to be a smooth $SO(3)$-invariant non-negative function on the Atiyah-Hitchin manifold, which is equal to $\tilde{r} = \frac{1}{2t}$ for $t$ close to zero. For later convenience, we also require $\tilde{r}=0$ in a neighborhood of the minimal $\mathbb{RP}^2$.

Then we obtain the asymptotic formula for the Atiyah-Hitchin metric for small $t$, or equivalently for large $\tilde{r}$:
	\begin{equation}
		g_{AH} =  (1-\frac{2}{\tilde{r}})(d\tilde{r}^2 + \tilde{r}^2 ( \sigma_1^2 + \sigma_2^2))
		+ \frac{4}{1-\frac{2}{\tilde{r}}} \sigma_3^2  
		+ \text{exp. small}.
	\end{equation}
	In the Euler angle coordinates
	\begin{align*}
		g_{AH}&=V \left( d\tilde{r}^2+\tilde{r}^2(d\theta^2+ \sin^2\theta d\phi^2) \right) + V^{-1}(2d\psi + 2 \cos \theta d \phi)^2 + \text{exp. small} \\&=V g_{\R^3/\Z_2} + V^{-1}\vartheta^2 + \text{exp. small}.
	\end{align*}
	Here we identify $d\tilde{r}^2+\tilde{r}^2 (d\theta^2+ \sin^2\theta d\phi^2) $
	with the Euclidean metric on $\R^3/\Z_2$; the $\Z_2$ quotient arises from the stabilizer of the $SO(3)$-action on the principal orbit, and identifies the antipodal points  $\pm (u_1,u_2,u_3)\in \R^3$. Now $\tilde{r}$ agrees with the Euclidean distance to the origin on $\R^3/\Z_2$, and
    \[V =  1 - \frac{2}{\tilde{r}}= 1-\frac{2}{\sqrt{u_1^2+u_2^2+u_3^2}},\]  and $\vartheta = 2 d\psi + 2 \cos \theta d \phi$ is an $S^1$-connection, which is close to $\vartheta_{GH}$ in the Gibbons-Hawking ansatz up to an exponentially small error. 
    
    In conclusion, up to an exponentially small error, the asymptotic region of the Atiyah-Hitchin manifold can be matched with the Gibbons-Hawking ansatz with $k=-4$; this asymptotic model is sometimes referred to as a negative-mass Taub-NUT space.

	\begin{remark}\emph{
		If we replace $X_{AH}$ by its double cover $\tilde{X}_{AH}$, then each circle fiber is doubled, so the new $S^1$-connection is $d\psi + \cos \theta d \phi$ instead, and modulo an exponentially small error, we can identify the asymptote of the hyperk\"ahler structure on $\tilde{X}_{AH}$ as the Gibbons-Hawking ansatz with $k=-2$, up to a rescaling constant. In other words, the Atiyah-Hitchin manifold is a gravitational instanton of $D_0$-type, while its double cover is a gravitational instanton of $D_1$-type \cite[Section 3.2.2.]{MR3948228}.} 
	\end{remark}

\begin{remark}\emph{
	The geometric meaning of the Gibbons-Hawking coordinates on $\R^3/\Z_2$ is that when $(u_1,u_2, u_3)\in \R^3/\Z_2$ is large, then the points in the Atiyah-Hitchin manifold describe a superposition of two monopoles with charge one centered around $\pm(u_1,u_2,u_3)\in \R^3$.}
\end{remark}

	\noindent
	\textbf{The $SO(3)$-equivariant map $\pi$.}  The above discussion also specifies an $SO(3)$-equivariant projection map 
	\[
	\pi: X_{AH}\to \R^3/\Z_2.
	\]
    On each principal $SO(3)$-orbit, this map is obtained by the $SO(3)$-equivariant projection $SO(3)/(\Z_2\times \Z_2)\to S^2/\pm 1$, given by forgetting the $S^1$-coordinate $\psi$, followed by multiplying the image in $S^2/\pm 1$ with the factor $\tilde{r}$. In particular, since we arranged $\tilde{r}$ to be zero near the minimal orbit, $\pi$ projects a neighborhood of the minimal orbit to $0\in \R^3/\Z_2$.

	\subsection{The affine variety structure}\label{sect:affinevariety}

    In this section, we review an algebro-geometric model for monopole spaces.	As shown by Donaldson \cite{MR769355} and Hurtubise \cite{MR804459}, given a choice of direction in $\R^3$, the framed $k$-monopoles on $\R^3$ correspond to degree $k$ rational maps  $S: \mathbb{CP}^1 \to \mathbb{CP}^1$ of the form
	\[
	S(\zeta) = \frac{p(\zeta)}{q(\zeta)},
	\]
	where \( q(\zeta) \) is a monic polynomial of degree \( k \) and \( p(\zeta) \) is a polynomial of degree less than \( k \), with no common factors with \( q(\zeta) \).

    The Atiyah-Hitchin manifold parameterizes the \emph{centered charge $k=2$ monopoles} on $\R^3$. The centering imposes the extra condition that the coefficient of $\zeta^{k-1}$ in $q(\zeta)$ vanishes, and we need to take a further $\C^*$ quotient to forget the framing information. This leads to the algebro-geometric description
	\[
	X_{AH}=\{ S(\zeta)=  \frac{a_1\zeta + a_0}{\zeta^2 + b_0} \; | \; a_0,a_1,b_0 \in \C, a_0^2 + b_0 a_1^2 \neq 0 \}/(a_0,a_1)\sim (a_0 c, a_1 c), c\in \C^*.
	\]
	Its simply connected double cover is
	\begin{align*}
        \widetilde{\text{Mon}}^{\circ}_2(\R^3) &  \simeq 
		\{ (a_0,a_1, b_0) \in \C^3 \; |a_0^2 + b_0 a_1^2 = 1  \} \subset \C^3.
	\end{align*}  
	Here, the condition $a_0^2 + b_0 a_1^2 = 1 $ is a convention for finding a representative under the $\C^*$-equivalence relation.
 
    The Atiyah-Hitchin manifold $X_{AH}$ can be identified as the quotient by the involution $(a_0,a_1,b_0) \to (-a_0, -a_1, b_0)$, or equivalently
	\[
	X_{AH} \simeq \{ (z_1, z_2, z_3, z_4)\in \C^4 \; | \; z_1z_3=z_2^2, z_1+ z_3z_4=1  \},
	\]
	where we identify $z_1=a_0^2, z_2= a_0a_1, z_3=a_1^2, z_4=b_0$.

	There is a copy of $U(1)\subset SO(3)$ fixing the preferred direction in $\R^3$. This $U(1)$ acts on the Atiyah-Hitchin manifold. In the rational map description \cite[eqn 2.11]{MR934202}, this $U(1)$-actions is given by
	\[
	\lambda\cdot S:  \zeta\mapsto \lambda^{-4} S(\lambda^{-1}\zeta), \quad \forall \lambda \in U(1),
	\]
	and to achieve the normalization convention $a_0^2+b_0 a_1^2=1$, we are required to adjust the resulting rational function by a multiplicative constant in $\C^*$. In the end, the $U(1)$-action is
	\[
	(a_0,a_1,b_0) \mapsto (  a_0, \lambda^{-1} a_1, \lambda^2 b_0), \quad \forall \lambda\in U(1).
	\]
	In terms of the coordinates on $X_{AH}$, this action is
	\[
	(z_1,z_2,z_3,z_4)\mapsto ( z_1, \lambda^{-1} z_2, \lambda^{-2} z_3, \lambda^2 z_4), \quad \forall \lambda\in U(1).
	\]

	\subsubsection{Asymptotic region of the space of based rational maps.}\label{sect:asymptoticregionrationalmap}

    Let's choose the $x_3$-direction as the preferred direction in $\R^3$. According to \cite[Prop. 3.12]{MR934202}, if a rational map $S(\zeta)$ associated to a framed uncentered $k$-monopole has the form 
	\[S(\zeta) = \sum_{i=1}^{k} \frac{a_i}{\zeta - b_i},\] 
	where $a_i \in \C^*$ and $|b_i - b_j| > R$, for a sufficiently large $R$, then the rational map represents a $k$-monopole that can be viewed as an approximate superposition of $1$-monopoles located at the points $\left(b_i, -\frac{1}{2}\log|a_i| \right) \in \C \times \R$ with phases $\arg(a_i)$. To apply this to the Atiyah-Hitchin manifold, a further translation is needed to move the center of mass of the two monopoles to the origin in $\R^3$.

The non-generic situation where $b_i$ can coincide has been considered in previous works, such as \cite{MR1392287, houghton1996monopole}. According to \cite[Page 5]{houghton1996monopole}, a rational map of the form 
	\[S(\zeta) = \frac{a\zeta^l + 1}{\zeta^k}, 
	\] 
	describes axisymmetric monopoles of charge $k-l$ located at $(0, 0, \pm \log |a|)$, along with an axisymmetric charge $2l - k$ monopole at the origin. The special case which plays a role in this paper, is when $k = 2$ and $l = 1$, and $a$ is large. This describes the superposition of a pair of charge one monopoles positioned at $(0, 0, \pm \log |a|)$. In particular, the Gibbons-Hawking coordinate is $(0, 0, \log |a|)\in \R^3/\Z_2$ up to $O(1)$-error.

\subsection{Hessian of radial functions}\label{sec:Hessian}

We recall the computation of Lotay-Oliveira for the Hessian of an $SO(3)$-symmetric function $f=F(\tilde{r}^2)$ on the Gibbons-Hawking spaces \cite{MR4753494}. The reader should keep in mind the main case $f= \tilde{r}^2=\sum u_i^2$, but the general $F$ affords more flexibility later. This will be used later in \(C^0\) Laplacian computation of $f$.

\begin{lemma}\label{lem:HessianAH}
	The Hessian of $f=F(\tilde{r}^2)$ on the Atiyah-Hitchin manifold is 
	\begin{align*}
		\text{Hess}_{X_{AH}} (f)  = 
		- \frac{2 F'}{\tilde{r} V^3} \vartheta^2 &+
		2\sum_{1 \leq i \leq 3} F'(1 + 
		\frac{1}{\tilde{r} V}) du_i^2
		\\&+ \sum_{1 \leq i,j \leq 3} (4F''-\frac{4F' }{\tilde{r}^3V}  ) u_iu_jdu_i du_j + Err,
	\end{align*}
	such that when $\tilde{r}$ is large, the error term decays exponentially: \[|Err|\leq C(|F'|+|F''|)e^{-\nu \tilde{r}},\]
    for some $\nu>0$. 
\end{lemma}

\begin{proof}
	The Atiyah-Hitchin manifold is asymptotic to the Gibbons-Hawking ansatz space up to exponentially small error $O(e^{-\nu \tilde{r}})$, so we shall compute the Hessian of $F(u_1^2+u_2^2+u_3^2)$ instead on the Gibbons-Hawking ansatz, using the result of Lotay-Oliveira. We take the orthonormal frame $\{ e_i\}_0^3$ whose dual coframe is given by
	\begin{align*}
		e^0 = V^{-\frac{1}{2}} \vartheta_{{GH}}, \quad 
		e^i = V^{\frac{1}{2}} du_i, \quad \text{for} \quad i \in \{1,2,3\},
	\end{align*}
	The Hessian in the Gibbons-Hawking ansatz space $X$ is
	\begin{align*}
		\text{Hess}_X f(e_i,e_j) = \nabla_{e_i} (\nabla_{e_j} f) - (\nabla_{e_i} e_j) \cdot f.
	\end{align*}
	In the computation below, we will use the summation convention.

	In these Gibbons-Hawking coordinates, from \cite[Appendix B]{MR4753494}, we have
	\begin{align*}
		\nabla_{e_0} e_0 &= \frac{1}{2V^{\frac{3}{2}}} \frac{\partial V}{\partial{u_i}} e_i,\\
		\nabla_{e_i} e_0 &= \frac{1}{2V^{\frac{3}{2}}} \epsilon_{ijk} \frac{\partial V}{\partial{u_j}}  e_k,\\
		\nabla_{e_0} e_i &= \frac{1}{2V^{\frac{3}{2}}} (- \frac{\partial V}{\partial{u_i}}  e_0 + \epsilon_{ijk}  \frac{\partial V}{\partial{u_j}}  e_k ),\\
		\nabla_{e_i} e_j &= \frac{1}{2V^{\frac{3}{2}}} (
		\epsilon_{ijk}
		\frac{\partial V}{\partial{u_k}}
		e_0 + 
		\frac{\partial V}{\partial{u_j}}
		e_i 
		- \delta_{ij}
		\frac{\partial V}{\partial{u_k}}  e_k ).
	\end{align*}
	Therefore, 
	\begin{align*}
		(\nabla_{e_0} e_0) \cdot f &= \frac{1}{2V^2 } 
		\frac{\partial V}{\partial{u_i}} \frac{\partial f}{\partial{u_i}},\\
		(\nabla_{e_i} e_0) \cdot f &= (\nabla_{e_0} e_i) \cdot f = \frac{1}{2V^2} \epsilon_{ijk} \frac{\partial V}{\partial{u_j}} \frac{\partial f}{\partial{u_k}},\\
		(\nabla_{e_i} e_j) \cdot f &= \frac{1}{2V^2} (
		\frac{\partial V}{\partial{u_j}}
		\frac{\partial f}{\partial{u_i}}
		- \delta_{ij}
		\frac{\partial V}{\partial{u_k}}  
		\frac{\partial f}{\partial{u_k}}).
	\end{align*}
	Since $f $ is constant on the circle fibers of the Gibbons-Hawking ansatz space, we have $e_0 \cdot f = 0$. Moreover,
	\begin{align*}
		\nabla_{e_i} (\nabla_{e_j} f) = \nabla_{e_i} \frac{1}{V^{\frac{1}{2}}} \frac{\partial f}{\partial u_j} = 
		\frac{1}{V} \frac{\partial^2 f}{\partial u_i \partial u_j} - \frac{1}{2V^2} \frac{\partial V}{\partial u_i} \frac{\partial f}{\partial u_j}.
	\end{align*}
	Furthermore, $\nabla_{e_0} (\nabla_{e_i} f )= \nabla_{e_i} (\nabla_{e_0} f )= 0$. These sum up to
	\begin{align*}
		\text{Hess}_X(f)_{00} &= - \frac{1}{2V^2} 
		\frac{\partial V}{\partial{u_i}} \frac{\partial f}{\partial{u_i}}= - F'\frac{u_i}{V^2} \frac{\partial V}{\partial u_i} ,
		\\ \text{Hess}_X(f)_{i0}
		&= \text{Hess}(f)_{0i} = 
		- \frac{1}{2V^2} \epsilon_{ijk} \frac{\partial V}{\partial{u_j}} \frac{\partial f}{\partial{u_k}} = \epsilon_{ijk} F'
		\frac{u_j}{V^2}
		\frac{\partial V}{\partial u_k}, \\
		\text{Hess}_X(f)_{ij} &= 
		\frac{1}{V} \frac{\partial^2 f}{\partial u_i \partial u_j} + 
		\frac{\delta_{ij}}{2V^2} 
		\frac{\partial V}{\partial{u_k}}
		\frac{\partial f}{\partial{u_k}}
		- \frac{1}{2V^2} (
		\frac{\partial f}{\partial{u_i}}  
		\frac{\partial V}{\partial{u_j}}
		+
		\frac{\partial f}{\partial{u_j}}  
		\frac{\partial V}{\partial{u_i}}
		) \\&=  F'\delta_{ij}(
		\frac{2}{V} + 
		\frac{u_k}{V^2}
		\frac{\partial V}{\partial u_k}
		)
		- \frac{F'}{V^2}
		(
		u_i 
		\frac{\partial V}{\partial u_j}
		+
		u_j 
		\frac{\partial V}{\partial u_i}
		)  +4F'' \frac{u_i u_j}{V}  .
	\end{align*}
	In our case $V=1- \frac{2}{ \tilde{r} }= 1- \frac{2}{ \sqrt{u_1^2+u_2^2+u_3^2} }  $, so
	$
	\frac{\partial V}{\partial u_i} = 
	\frac{2u_i}{  \tilde{r}^3},
	$  
	and therefore, 
	\begin{align}
		\begin{split}
			\text{Hess}_X(f)_{00} &
			= - F'\frac{2}{\tilde{r}V^2},\\
			\text{Hess}_X(f)_{i0} &= \text{Hess}_X(f)_{0i} = 0,
			\\
			\text{Hess}_X(f)_{ij} & = F' \delta_{ij}(\frac{2}{V} + \frac{2}{V^2 \tilde{r}})- \frac{4F' u_iu_j}{V^2 \tilde{r}^3} + 4F'' \frac{u_iu_j}{V}.
		\end{split}
	\end{align}
	or equivalently,
	\begin{align*}
		\text{Hess}_X(f)  = 
		- \frac{2 F'}{\tilde{r} V^3} \vartheta_{{GH}}^2 +
		2\sum_{1 \leq i \leq 3} F'(1 + 
		\frac{1}{\tilde{r} V}) du_i^2
		+ \sum_{1 \leq i,j \leq 3} (4F''-\frac{4F' }{\tilde{r}^3V} )  u_iu_jdu_i du_j.
	\end{align*}
	This gives the Hessian on the Atiyah-Hitchin manifold up to an exponentially small error.
\end{proof}

\subsection{The Atiyah-Hitchin bundle and Fueter sections}\label{sec:AH-bundle}

Let $(M,g)$ be a closed oriented Riemannian 3-manifold, and let $P_{SO(3)} \to M$ denote its frame bundle. Using the $SO(3)$-action on the Atiyah-Hitchin manifold, we can form the associated bundle $\mathfrak{X}_{AH}:= P_{SO(3)} \times_{SO(3)} X_{AH} \to M$ with fibers isometric to the Atiyah-Hitchin manifold. This bundle inherits the Levi-Civita connection $\nabla$, and each oriented orthonormal basis $e_1, e_2, e_3$ in the tangent space of $M$ determines a quaternionic triple of complex structures $I_1, I_2, I_3$ on the Atiyah-Hitchin fibers. We shall consider Fueter sections $u$ of $\mathfrak{X}_{AH}\to M$, namely solutions to the equation
\[
\mathfrak{F} u= \sum_{i=1}^3 I_i \nabla_i u=0. 
\]
Here $\nabla_i$ is a shorthand for $\nabla_{e_i}$, and by $SO(3)$-equivariance one easily checks that the Fueter operator $\mathfrak{F}$ is independent of the choice of orthonormal basis.

We have the $SO(3)$-equivariant map $\pi: X_{AH}\to \R^3/\Z_2$, which in the asymptotic region displays the Gibbons-Hawking coordinates. Taking the associated bundle construction for the frame bundle $P_{SO(3)}$, we obtain a projection map 
\[
\pi: \mathfrak{X}_{AH} \to P_{SO(3)}\times_{SO(3)} \R^3/\Z_2\simeq TM/\Z_2\simeq  T^*M/\Z_2 ,
\]
where
$T^*M/\Z_2$ is the quotient of the cotangent bundle of $M$ under the $\Z_2$-involution on the fibers. 

The Euclidean distance to the origin  $\tilde{r}$ on $\R^3/\Z_2$ induces via this projection map a function $\tilde{r}$ on $\mathfrak{X}_{AH}$. We shall denote by $|u|$ the evaluation of this function $\tilde{r}\circ u$ along the Fueter section $u$. Via the projection, the Fueter section $u$ gives rise to a section $\pi(u)$ of $T^*M/\Z_2$, namely a $\Z_2$-valued 1-form on $M$. 

We define \[\norm{u}_{L^\infty}:= \norm{\pi(u)}_{L^\infty},\] 
and in the primary case of interest, this $L^\infty$-norm is very large.

\subsection{Energy formula}\label{sec:energy}

We now review an energy formula for Fueter sections, which will be used in controlling the Dirichlet energy by the $L^2$-norm of the Fueter section.

The $SO(3)$-action on $X_{AH}$ induces an action of $SO(3)$ on the 2-sphere of symplectic forms $a\omega_1 + b\omega_2 + c\omega_3$, where $(a, b, c) \in \R^3$ with $a^2 + b^2 + c^2 = 1$, and $\omega_1, \omega_2, \omega_3$ is the triple symplectic forms on $X_{AH}$. This action is the standard $SO(3)$-action on $S^2 \subset \R^3$. Such an $SO(3)$-action on a hyperkähler manifold is called a permuting action. The existence of this action implies that every symplectic form $\omega$ in the 2-sphere family is exact. Let $(\tau_1, \tau_2, \tau_3)$ be a basis for $\mathfrak{so}(3)$, then
\begin{align*}
	\omega_i = d \alpha_i \quad \text{where} \quad \alpha_i := \frac{1}{2} (\iota_{v_{k}}\omega_{j} -\iota_{v_{j}}\omega_{k}) \in \Omega^1(X_{AH}),
\end{align*}
for cyclic $(i,j,k)$, where $v_i$ is the vector field generating the infinitesimal action of $\tau_i$, dual to the 1-form $\sigma_i$.

We now present an asymptotic formula for the primitives of $\omega_i$ on the Gibbons-Hawking model. Let
\begin{align*}
\begin{cases}
	\alpha_1^{GH}&= -u_1\vartheta_{GH} + u_2(\frac{1}{2} - \frac{2}{\tilde{r}}) du_3-  u_3(\frac{1}{2} - \frac{2}{\tilde{r}}) du_2,\\
	\alpha_2^{GH}&= -u_2\vartheta_{GH} + u_3(\frac{1}{2} - \frac{2}{\tilde{r}}) du_1-  u_1(\frac{1}{2} - \frac{2}{\tilde{r}}) du_3,\\
 	\alpha_3^{GH}&= -u_3\vartheta_{GH} + u_1(\frac{1}{2} - \frac{2}{\tilde{r}}) du_2-  u_2(\frac{1}{2} - \frac{2}{\tilde{r}}) du_1.
\end{cases}
\end{align*}
A direct computation shows $d \alpha_i^{GH} = \omega_i^{GH}$.

\begin{lemma}\label{AH-GH}
	On the asymptotic region of the Atiyah-Hitchin manifold, the 1-forms $\alpha_i$ satisfy $|\alpha_i- \alpha_i^{GH}|= O( e^{-\nu \tilde{r}})$.
\end{lemma}

Moreover, note that the $SO(3)$ acts on the space spanned by $\alpha_1,\alpha_2,\alpha_3$, compatible with the $SO(3)$-action on the hyperk\"ahler triple $\omega_1,\omega_2,\omega_3$.

\begin{proof}

The infinitesimal vector fields of the $SO(3)$-action on the $X_{AH}$ in the asymptotic model is given by the following:
\begin{align}\label{GH-vector}
    v_i^{GH} = u_j \partial u_k - u_k \partial u_j + \frac{2u_i}{\tilde{r}} \partial_\theta,
\end{align}
where $\partial_\theta$ is the vector field dual to the 1-form $\vartheta_{GH}$.

Since on the asymptotic region $\omega_i = \omega_i^{GH} + \text{exponentially decaying error}$ and $v_i = v_i^{GH} + \text{exponentially decaying error}$, we have 
\begin{align*}
    \frac{1}{2} (\iota_{v_{k}}\omega_{j} -\iota_{v_{j}}\omega_{k}) = \frac{1}{2} (\iota_{v_{k}^{GH}}\omega_{j}^{GH} -\iota_{v_{j}^{GH}}\omega_{k}^{GH})  + \text{exp. error}.
\end{align*}
The triple of 1-forms $\frac{1}{2} (\iota_{v_{k}^{GH}}\omega_{j}^{GH} -\iota_{v_{j}^{GH}}\omega_{k}^{GH})$ is given by
\begin{align*}
     \frac{1}{2} (\iota_{v_{k}^{GH}}\omega_{j}^{GH} &- \iota_{v_{j}^{GH}}\omega_{k}^{GH}) \\ = &
    \frac{1}{2}  (
    \iota_{u_i \partial u_j - u_j \partial u_i + \frac{2u_k}{\tilde{r}} \partial_\theta} (\vartheta_{GH} \wedge du_j + V du_k \wedge du_i)  
  \\& -\iota_{u_k \partial u_i - u_i \partial u_k + \frac{2u_j}{\tilde{r}} \partial_\theta} (\vartheta_{GH} \wedge du_k + V du_i \wedge du_j 
 ) )
 \\=& \frac{1}{2}\left(- u_i \vartheta_{GH} + u_j V du_k + \frac{2u_k}{\tilde{r}} du_j
 - (V u_k du_j + u_i \vartheta_{GH} + \frac{2u_j}{\tilde{r}} du_k) \right)
 \\=& -u_i\vartheta_{GH} + u_j(\frac{1}{2} - \frac{2}{\tilde{r}}) du_k-  u_k(\frac{1}{2} - \frac{2}{\tilde{r}}) du_j.
\end{align*}
Therefore,
\begin{align*}
    \alpha_i  = -u_i\vartheta_{GH} + u_j(\frac{1}{2} - \frac{2}{\tilde{r}}) du_k-  u_k(\frac{1}{2} - \frac{2}{\tilde{r}}) du_j + \text{exp. error}.
\end{align*}

\end{proof}

\begin{corollary}\label{cor:alphai}
	In the asymptotic region of the Atiyah-Hitchin manifold, the 1-forms $\alpha_i$ satisfy 
	\[
	|\alpha_1|^2+ |\alpha_2|^2+ |\alpha_3|^2 \leq \frac{3}{2}\tilde{r}^2 (1+ O( \tilde{r}^{-1}) ).
	\]
	Moreover, $\alpha_i$ are almost orthogonal to the radial 1-form $\sum_1^3 u_i du_i$ up to exponentially small error.
\end{corollary}

\begin{proof}
    Following Lemma \ref{AH-GH}, we can replace $\alpha_i$ by $\alpha_i^{GH}$ and compute all norms in the Gibbons-Hawking model metric. We have the orthonormal coframe on $X$ given by $e^0 = V^{-1/2} \vartheta_{GH}, e^i = V^{1/2} du_i$. We write 
    \[ \alpha_i^{GH} = - V^{\frac{1}{2}}  u_i e^0 + V^{-\frac{1}{2}} u_{j}(\frac{1}{2} - \frac{2}{\tilde{r}}) e^k - V^{-\frac{1}{2}} u_k (\frac{1}{2} - \frac{2}{\tilde{r}})e^j,
    \]
    with cyclic $(i,j,k)$. Hence,
    \begin{align*}
        |\alpha_i^{GH}|^2 = V u_i^2  + V^{-1}  u_{j}^2 (\frac{1}{2} - \frac{2}{\tilde{r}})^2 + V^{-1} u_k ^2 (\frac{1}{2} - \frac{2}{\tilde{r}})^2,
    \end{align*}
    so using $V= 1- \frac{2}{\tilde{r}}$ and $\tilde{r}^2=u_1^2+u_2^2+u_3^2$ in the asymptotic region, we deduce
    \begin{align*}
        \sum_{i=1}^3 |\alpha_i^{GH}|^2 & = V \sum_{i=1}^3u_i^2  + 2V^{-1} (\frac{1}{2} - \frac{2}{\tilde{r}})^2 \sum_{i=1}^3u_{j}^2 
         =  \frac{3}{2} \tilde{r}^2 + O(\tilde{r}).
    \end{align*}
    Furthermore, in the Gibbons-Hawking model, the 1-forms $\vartheta_{GH}, u_2du_3-u_3du_2, u_3du_1-u_1du_3, u_1du_2-u_2du_1$ are all orthogonal to the radial 1-form $\sum_1^3 u_i du_i$, so the 1-forms $\alpha_i^{GH}$ are orthogonal to the radial 1-form. 
\end{proof}

Using any orthonormal frame $\{e_i\}_1^3$ on the 3-manifold $M$, there is an orthonormal coframe $\{e_i^*\}_1^3$, and we can invariantly define a 3-form $\sum \omega_i \wedge e_i^*$ on $\mathfrak{X}_{AH} $, where $\omega_i$ is a 2-form of horizontal-vertical type $(0,2)$, and the 3-form is independent of the choice of the frame. For Fueter sections, this 3-form captures the energy density $|\nabla u|^2 dvol$. Similarly by $SO(3)$-equivariance, the 2-form $\sum_1^3 u^*\alpha_i\wedge e_i^*$ is invariantly defined.

\begin{lemma}
	[Energy formula] For a Fueter section $u$,
	\[
	\frac{1}{2}|\nabla u|^2 dvol=  -\sum_1^3 u^*\omega_i \wedge e_i^* = - d(\sum_1^3 u^*\alpha_i \wedge e_i^*)  + \sum_{cyc} Rm_{ij}\# u^*\alpha_k.
	\]
\end{lemma}

\begin{proof}
As in \cite[Proof of Prop. 2.1]{MR3718486}, the first equality follows from a direct computation. Over the Atiyah-Hitchin space we have $\omega_i = d \alpha_i$. The exterior derivative on
the total space $\mathfrak{X}_{AH}$ decomposes as 
\[ d = d_f + d_H + F_H: \Omega^{0,1} \to \Omega^{0,2} \oplus \Omega^{1,1} \oplus \Omega^{2,0},\]
where $\Omega^{p,q}$ denotes the differential forms with $p$ covectors in the base directions and $q$ in the fiber directions.

We can, without loss of generality, compute using a geodesic frame around a given point on $M$, so $\nabla e_i=0$ and $d_H \alpha_i=0$. Thus
\begin{align*}
    \frac{1}{2}|\nabla u|^2 dvol=-\sum_1^3 u^* \omega_i \wedge e_i^*= -\sum_1^3 u^* d_f \alpha_i \wedge e_i^* = -\sum_1^3 d u^*  \alpha_i \wedge e_i^* + \sum_1^3  F_H u^*\alpha_i \wedge e_i^*.
\end{align*}
It remains to understand $ \sum_1^3  F_H u^*\alpha_i \wedge e_i^*$. The Riemannian curvature tensor is an $so(3)$-valued 2-form, where the $so(3)$ factor induces a vertical tangent vector field on the Atiyah-Hitchin fiber, so we obtain a 2-form valued in the vertical tangent bundle, which is precisely the curvature term $F_H$. The tensor $F_H$ acts algebraically on $\Omega^{0,1}$ by wedge product on the horizontal term and contraction on the vertical term. Now the term $u^* F_H \alpha_k \wedge e_k^*$ is a 3-form on $M$ obtained by the contraction of the curvature component $Rm_{ij} e_i^*\wedge e_j^*$  with the term $u^*\alpha_k\wedge e_k^*$, so we can write $\sum_1^3 u^* F_H \alpha_k \wedge e_k^*$ as a tensor contraction 
 $\sum_{cyc} Rm_{ij}\# u^*\alpha_k$.
\end{proof}

\begin{remark} \emph{
	This formula bears some similarity to the topological energy formula for J-holomorphic curves. The key difference is that $\int_M |\nabla u|^2$ is \emph{not} determined by the topological information of $u$ alone, so it does not imply an a priori topological energy bound for Fueter sections.
 } 
\end{remark}

\begin{corollary}\label{cor:topologicalenergy1} We have
	\[E:=\frac{1}{2}\int_M |\nabla u|^2  \leq C\int |u|^2+ C.\]
\end{corollary}

\begin{proof}
	Using the Stokes formula, and the fact that $|u^*\alpha_i|\leq C(|u|+1) |\nabla u|$, we obtain
	\[
	E\leq \int_M |\sum_{cyc} Rm_{ij}\# u^*\alpha_k|\leq C\int_M (1+ |u|) |\nabla u| \leq C(1+ \norm{u}_{L^2}) E^{1/2}.
	\]
	The claim follows by rearranging.
\end{proof}

This has a local version.

\begin{corollary}\label{cor:topologicalenergylocal2}
	On a coordinate ball $B(x,2r)$, we have
	\[
	\int_{B(x,r) } |\nabla u|^2 \leq  C r+ Cr^{-2} \int_{B(x,2r)} |u|^2 \leq Cr \max(1, \sup_{B(x,2r)} |u|^2) .
	\]
\end{corollary}

\begin{proof}
	We take a standard non-negative smooth cutoff function $0\leq \phi\leq 1$ supported in $B(x,2r)$, and equal to one on $B(x,r)$. 
	Using the energy formula,
	\[
	\frac{1}{2}\int |\nabla u|^2 \phi^2 = \int  \sum_1^3 u^* \alpha_i \wedge e_i^*\wedge  d(\phi^2) +\phi^2 \sum_{cyc} Rm_{ij}\# u^*\alpha_k \leq Cr^{-1} \int  (1+ |u|) |\nabla u| \phi.
	\]
	By Cauchy-Schwarz, the RHS is bounded by 
	\[
	Cr^{-1} (\int |\nabla u|^2 \phi^2)^{1/2} (Cr^3+ \int_{B(x,2r)} |u|^2 )^{1/2},
	\]
	hence, the result follows by rearranging.
\end{proof}

Moreover, the energy monotonicity formula holds as in Walpuski \cite[Prop. 2.1]{MR3718486}. (Here the non-compactness of the fibers does not cause any problems, because the bundle curvature is bounded.)

\begin{proposition}\label{energymonotonicity}
	For any $x\in M$ and any $0<s<r\leq 1$, we have
	\[
	\frac{e^{cr}}{r} \int_{B(x,r)} |\nabla u|^2 - \frac{e^{cs}}{s} \int_{B(x,s)} |\nabla u|^2 \geq \int_{ B(x,r)\setminus B(x,s) } d(x,\cdot)^{-1} |\nabla_r u|^2 - c(r^2-s^2).
	\]
\end{proposition}

\section{Examples of Fueter sections}\label{sec:examples}

We use dimensional reduction to produce examples of Fueter sections of the Atiyah-Hitchin monopole bundle $\mathfrak{X}_{AH}=P_{SO(3)}\times_{SO(3)} X_{AH}$ by reducing the problem to holomorphic sections. For comparison, we also consider the variant situation by replacing the Atiyah-Hitchin manifold with its double cover, and compute the Fueter sections of the bundle $\tilde{\mathfrak{X}}_{AH}=P_{SO(3)}\times_{SO(3)} \widetilde{\text{Mon}}^{\circ}_2(\R^3)\to M $.

\subsection{Holomorphic sections of monopole bundle}

Let $M = \Sigma \times \R$ or $\Sigma \times S^1$, equipped with the product metric. Let $u \in \Gamma(M, \mathfrak{X}_{AH})$ be a translation-invariant Fueter section of $\mathfrak{X}_{AH}$ (resp. the double cover $\tilde{\mathfrak{X}}_{AH}$), hence,
\begin{align*}
	\mathfrak{F}(u) = I (e_1) \nabla_{e_1}u + I(e_2) \nabla_{e_2}u,
\end{align*}
where $(e_1, e_2)$ is a frame on the Riemann surface $\Sigma$, and $e_3$ is the translation direction. Equivalently, the restriction $f=u|_\Sigma \in \Gamma (\Sigma, \mathfrak{X}_{| \Sigma})$ satisfies
\begin{align*}
	\overline{\partial}_{-I(e_3)} f = \nabla_{e_1} f - I(e_3) \nabla_{e_2} f = 0.
\end{align*}
We shall study the moduli space of solutions
\begin{align*}
	\mathcal{M} (\mathfrak{X}, \Sigma) = \{ f \in \Gamma(\Sigma, \mathfrak{X}|_{\Sigma}) \; | \;  \overline{\partial}_{-I(e_3)} f = 0 \}.
\end{align*} 
Note that the translation direction $e_3$ specifies a special direction in $\R^3$, and therefore, a preferred complex structure $I(e_3)$ on $X_{AH}= {\text{Mon}}^{\circ}_2(\R^3) $ and its double cover $\widetilde{\text{Mon}}^{\circ}_2(\R^3)$. Furthermore, it reduces the frame bundle structure group $SO(3)$ to $U(1)$, and the associated bundle $\mathfrak{X}_{AH}$ is formed from this $U(1)$-principal bundle via the $U(1)$-action on $X_{AH}$ (resp. $\widetilde{\text{Mon}}^{\circ}_2(\R^3)$). We shall use the affine variety description of $X_{AH}$ and $\widetilde{\text{Mon}}^{\circ}_2(\R^3)$, quoted in Section \ref{sect:affinevariety}.

\begin{example}[Double cover case]\label{eg:doublecovercase} \emph{  We first consider the double cover case $\mathfrak{X}_{AH}=P_{SO(3)}\times_{SO(3)}\widetilde{\text{Mon}}^{\circ}_2(\R^3) $. The $U(1)$-equivariant embedding
	\[ \widetilde{\text{Mon}}^{\circ}_2(\R^3) \simeq 
	\{ (a_0,a_1, b_0) \in \C^3 \; |a_0^2 + b_0 a_1^2 = 1  \} \subset \C^3
	\]
	induces the bundle embedding over the Riemann surface $\Sigma$,
	\[
	\mathfrak{X}|_\Sigma \subset \widetilde{E}:=P_{U(1)} \times_{U(1)} \C^3.
	\]
	The rank 3 vector bundle $\widetilde{E}$ decomposes to the direct sum of three lines bundles, according to the $U(1)$ weights $(0,-1,2)$ on $\C^3$:
	\begin{align*}
		\widetilde{E} = \underline{\C} \oplus L \oplus L^{\otimes (-2)} \to \Sigma,
	\end{align*}
	where $L = T^* \Sigma$. }

	\emph{
	Let $i$ denote the complex structure on $\widetilde{E}$ induced from the standard complex structure on $\C^3$, which agrees with $I$ on the monopole bundle, and let $j$ be the complex structure of the Riemann surface $\Sigma$. 
	Given any solution $f$ to the equation $\overline{\partial}_{-I(e_3)} f = 0 $, we obtain 
	a $(j,-i)$-holomorphic section of the rank 3 bundle $\widetilde{E}$, or equivalently, $(j,i)$-}anti-holomorphic \emph{sections $(f_1,f_2,f_3)$ for the three component line bundles $\underline{\C} \oplus L \oplus L^{\otimes (-2)} \to \Sigma$.}

	\emph{ 
	The trick is to consider the} complex conjugate \emph{of these sections, which are $(j,i)$-holomorphic sections that take value in the} dual \emph{ bundles. To see this, suppose $\tilde{L}$ is a line bundle associated to a $U(1)$-action defined by $\lambda \cdot z := \lambda^k z$ for some $k \in \Z$. We can think about a section of this bundle $v \in \Gamma (\tilde{L})$ as a map 
	\[
	v: P_{U(1)} \to \C,
	\]
	such that $v(\lambda \cdot x) = \lambda^k v(x)$. 
	The complex conjugate $\bar{v}$, defined by $\bar{v}(x) := \overline{v(x)}$, is equivariant with respect to the inverse action
	\begin{align*}
		\bar{v}(\lambda \cdot x) = \bar{\lambda^k} \overline{v(x)} = 
		\lambda^{-k} \bar{v}(x).
	\end{align*}
	Therefore, $\bar{v}$ is a section of the inverse line bundle associated with the action $\lambda \cdot z := \lambda^{-k} z$.}

	\emph{
	The complex conjugate sections $\bar{f}_1, \bar{f}_2, \bar{f}_3$ are holomorphic sections of $\underline{\C}, L^{-1}$, and $L^{\otimes 2}$ respectively. Therefore:
	\begin{itemize}
		\item The section $\bar{f}_1$ is constant.
		\item The  $\bar{f}_2$ is a holomorphic vector field on $\Sigma$. When the genus  $g=0$, $g = 1$, and $g \geq 2$ respectively, this space is a complex $3$-dimensional, complex $1$-dimensional, and $\{0\}$, respectively. 
		\item The $\bar{f}_3$ is a quadratic differential on $\Sigma$. Using the Riemann-Roch theorem, we see   $h^{0}(\Sigma,L^{\otimes 2}) = 0$ for $g=0$, and $h^{0}(\Sigma,L^{\otimes 2}) = 1$ for $g=1$, and $h^{0}(\Sigma,L^{\otimes 2}) = 3(g-1)$ for $g\geq 2$. 
	\end{itemize} }
	\emph{
	Now, we restrict to the $(j,-i)$-holomorphic sections $f = (f_1, f_2, f_3)$ of $\widetilde{E}$ such that $f_1^2 - f_3 f_2^2 = 1$, namely the sections that fall inside ${\mathfrak{X}}$.}

	\emph{
	\textbf{Case 1 (genus $g = 0$).} 
	We have $f_3 = 0$, and therefore $f_1 = \pm 1$. The section $\bar{f}_2$ is an arbitrary holomorphic section of $L$. Therefore, $\mathcal{M}(\widetilde{E}, \Sigma)$ contains two connected components, each being a copy of $ H^0 (\mathbb{CP}^1, L^{-1})$. In particular, the space of translation-invariant Fueter sections on $S^2\times S^1$ has dimension $\dim_{\R} (\mathcal{M}(\mathfrak{X}, \mathbb{CP}^1)) = 6$. }

	\emph{
	\textbf{Case 2 (genus $g = 1$).} All component sections $f_1, f_2, f_3$ are constants, so
	\begin{align*}
		\mathcal{M}(\mathfrak{X}, \Sigma) = \text{Mon}_2^{\circ}(\R^3).
	\end{align*}}

	\emph{
	\textbf{Case 3 (genus $g \geq 2$).} We have $f_1 = \text{const}=\pm 1$, and $f_2 = 0$. Therefore $\bar{f}_3$ can be any holomorphic section of $L^{\otimes 2}$, so
	\begin{align*}
		\mathcal{M}(\mathfrak{X}, \Sigma_g) = \C^{3g-2}\sqcup \C^{3g-2}.
	\end{align*}}

\begin{example}[Atiyah-Hitchin bundle case]\label{eg:AHcase}\emph{
	The case $\mathfrak{X}=P_{SO(3)}\times_{SO(3)} X_{AH} $ follows the same strategy. The $U(1)$-invariant embedding}
	\[
	X_{AH} \simeq \{ (z_1, z_2, z_3, z_4)\in \C^4: z_1z_3=z_2^2, z_1+ z_3z_4=1  \}
	\]
	\emph{induces the bundle embedding over $\Sigma$,}
	\[
	\mathfrak{X}\subset  E = \underline{\C} \oplus L \oplus L^2 \oplus  L^{\otimes (-2)} \to \Sigma
	\]
	\emph{where $L=T^*\Sigma$.}
 
	\emph{Given any solution $f=(f_1,f_2,f_3,f_4)$ to the equation $\overline{\partial}_{-I(e_3)} f = 0$, 
	the complex conjugate trick as above gives holomorphic sections $(\bar{f}_1,\bar{f}_2,\bar{f}_3,\bar{f}_4)$ of}
	\[
	\underline{\C} \oplus L^{-1} \oplus L^{-2} \oplus  L^{\otimes 2} \to \Sigma.
	\]
	\emph{The condition that $f$ lands inside $\mathfrak{X}$ imposes that }
	\[
	f_1f_3=f_2^2, \quad  f_1+ f_3f_4=1 
	\]

 \emph{
	\textbf{Case 1 (genus $g = 0$).} We have $f_4=0$, $f_1=1$, $f_3=f_2^2$, and $f_2$ is an arbitrary section of $L^{-1}$. The moduli space of solutions is one copy of $H^0(\mathbb{CP}^1, L^{-1})$, which has real dimension $6$. }

	\emph{
	\textbf{Case 2 (genus $g = 1$).} All component sections $f_1, f_2, f_3, f_4$ are constants, so
	\begin{align*}
		\mathcal{M}(\mathfrak{X}, \Sigma) = X_{AH}.
	\end{align*}}

     \emph{
	\textbf{Case 3 (genus $g \geq 2$).} We have $f_2=f_3=0$, and $f_1=1$, while $f_4$ is any section of $L^{\otimes 2}$. The moduli space of solutions is $\C^{3g-2}$. }

	\begin{remark}\emph{
		In these examples, each Fueter section to $\mathfrak{X}_{AH}$ lifts to two Fueter sections to the double cover $\tilde{\mathfrak{X}}_{AH}$. On general compact oriented 3-manifolds, not necessarily of product type, we do not know if this lifting property always holds.}
	\end{remark}
	
\end{example}

	\begin{remark}\emph{
		Since the linearized Fueter operator has index zero, one expects a 0-dimensional space of Fueter sections on closed 3-manifolds for a generic Riemannian metric. This means the product examples above are non-generic.}
	\end{remark}
\end{example}

\begin{example}[Local product examples]\label{eg:Localproduct}\emph{
	 We now replace $\Sigma$ by the unit disc $D(1)$ in $\C$, with the Euclidean metric. Then Fueter sections into the Atiyah-Hitchin bundle are simply Fueter maps into $X_{AH}$, and the dimensional reduction gives an anti-holomorphic map $f: D(1)\to X_{AH}$. Equivalently, $\bar{f}$ is any holomorphic map $D(1)\to X_{AH}\subset \C^4$. Such maps exist in great abundance. In particular, we can take $f=(f_1,f_2,f_3,f_4)$,
	\[
	f_1=1,\quad  f_3=f_2^2, \quad f_4=0,
	\]
	and $\bar{f}_2$ is any holomorphic function on the unit disc. Alternatively, we can take
	\[
	f_1=1,\quad f_2=f_3=0,
	\]
	and $\bar{f}_4$ is any holomorphic function on the unit disc.
	}
\end{example}

\subsection{Gibbons-Hawking perspective}\label{sec:examples-GH}

We use the projection map 
\[
\pi: X_{AH}\to \R^3/\Z_2
\]
and its bundle version
\[
\pi: \mathfrak{X}_{AH}\to T^*M/\Z_2
\]
to better understand these Fueter sections on product manifolds. Using the results quoted in Section \ref{sect:asymptoticregionrationalmap}, 
we now revisit revisit Example \ref{eg:doublecovercase}, which corresponds to the lift of Example \ref{eg:AHcase} to the double cover. 
	\begin{itemize}
		\item In the genus zero case, $f_1=\pm1, f_3=0$, and $f_2$ is arbitrary. We are interested in the case where $f_2$ is large. Here the denominator of the rational map has double roots, and the projection is 
        $\pm (0,0, \log |f_2|) \in T^*M/\Z_2= (T^*\Sigma \oplus \R)/\Z_2$ up to $O(1)$ error.

		\item  In the genus $\geq 2$ case, $f_1=\pm 1$, $f_2=0$, and $f_3$ is arbitrary. We are interested in the case where $f_3$ is large. The corresponding rational map is
		\[
		S(\zeta)= \frac{1}{\zeta^2+ f_3}= \frac{1/ 2\sqrt{- f_3}    }{ \zeta- \sqrt{- f_3}  } -  \frac{1/2\sqrt{- f_3}   }{ \zeta+ \sqrt{- f_3}  },
		\] 
		which describes the superposition the charge one monopoles centered around the points  $(\sqrt{-f_3},\frac{1}{2} \log |2\sqrt{-f_3}|  )$ and $( -\sqrt{-f_3} , \frac{1}{2} \log |2\sqrt{-f_3}|)$. Applying a further translation to move the center of mass to the origin, we see that  the projection to $T^*M/\Z_2$ is $\pm (\sqrt{f_3}, 0)$ modulo $O(1)$ error.
	\end{itemize}
The same formulas work also in the local analogue over $D(1)\times \R$, where the $f_i$ are replaced by anti-holomorphic functions instead of complex conjugate of holomorphic sections.

\subsection{Analytic subtleties}\label{sect:analyticsubtlety}\label{sec:GH-analytic}

The examples above reveal that the analytic properties of the Fueter sections are rather delicate.

\begin{itemize}
	\item  In the genus zero case of Example \ref{eg:doublecovercase}, we can consider the sequence of Fueter sections $u=(f_1=\pm 1, \lambda f_2, f_3=0)$ and send the parameter $\lambda\to \infty$. Then \[\norm{u}_{L^\infty}=\norm{\pi(u)}_{L^\infty} = \log \lambda+O(1),\] and $\pi(u)/ \norm{u}_{L^\infty}$ converges to $(0,0,1)$ at every point, except for the zeros of $f_2$, where $\pi(u)/ \norm{u}_{L^\infty}$ converges to zero.

	This example shows that for a sequence of Fueter sections with $L^\infty$-norm diverging to infinity, the normalized sequence of $\Z_2$-valued 1-forms 	$\pi(u)/\norm{u}_{L^\infty}$ can \emph{fail to converge in the $C^0$-topology}.

	\item We now perform a heuristic calculation for the gradient. We consider a section \(u=(f_1=\pm 1, \lambda f_2, f_3=0)\), as above, where $f_2$ is a fixed anti-holomorphic section, which for simplicity we assume to have simple zeros. We know  $\pi(u)= (  0,0, \log \lambda+ \log |f_2|  )$ modulo bounded error. Ignoring the error term, we formally compute 
    \[
	\nabla \pi(u) \approx  (0, 0, \nabla \log |f_2|),\quad |\nabla \pi(u)|\approx |f_2|^{-1} |\nabla f_2|. 
	\]
	This computation is expected to be valid when $|\lambda f_2|\gtrsim 1$. Thus for $p> 2$,
	\[
	\int_M |\nabla (\frac{\pi(u)} { \norm{u}_{L^\infty}  })|^p \gtrsim |\log \lambda|^{-p} \int_{ |f_2|\gtrsim \lambda^{-1} }  |f_2|^{-p} |\nabla f_2|^p.
	\]
	The local contribution around a simple zero of $f_2$ is of the order
	\[
 |\log \lambda|^{-p} 	\int_{D(1)\cap \{  |z|\gtrsim \lambda^{-1}  \}} |z|^{-p} \sqrt{-1}dz\wedge d\bar{z} \gtrsim   |\log \lambda|^{-p} \lambda^{p-2},
	\]
	which diverges to infinity as $\lambda\to \infty$.

	This heuristic calculation shows that we cannot expect to control the gradient of $\pi(u)/\norm{u}_{L^\infty}$ uniformly in any $L^p$ norm with $p>2$. The best hope is to control $\pi(u)/\norm{u}_{L^\infty}$ in the $W^{1,2}$-topology, which is the content of the main Theorem \ref{mainthm}.

	\item We can consider the local analogue of this example on  $D(1) \times \R \subset  \C \times \R$. In the local setting, $f_2$ can be any anti-holomorphic function $D(1)\to \C$. In particular, we can take
	\[
	\bar{f}_2= \prod_1^N (z- z_i),
	\]
	with any prescribed zeros $z_1,\ldots z_N$.

	This example shows that in the local setting, the `singular locus' $z=z_i$ can be arbitrarily dense. In particular, one cannot hope that after deleting some codimension two locus on the 3-manifold, the $\Z_2$-valued 1-forms 	$\pi(u)/\norm{u}_{L^\infty}$ converges in $C^\infty_{loc}$ (or even $C^0_{loc}$) topology to some limiting $\Z_2$-valued harmonic 1-form. We do not know if the same pathological phenomenon also happens on compact 3-manifolds.

	This situation should be compared with the result of Walpuski \cite{MR3718486} that states that for compact hyperk\"ahler target manifolds, under uniform energy bounds, then after deleting some closed subset with finite codimension two Hausdorff measure, the Fueter sections subsequentially converge to some limiting Fueter section in the $C^\infty_{loc}$-topology.

\end{itemize}

\section{A priori estimates on Fueter sections}\label{sec:4}

\subsection{$C^0$-estimate}\label{sect:C0}

Let $u \in \Gamma(M, \mathfrak{X}_{AH})$ be a Fueter section. The linearization of the Fueter operator at the Fueter section is given by \[\mathcal{L}_u(g) = \sum_{i=1}^3 I (e_i) \bar{\nabla}_{e_i} g,\] where $\bar{\nabla}$ is the induced connection on the pullback of the vertical tangent bundle. 

We shall derive a second-order equation, which can be viewed as a bundle twisted analogue of a harmonic map.

\begin{lemma}
	The Fueter section $u$ satisfies the second-order equation
	\begin{equation}\label{eqn:Fuetersecondorder}
		\Lap u= \sum_{cyc} I(e_i) \text{Rm}(e_{i+1},e_{i+2})\cdot  u(x),
	\end{equation}
	where the $e_1, e_2, e_3$ is any oriented orthonormal basis on the tangent space of $M$, the $\Lap$ is the rough Laplacian $\sum \bar{\nabla}_i \nabla_i$,  the $I(e_i)$ is the complex structure on the Atiyah-Hitchin fiber corresponding to $e_i$, and the Riemannian curvature component $Rm(e_i, e_j)\in so(3)$ acts by a vector field on Atiyah-Hitchin fibers.  
\end{lemma}

\begin{proof}
	Let $(e_1, e_2, e_3)$ be a local orthonormal frame such that $\nabla e_i=0$ at a given point $x$ on the 3-manifold. We differentiate the Fueter equation on $u $, 
	\begin{equation*}
		\begin{split}
			0 = \sum \mathcal{L}_u (\mathfrak{F}(u))
			&= \sum_{i,j=1}^3   I(e_i) \bar{\nabla}_{e_i} (I(e_j) \nabla_{e_j} u) = 
			\sum_{i,j=1}^3  I(e_i) I(e_j) \bar{\nabla}_{e_i}  \nabla_{e_j} u \\&= 
			- \sum_{{i=1}}^3 \bar{\nabla}_{e_i}  \nabla_{e_i} u 
			+ \sum_{cyc} I(e_i) 
			(\bar{\nabla}_{e_{i+1}}  \nabla_{e_{i+2}} - \bar{\nabla}_{e_{i+2}}  \nabla_{e_{i+1}})u
			\\
			& = -\Lap u + \sum_{cyc} I(e_i) \text{Rm}(e_{i+1},e_{i+2}) u(x),
		\end{split}  
	\end{equation*}
	where the third equality uses that $\bar{\nabla} I(e_i) = 0$, the second line uses the quaternionic algebra, and the third line uses the curvature operator $so(3)$ action on the associated bundle.
\end{proof}

Let $f$ be a radially symmetric function on $X_{AH}$, which induces a function on the bundle $\mathfrak{X}_{AH}$. Therefore, along a Fueter section, we get a function $f(u)$ on $M$. For $|u|$ larger than a fixed large value, we can write $f(u)=F(|u|^2)$.

\begin{lemma}[$C^0$ Laplacian computation]\label{lem:C0Lap}
	At any point $x \in M$ where $|u|(x)$ is larger than some fixed constant, we have  
	\begin{align*}
		\Delta (f(u)) &= 2F'  |\nabla \pi( u)|^2 + F'' |\nabla |u|^2|^2 
		+ 
		2F' Ric(\pi(u),\pi(u))
		+ Err,
	\end{align*}
	where
	\[
	|Err| \leq C (|F'| |u|^{-1} |\nabla u|^2+ |F''|  |u|^{-1} |\nabla |u|^2|^2 + |F'| |u|).
	\]
\end{lemma}

\begin{proof}
	We compute in the geodesic coordinate on $M$. The gradient is
	\[
	\nabla_i f(u)= Df(u)\cdot \nabla_i u,
	\]
	where $Df$ is valued in the cotangent space of the Atiyah-Hitchin fiber, while $\nabla_i u$ is valued in the tangent space of the Atiyah-Hitchin fiber, and $\cdot$ denotes the natural pairing. Differentiating again, we get
	\[
	\Lap f(u)= \bar{\nabla}_i ( Df(u)\cdot \nabla_i u  )= \text{Hess}_{X_{AH}} f(\nabla_i u, \nabla_i u) + Df (\Lap u),
	\]
	where $\text{Hess}_{X_{AH}} f$ is the Hessian of $f$ evaluated on the Atiyah-Hitchin fibers, which takes value in $Sym^2$ of the vertical cotangent bundle, so can be naturally paired with two vertical tangent vectors $\nabla_i u$. This should not be confused with Hessians computed on the domain 3-manifold.
 
	From (\ref{eqn:Fuetersecondorder}), we get
	\begin{align}
		\Delta (f(u)) = 
		\sum_{i=1}^3 \text{Hess}_{X_{AH}} f(\nabla_i u, \nabla_i u) + \sum_{cyc} Df(I(e_i) 
		(\text{Rm}(e_{i+1},e_{i+2})u)).
	\end{align}

	The gradient term uses the Hessian on the Atiyah-Hitchin manifold. From Lemma \ref{lem:HessianAH}, for $|u|=\tilde{r}\circ u$ larger than a fixed constant, the leading order asymptote is
	\[
	\text{Hess}_{X_{AH}} f= 2 F' \sum_1^3 du_i^2 (1+O(|u|^{-1})) + 4F'' (\sum u_idu_i)^2 (1+O(|u|^{-1})).
	\]
	Thus 
	\begin{equation}
		\sum_{i=1}^3 \text{Hess}_{X_{AH}} f(\nabla_i u, \nabla_i u)= 2F' (1+ O(|u|^{-1}) |\nabla \pi(u)|^2 + F'' |\nabla |u|^2|^2 (1+ O(|u|^{-1})).
	\end{equation}

	We now compute the leading order Riemannian curvature term in the Gibbons-Hawking picture, which up to exponentially small error agrees with the Atiyah-Hitchin manifold computation, outside of a fixed compact region in the Atiyah-Hitchin manifold. Recall that $v=\pi(u)$ denotes the (fiberwise) projection from Atiyah-Hitchin manifold to $\R^3/\Z_2\simeq T_x M/\Z_2 \simeq T_x^* M/\Z_2$. 
 
    We claim that for $|u|$ larger than a fixed constant, 
	\begin{equation}
		\sum_{cyc} Df(I(e_i) 
		(\text{Rm}(e_{i+1},e_{i+2})u))= 2F' Ric(v, v) + O(|F'||u| ) .
	\end{equation}
	Here $Ric(v, v)$ is the evaluation of the Ricci curvature of the domain $(M,g)$ on two copies of the tangent vector $v$. We notice that even though $v$ is only defined up to $\Z_2$, the term $Ric(v,v)$ has no ambiguity. The point is that $Ric(v, v)$ is of quadratic order in $u$, so it will be one order larger than the error $O(|F'||u|)$. The rest of this proof is to justify this claim.

	First, we recall that $Rm(e_i, e_j) \cdot u$ is the vector field on $X_{AH}$ generated by the $so(3)$ action, evaluated at the point $u(x)$. In terms of the Gibbons-Hawking picture, for large $|u|$, the horizontal components of this vector field in the $\R^3$-directions are linear in $v$, but the circle component of this vector field stays bounded. Now $Df$ is $2F' \sum u_i du_i$, and the circle component contribution to $\sum_{cyc} Df(I(e_i) 
	(\text{Rm}(e_{i+1},e_{i+2})u))$ is of order $O(|F'| |u|)$.

	We write out the Riemannian curvature of the domain $M$,
	\begin{align*}
		\text{Rm}(e_1, e_2) &= 
		R_{1212} \; e_1 \wedge e_2 + 
		R_{1223} \; e_2 \wedge e_3 +
		R_{1231} \; e_3 \wedge e_1, \\
		\text{Rm}(e_2, e_3) &= 
		R_{2312} \; e_1 \wedge e_2 + 
		R_{2323} \; e_2 \wedge e_3 +
		R_{2331} \; e_3 \wedge e_1, \\
		\text{Rm}(e_3, e_1) &= 
		R_{3112} \; e_1 \wedge e_2 + 
		R_{3123} \; e_2 \wedge e_3 +
		R_{3131} \; e_3 \wedge e_1. \\
	\end{align*}
	After ignoring the circle component, we can view $Rm(e_i, e_j) \cdot u$ as a horizontal vector field, and the $so(3)$ action reduces to the standard action on $\R^3$. Now $v$ can be regarded as a cotangent vector $\sum u_i du_i$ on the base $\R^3$, which can be paired with the vectors $I(e_i) Rm(e_{i+1}, e_{i+2})\cdot u$. We compute
	\begin{align*}
		\langle v,& 
		I(e_3)\text{Rm}(e_1, e_2) \cdot u \rangle  \\&= 
		R_{1212} \; \langle u,  I(e_3) e_1 \wedge e_2  \cdot u \rangle + 
		R_{1223} \;  \langle u, I(e_3) e_2 \wedge e_3 \cdot u \rangle +
		R_{1231} \;  \langle u, I(e_3) e_3 \wedge e_1 \cdot u \rangle
		\\&= R_{1212} (u_1^2 + u_2^2) + 
		R_{1223} \;  (-u_3 u_1) +
		R_{1231} \; (-u_2 u_3)
		\\
		\langle v, &I(e_1) \text{Rm}(e_2, e_3) \cdot u \rangle \\&= 
		R_{2312} \; \langle u,I(e_1) e_1 \wedge e_2 \cdot u \rangle + 
		R_{2323} \; \langle u,I(e_1) e_2 \wedge e_3 \cdot u \rangle +
		R_{2331} \; \langle u,I(e_1) e_3 \wedge e_1 \cdot u \rangle\\
		&= 
		R_{2312} \; (-u_1 u_3) + 
		R_{2323} \; (u_2^2 + u_3^2) +
		R_{2331} \; (-u_1 u_2), \\
		\langle v,& I(e_2) 
		\text{Rm}(e_3, e_1) \cdot u \rangle \\&= 
		R_{3112} \; \langle u, I(e_2) e_1 \wedge e_2\cdot u \rangle + 
		R_{3123} \; \langle u, I(e_2)  e_2 \wedge e_3\cdot u \rangle +
		R_{3131} \; \langle u, I(e_2)  e_3 \wedge e_1 \cdot u \rangle \\& 
		= 
		R_{3112} \; (-u_2 u_3) + 
		R_{3123} \; (-u_1 u_2) +
		R_{3131} \; (u_1^2+ u_3^2),
	\end{align*}
	and using the standard properties of the Riemannian curvature,
	\begin{align*}
		&\langle v, \sum_{i=1}^3 I(e_i) 
		\text{Rm}(e_{i+1},e_{i+2}) \cdot u \rangle \\& =  R_{1212} (u_1^2 + u_2^2) + 
		R_{1223} \;  (-u_3 u_1) +
		R_{1231} \; (-u_2 u_3) + 
		R_{2312} \; (-u_1 u_3) + 
		R_{2323} \; (u_2^2 + u_3^2) \\&+
		R_{2331} \; (-u_1 u_2) +
		R_{3112} \; (-u_2 u_3) + 
		R_{3123} \; (-u_1 u_2) +
		R_{3131} \; (u_1^2+ u_3^2) \\&= 
		u_1^2 (R_{1212} + R_{3131}) +
		u_2^2 (R_{1212}+R_{2323}) +
		u_3^2 (R_{2323}+R_{3131}) +
		u_1 u_2 (-R_{3123}-R_{2331}) \\& +
		u_2 u_3 (-R_{1231}-R_{3112}) +
		u_3 u_1 (-R_{1223}-R_{2312})
		\\& + u_1 u_2 (-R_{3123}-R_{2331}) +
		u_2 u_3 (-R_{1231}-R_{3112}) +
		u_3 u_1 (-R_{1223}-R_{2312}) \\&=
		u_1^2 (Ric_{11}) +
		u_2^2 (Ric_{22}) +
		u_3^2 (Ric_{33}) +
		2u_1 u_2 (Ric_{12}) +
		2u_2 u_3 (Ric_{23}) +
		2u_3 u_1 (Ric_{31}) \\&= \sum_{i,j} u_i u_j Ric_{ij},
	\end{align*}
	as required.
\end{proof}

We can pick the radially symmetric function, by choosing $F$ a smooth convex function on $\R_{\geq 0}$, such that $F=\text{const}\approx r_0$ in a large enough but fixed interval $[0, r_0]$, and $F(x)=x$ on $[r_0+1,+\infty)$. In particular, $f=|u|^2$ when $|u|\geq r_0+1$, and $f=\text{const}$ for $|u|\leq r_0$, and $f$ is an increasing function of $|u|$. By taking $r_0$ large enough, we can absorb the error term in Lemma \ref{lem:C0Lap} by the main term, using the small $O(|u|^{-1})$ factor. Thus globally on $M$,
\begin{equation}\label{eqn:Lapf(u)}
	\Lap f(u) \geq 2F' Ric(\pi(u), \pi(u)) -C|F'| |u| .
\end{equation}
This has some basic consequences:

\begin{corollary} [$L^\infty$ bound]
	 On coordinate balls on $M$, we have
	\[
	\sup_{B(x,r) } |u|^2 \leq \max(C, C r^{-3}  \int_{B(x,2r)} |u|^2).
	\]
	In particular, $\norm{u}_{L^\infty} \leq C+C\norm{u}_{L^2(M)}$.
\end{corollary}

\begin{proof}
	The above inequality (\ref{eqn:Lapf(u)}) implies
	\[
	\Lap f(u) \geq -C |u|^2 \geq -C f(u).
	\]
	By the mean value inequality,
	\[
	\sup_{B(x,r) } f(u) \leq C r^{-3}  \int_{B(x,2r)} f(u).
	\]
	Since $f(u)=|u|^2$ for $|u|$ bigger than a fixed constant, and otherwise $f$ is bounded, we deduce the claim.
\end{proof}

\begin{corollary}
	Suppose the Ricci tensor of $(M,g)$ is \emph{strictly positive} on $M$. Then there is a uniform bound $\max_M |u|\leq C$ depending only on $(M,g)$.

\end{corollary}

\begin{proof}
	Suppose that $\max_M |u|$ is very large, so that at the point $x\in M$ where $|u|(x)$ takes the maximum, $u(x)$ is in the asymptotic region of the Atiyah-Hitchin fiber. By the maximum principle $\Lap |u|^2\leq 0$ at this point. The gradient term is non-negative, and the Ricci term dominates the error term, so $\Lap |u|^2>0$ by Lemma \ref{lem:C0Lap}, which is a contradiction.
\end{proof}

\begin{remark}\emph{
	Here the strict positivity is important. In the example of $M=S^1\times \mathbb{CP}^1$ with the product metric, the Ricci tensor is non-negative, but there is a sequence of Fueter sections with $\max_M |u|$ diverging to infinity.}
\end{remark}

\begin{corollary}[$L^2$-gradient] \label{cor:L2gradient} 
	\[
	\int_M |\nabla \pi(u)|^2+ Ric(\pi(u),\pi(u)) \leq C\int_M |\nabla u|^2 (1+ |u|)^{-1} + C\int_M |u| +C.
	\]
\end{corollary}

\begin{proof}
	We start from $\int_M \Lap f(u)=0$. Converting the integrand to the RHS of Lemma \ref{lem:C0Lap}, for the above choice of radial function $f$, we find that
	\[
	\int_M 2F'  |\nabla \pi(u)|^2 + F'' |\nabla |u|^2|^2 
	+ 
	2F' Ric(\pi(u),\pi(u)) \leq C \int_M |\nabla u|^2 (1+ |u|)^{-1} + C\int_M |u| +C.
	\]
	Here $F'(|u|^2)=1$, and $F''=0$ for $|u|$ large enough. On the other hand, the contribution of $|u|$ smaller than a fixed bound, can be absorbed by the RHS. This shows the claim.
\end{proof}

The intuition here is that the RHS integrand is smaller than the LHS by a factor $O(|u|^{-1})$, so the RHS is morally much smaller than the $L^2$-energy when the average value of $|u|$ is big. Thus, the above inequality hints at the approximate cancellation between the $L^2$-energy of $\pi(u)$ and the integral of the Ricci term.

\subsection{$L^2$-type gradient estimate}

We shall derive a power law decay bound on the $L^2$-gradient of the sublevel sets
\[
E(t)= \frac{1}{2}\int_{ |u|\leq t } |\nabla u|^2, \quad r_0\leq  t\leq \norm{u}_{L^\infty},
\]
where $r_0$ is some sufficiently large but fixed constant, so that we are in the asymptotic region of Atiyah-Hitchin manifold if $|u|\geq r_0$. 
We recall that
\[
E=\frac{1}{2}\int_M |\nabla u|^2
\]
is the total $L^2$-energy.

\begin{lemma}
	\label{prop:powerlawL2gradient}
	Suppose $0< p< \sqrt{2/3}$ is a fixed exponent. Then  there is some large constant $C(p)$ independent of $u$, such that we have the power law estimate
	\[
	E(t) \leq \max(C, C(p) \norm{u}_{L^\infty}^{2-p}, C(p) \norm{u}_{L^\infty}^{2-p}   t^p),\quad \text{ for } 0\leq t\leq \norm{u}_{L^\infty}.
	\]   
\end{lemma}

\begin{proof}
	As a preliminary remark, the statement is only non-trivial when $\norm{u}_{L^\infty}$ is much larger than $r_0$. We shall derive a differential inequality and use an ODE comparison argument.

	Recall that the Atiyah-Hitchin manifold has a triple of symplectic forms $\omega_i=d\alpha_i$. 
	By the energy formula, when $t$ is not a critical value of $|u|$, so the level set is smooth, we have
	\[
	E(t)=- \int_{|u|\leq t} \sum_1^3 u^*\omega_i \wedge e_i^* = -\int_{|u|=t} \sum_1^3 u^*\alpha_i\wedge e_i^*   + \int_{|u|\leq t}\sum_{cyc} Rm_{ij}\# u^*\alpha_k .
	\]
	 We observe that
	\begin{equation}
		\int_{|u|\leq t} \sum_{cyc} Rm_{ij}\# u^*\alpha_k \leq C \int_{|u|\leq t} |u| |\nabla u| \leq Ct E(t)^{1/2},
	\end{equation}
	which is a secondary effect if $E(t)\gg t^2$.

	The main term is $\int_{|u|=t} \sum_1^3 u^*\alpha_i\wedge e_i^*$. Recall from Lemma \ref{cor:alphai} that in the asymptotic region of Atiyah-Hitchin manifold (namely if $|u|$ is larger than some fixed constant),
	\[
	|\alpha_1|^2+ |\alpha_2|^2+ |\alpha_3|^2 \leq \frac{3}{2} |u|^2 (1+ C|u|^{-1}),
	\]
	and the $\alpha_i$ are almost orthogonal to the radial 1-form $\sum u_i du_i$ up to exponentially small error, so $u^*\alpha_i$ is only sensitive to the part of $\nabla u$ annihilated by the radial 1-form. Thus by Cauchy-Schwarz, 
	\[
	\begin{split}
		 \int_{|u|=t} \sum_1^3 u^*\alpha_i\wedge e_i^* \leq & \int_{|u|=t} (\sum_1^3 |\alpha_i|^2)^{1/2} (|\nabla u |^2- |\nabla |u||^2)^{1/2} 
		\\
		\leq &  \sqrt{3/2} \int_{|u|=t} |u| (  |\nabla u |^2- |\nabla |u||^2)^{1/2} (1+ C|u|^{-1}) 
		\\
		=    &   \sqrt{3/2}t (1+ Ct^{-1})\int_{|u|=t}  ( |\nabla u |^2- |\nabla |u||^2)^{1/2}   . 
	\end{split}
	\]
	On the other hand, by the coarea formula
	\[
	E(t)=\frac{1}{2} \int_0^t \int_{|u|=s} \frac{|\nabla u|^2}{ |\nabla |u|| },
	\]
	so for almost every $t$, we have 
	\[
	E'(t)=\frac{1}{2} \int_{|u|=t} \frac{|\nabla u|^2}{ |\nabla |u|| }.
	\]
	By the elementary inequality $(a^2-b^2)^{1/2}\leq \frac{a^2}{2b}$ for $0\leq b\leq a$, we have
	\[
	\int_{|u|=t} ( |\nabla u |^2- |\nabla |u||^2)^{1/2} \leq  \frac{1}{2} \int_{|u|=t}  \frac{|\nabla u|^2}{ |\nabla |u|| }= E'(t), 
	\]
	hence
	\begin{equation}
		\int_{|u|=t} \sum_1^3 u^*\alpha_i\wedge e_i^* \leq \sqrt{3/2} t (1+Ct^{-1})E'(t).
	\end{equation}

	Feeding this back into the  energy formula, we obtain
	\[
	\begin{split}
		E(t)
		\leq 
		Ct E(t)^{1/2} + \sqrt{3/2} t (1+Ct^{-1})E'(t),
	\end{split}
	\]
	namely the increasing function $E(t)$ satisfies the differential inequality
	\begin{equation}
		E'(t)\geq \sqrt{2/3} t^{-1} (1+Ct^{-1})^{-1} (E(t)- Ct E(t)^{1/2} ).
	\end{equation}

	The rest is an ODE comparison argument. We assume $\norm{u}_{L^\infty}$ is sufficiently large depending on $p$, and let $C(p)$ denote some large constant depending on the choice of $p$. Then
	$\tilde{E}=  C(p) \norm{u}_{L^\infty}^{2-p} t^p $  for $C(p)'\leq t\leq \norm{u}_{L^\infty}^p$ satisfies
	\begin{equation}
		\tilde{E}'(t)\leq \sqrt{2/3} t^{-1} (1+Ct^{-1})^{-1} (\tilde{E}(t)- Ct \tilde{E}(t)^{1/2} ).
	\end{equation}
	This uses that $p< \sqrt{2/3}$, and
	\[
	LHS \approx C(p) p \norm{u}_{L^\infty}^{2-p} t^{p-1},\quad RHS\approx \sqrt{2/3} C(p)  \norm{u}_{L^\infty}^{2-p} t^{p-1},
	\]
	where the relative errors can be neglected because  $t\gg 1$ and $\tilde{E}(t)\gg t^2$.

	Moreover $E= E(t= \norm{u}_{L^\infty})\leq C\norm{u}_{L^\infty}^2$ as a consequence of the  energy formula. So by comparing the differential inequalities on $E(t)$ and $\tilde{E}(t)$,  we deduce that
	\[
	E(t) \leq \tilde{E}(t), \quad C(p)'\leq t\leq \norm{u}_{L^\infty}.
	\]
	A small caveat is that the monotone function $E(t)$ is a priori allowed to jump down as $t$ decreases, but this only helps with the desired inequality.

	Finally, the case $t\leq C(p)'$ can be included in the statement up to enlarging the constant $C(p)$, because $E(t)$ is increasing in $t$ and we have a bound at $t=C(p)'$.
\end{proof}

\begin{remark}\emph{
	We know by the energy formula that $E\leq C\norm{u}_{L^\infty}^2\leq C\norm{u}_{L^2}^2$, so unless we are in the rare situation where $E\ll \norm{u}_{L^\infty}^2$ (so that $u$ is almost covariant constant in some integral sense), then $E$ would be comparable to $\norm{u}_{L^\infty}^2$.  }
\end{remark}

\begin{remark}\emph{
	The exponent $p$ will appear in many later estimates, and larger value of $p$ means better estimates when $\norm{u}_{L^\infty}$ is very large. 
	We expect that the exponent threshold $\sqrt{2/3}$ is not sharp.}
\end{remark}

This gradient decay estimate has a localized version on coordinate balls $B(x,r)$. Let $\phi$ be a standard cutoff function supported in $B(x,r)$, with $0\leq \phi\leq 1$ and $|\nabla \phi|\leq Cr^{-1}$, and let
\[
E_\phi(t)=\frac{1}{2} \int_{|u|\leq t} |\nabla u|^2 \phi^2.
\]

\begin{lemma}
	Suppose $0<p<\sqrt{2/3}$ is a fixed exponent. Then there is some large constant $C(p)$ independent of $u$ and $r$, such that we have the power law estimate
	\[
	E_\phi(t) \leq \max(C r, C(p)r\norm{u}_{L^\infty(B(x,r))}^{2-p}, C(p) r\norm{u}_{L^\infty(B(x,r))}^{2-p} t^p ),\; \text{for } \; 0\leq t\leq \norm{u}_{L^\infty(B(x,r))}.
	\]

\end{lemma}

\begin{proof}
	(Sketch) As in Lemma \ref{prop:powerlawL2gradient}, we can, without loss of generality, work with large $t\leq \norm{u}_{L^\infty(B(x,r))}$.  We replace the energy formula by
	\[
	E_\phi(t)=  -\int_{|u|=t} \phi^2 \sum_1^3 u^*\alpha_i\wedge e_i^*+\int_{|u|\leq t}  \sum_1^3 u^*\alpha_i \wedge e_i^* \wedge d(\phi^2 ) +\phi^2 \sum_{cyc} Rm_{ij}\# u^*\alpha_k .
	\]
	The remainder term is bounded by Cauchy-Schwarz, 
    \begin{align*}
     | \int_{|u|\leq t} \sum_1^3 u^*\alpha_i \wedge e_i^* \wedge d(\phi^2 ) | &\leq C\int_{ |u|\leq t   } |u||\nabla u| \phi |\nabla \phi| \\& \leq C t r^{-1} \int_{ |u|\leq t   } |\nabla u| \phi  \leq C t r^{1/2} E_\phi(t)^{1/2} ,
    \end{align*}
	and similarly with the curvature term.
	Following the proof of Lemma
	\ref{prop:powerlawL2gradient}, we obtain
	\[
	E_\phi'(t)\geq \sqrt{2/3} t^{-1} (1+Ct^{-1})^{-1} (E_\phi(t)- Ct r^{1/2} E_\phi(t)^{1/2} ).
	\]
	Moreover the energy formula gives
	\[
	E_\phi( t= \norm{u}_{L^\infty(B(x,r))  }) \leq C \int |u|^2 |\nabla \phi|^2 \leq C r \norm{u}_{ L^\infty(B(x,r))  }^2
	\]
	A very similar ODE comparison argument gives the result.
\end{proof}

\begin{corollary}\label{cor:L2gradientpowerlaw2}
	Suppose $0< p< \sqrt{2/3}$ is a fixed exponent. 
	Then
	\[
	\int_M |\nabla u|^2(1+ |u|)^{- 1 } \leq C(\norm{u}_{L^\infty}^{2-p}+1).
	\]    
\end{corollary}

\begin{proof}
	We estimate
	\[
	\begin{split}
		\int_M |\nabla u|^2 (1+|u|)^{-1} = \int_0^\infty \frac{1}{(s+1)^2} \int_{|u|\leq s} |\nabla u|^2
	\end{split}
	\]
	Using the power law estimate in Lemma \ref{prop:powerlawL2gradient}, and since $p-2<-1$, we see
	\[
	\int_1^{\norm{u}_{L^\infty}} \frac{1}{s^2} \int_{|u|\leq s} |\nabla u|^2 \leq C+C\norm{u}_{L^\infty}^{2-p} \int_1^{ \norm{u}_{L^\infty}  } t^{p-2} dt  \leq C+ C\norm{u}_{L^\infty}^{2-p}.
	\]
	On the other hand,
	\[
	\int_{ \norm{ u }_{ L^\infty }}^{\infty} (s+1)^{-2} E = E (\norm{u}_{L^\infty} +1)^{-1} \leq C(1+ \norm{u}_{L^\infty}).
	\]
	Summing over the two contributions gives the result.
\end{proof}

The following Proposition means that in some $L^2$ gradient sense, the Fueter section with very large $L^\infty$-energy almost gives rise to an $\Z_2$-harmonic 1-form $v=\pi(u)$ and that the $L^2$-gradient component along the circle direction $\vartheta$ of the Atiyah-Hitchin manifold is small compared to the total $L^2$ energy.

\begin{proposition}\label{prop:almostharmonic1form}(Almost $\Z_2$-harmonic 1-form)
	Given any fixed exponent $0<p< \sqrt{2/3}$,
	we have
	\[
	\int_M |d_g^* v|^2+ |dv|^2 + \int_{|u|\geq 1} |\vartheta(\nabla u)|^2 \leq   C(p)\norm{u}_{L^\infty}^{2-p} +C.
	\]
	(The point is that the RHS is \emph{subquadratic} in $\norm{u}_{L^\infty}$).
\end{proposition}

\begin{proof}
	The key is a 
	Weitzenb\"ock formula argument. We notice $v= \pi(u)$ is a section of the bundle $TM/\Z_2$. For $v\neq 0$, we can compute in an orthonormal frame $e_i$  with $\nabla e_i=0$ at the given point on $M$, by locally lifting $v$ to $TM$,
	\[
	\nabla^* \nabla v= (d d^* + d^* d) v- Ric_{ij} v_i e_j^*.
	\]
	Let $\chi: \R_{\geq 0}\to \R_{\geq 0} $  be a smooth function which vanishes near the origin, and $\chi(x)=1$ for $x\geq 1$. We now take the metric pairing of the above Weitzenb\"ock formula with $\chi(|u|^2) v$, to obtain
	\[
	\int_M \chi(|u|^2)(v, \nabla^* \nabla v)=\int_M  (\chi(|u|^2)v, (d d^* + d^* d) v)- \chi(|u|^2) Ric_{ij} v_i v_j.
	\]
	We note two conceptual issues. Although $v$ is defined up to $\Z_2$ ambiguity, the integrands above are well-defined pointwise. Moreover, the cutoff effect of $\chi$ removes the possible singularity for $v=0$, so we are justified to integrate by parts below. We compute 
	\[
	\int_M \chi(|u|^2)(v, \nabla^* \nabla v)= \int_M \chi |\nabla v|^2 + \int_M ( v\nabla \chi, \nabla v),
	\]
	and 
	\[
	\int_M  (\chi(|u|^2)v, (d d^* + d^* d) v)= \int_M (d^*(\chi v), d^* v)+ (d(\chi v), dv)
	\]
	Comparing the above, and estimating the cutoff error, we obtain
	\begin{equation}
		\int_M |d^* v|^2+ |dv|^2 \leq \int_M |\nabla v|^2 + Ric(v,v) + C \int_{|u|\leq 1}  |\nabla v|^2.
	\end{equation}
	Since $v=\pi(u)$, and $\pi$ has Lipschitz bound, we know
	\[
	\int_{|u|\leq 1}  |\nabla v|^2 \leq C\int_{|u|\leq 1} |\nabla u|^2.
	\]
	Using Corollary \ref{cor:L2gradient} to estimate the first integral on the RHS, we get
	\[
	\int_M |d^* v|^2+ |dv|^2  \leq  C\int_M |\nabla u|^2 (1+ |u|)^{-1} + C\int_M |u| +C.
	\]
	Applying Corollary \ref{cor:L2gradientpowerlaw2} to estimate the first term on the RHS, we deduce 
	\[
	\int_M |d^* v|^2+ |dv|^2  \leq  C(p)\norm{u}_{L^\infty}^{2-p} +C.
	\]

	To estimate the circle component $\vartheta(\nabla u)$, we notice that for $|u|$ large enough so that $u$ lands in the asymptotic region of the Atiyah-Hitchin manifold, the Fueter equation $\sum_1^3 I_i \nabla_i u=0$ implies by a pointwise computation in the geodesic coordinates on $M$ that
	\begin{equation}\label{eqn:circlecomponentFueter}
		|\vartheta(\nabla u)|^2= |d v|^2 +O( |u|^{-1} |\nabla u|^2).
	\end{equation}
	The integral control on $|dv|^2$ and $|u|^{-1}|\nabla u|^2$ shows
	\[
	\int_{  M\cap \{ |u|\geq 1\} }  |\vartheta(\nabla u)|^2 \leq C(p)\norm{u}_{L^\infty}^{2-p} +C,
	\]
	as required.
\end{proof}

\subsection{$L^\infty$-type gradient estimate}

\begin{lemma}\label{lem:Lapgrad1}
	The gradient of any Fueter section $u$ satisfies the differential inequality
	\begin{align*}
		\frac{1}{2} \Delta |\nabla u|^2 &\geq
		|\bar{\nabla} \nabla u|^2- C \left( |Rm||\nabla u|^2+ |\nabla Rm| |u| |\nabla u| + |\text{Riem}_u| |\nabla u|^4 \right).
	\end{align*}
	Here $Rm$ stands for the Riemannian curvature of the domain 3-manifold, and $\text{Riem}_u$ is the Riemannian curvature of the Atiyah-Hitchin fibers evaluated along the section $u$.
\end{lemma}

\begin{proof}
	Let $(e_1, e_2, e_3)$ be a local orthonormal frame such that $\nabla e_i=0$ at a given point on the 3-manifold, so in particular, $[e_i, e_j]=0$ at the given point.

	We compute at the given point
	\begin{align*}
		\left (\bar{\nabla}_{e_i} \nabla_{e_j}  - \bar{\nabla}_{e_j} \nabla_{e_i}  \right) u (x) & = \text{Rm}(e_i,e_j)\cdot u(x),
		\\
		\left (\bar{\nabla}_{e_i} \bar{\nabla}_{e_j}  - \bar{\nabla}_{e_j} \bar{\nabla}_{e_i}  \right) \nabla u (x) &= \text{Riem}_{u(x)}(\nabla_{e_i} u, \nabla_{e_j} u) \nabla u(x) + \text{Rm}(e_i,e_j) \nabla u(x).
	\end{align*}
	where in the first equation, the Riemannian curvature operator on the domain 3-manifold $\text{Rm}(e_i,e_j)\in \mathfrak{so}(3)$ acts on the section in the associated hyperk\"ahler bundle. In the second equation, in addition to this effect, we have the Riemannian curvature of the target hyperk\"ahler manifold acting on the fiber tangent bundle part of $\nabla u$, and the Riemannian curvature of the 3-manifold acting on the 1-form part of $\nabla u$.

	Similarly,
	\[
	\bar{\nabla}_{e_j} \nabla_{[e_i,e_j]} u=\bar{\nabla}_{[e_i,e_j]} \nabla_{e_j} u+  \text{Rm}(e_j, [e_i,e_j])\cdot u=0,
	\]
	since $[e_i,e_j]=0$ at the point.

	We write the rough Laplacian as  $\Delta = \sum_{i=1}^3 \bar{\nabla}_{e_i}  \bar{\nabla}_{e_i}$.
	Then,
	\begin{align*}
		\Delta |\nabla u|^2 = \sum_{i=1}^3 \Delta |\nabla_{e_i}u|^2 = 2\sum_{i=1}^3 \langle \Delta \nabla_{e_i} u, \nabla_{e_i} u \rangle + 2\sum_{i,j=1}^3 |\bar{\nabla}_{e_i}  \nabla_{e_j} u|^2.
	\end{align*}
	Using the above curvature formulas, we compute the terms in the right-hand side, with the Einstein summation notation:
	\begin{align*}
		\Delta \nabla_{e_i} u = &  \bar{\nabla}_{e_j}  \bar{\nabla}_{e_j} \nabla_{e_i} u = 
		\bar{\nabla}_{e_j}  \bar{\nabla}_{e_i} \nabla_{e_j} u -  \bar{\nabla}_{e_j}\nabla_{[e_i, e_j]} u -  \bar{\nabla}_{e_j} (\text{Rm}(e_i, e_j) \cdot u)
		\\
		=& 
		\bar{\nabla}_{e_i} \bar{\nabla}_{e_j}\nabla_{e_j} u  - \text{Riem}_{u }(\nabla_{e_i} u, \nabla_{e_j} u) \nabla_{e_j} u -Rm(e_i,e_j)\nabla u(x)
		\\
		& -  \bar{\nabla}_{e_j}(\text{Rm}(e_i, e_j)) u- Rm(e_i,e_j)\nabla_{e_j} u
	\end{align*}
	The highest order derivative term can be computed using the second-order consequence (\ref{eqn:Fuetersecondorder}) of the Fueter equation,
	\[
	\bar{\nabla}_{e_i} 
	\Delta u = \bar{\nabla}_{e_i} (\sum_{j=1}^3 I(e_j) Rm(e_{j+1}, e_{j+2})\cdot u)  ,
	\]
	which can be bounded by $C|Rm||\nabla u|+ C|\nabla Rm| |u|$. By inspection on the lower order terms, we can estimate
	\begin{equation*}
		|\Delta \nabla_{e_i} u|\leq C \left( |Rm||\nabla u|+ |\nabla Rm| |u|+ |\text{Riem}_u| |\nabla u|^3 \right). 
	\end{equation*}
	Thus we have
	\begin{align*}
		\langle \Delta \nabla_{e_i} u  , \nabla_{e_i} u \rangle &\leq |\Delta \nabla_{e_i} u| |\nabla_{e_i} u| \\& \leq C(|Rm||\nabla u|^2+ |\nabla Rm| |u| |\nabla u| + |\text{Riem}_u| |\nabla u|^4). 
	\end{align*}
	Therefore,
	\begin{align*}
		\frac{1}{2} \Delta |\nabla u|^2 &=  |\bar{\nabla} \nabla u|^2 + \langle \Delta \nabla_{e_i} u  , \nabla_{e_i} u \rangle 
		\\
		&
		\geq  |\bar{\nabla} \nabla u|^2- C(|Rm||\nabla u|^2+ |\nabla Rm| |u| |\nabla u| + |\text{Riem}_u| |\nabla u|^4),
	\end{align*}
	as required.
\end{proof}

We shall use this to establish scale invariant gradient bound on $u$ \emph{provided $|u|$ is very large and does not oscillate too much on a given coordinate ball}.

\begin{lemma}[Local $L^\infty$-gradient estimate]\label{lem:localLinftygrad}
	 Suppse $u$ is a Fueter section on the coordinate ball $B(x,3r)$, such that 
	\[
	(\sup_{ B(x,3r) } |u|)^{2/3}   \gg \inf_{ B(x,3r) } |u|
	\gg r_0.
	\] Then
	\[
	\sup_{B(x,r)} |\nabla u|\leq  Cr^{-1} \norm{u}_{L^\infty(B(x,3r))}. 
	\]

\end{lemma}

\begin{proof}
	We will mimic the standard trick to derive interior gradient estimate for harmonic maps. We consider a point $y$ achieving the maximum
	\[
	M=\max_{y\in B(x,2r)}  (2r- |y-x|)^2 |\nabla u|^2,
	\]
	and denote $s_0= r- |x-y|/2$, so on $B(y,s_0)$
    \[|\nabla u|^2\leq 4s_0^{-2} M. \]

    By Lemma \ref{lem:Lapgrad1}, and the fact that the Riemannian curvature of the Atiyah-Hitchin manifold has cubic decay asymptotically, 
	on balls $B(y, s)$ with $0<s\leq s_0$ we have
	\[
	\frac{1}{2} \Delta |\nabla u|^2 \geq
	- C ( |\nabla u|^2+ |u|^2 +  |\nabla u|^4 |u|^{-3} ) \geq -C( s_0^{-2} M + |u|^2 + s_0^{-4} M^2 (\inf_{B(x,3r)} |u|)^{-3}    ) .
	\]
	Recall that for non-negative valued harmonic functions $f$ on the 3-dimensional ball $B(y,s)$, a version of the mean value inequality states that
	\[
	f(y) \leq C s^{-3} \int_{B(y,s)} f + C s^2 \norm{\Lap f}_{ L^\infty(B(y,s)) }.
	\]
	Thus
	\begin{align*}
	    s_0^{-2}M = |\nabla u|(y)^2 &\leq C s^{-3} \int_{ B(y,s)  } |\nabla u|^2dvol \\&+ Cs^2( s_0^{-2} M + \norm{u}_{L^\infty(B(x,3r))}^2 + s_0^{-4} M^2  (\inf_{B(x,3r)} |u|)^{-3}  ).
	\end{align*}
	By the energy monotonicity formula, and the topological energy bound,
	\[
	s^{-1} \int_{ B(y,s)  } |\nabla u|^2dvol \leq  C r^{-1} \int_{B(y,r)} |\nabla u|^2 \leq C  \norm{u}_{L^\infty(B(x,3r))}^2. 
	\]
	Combining the above,
	\[
	s^2s_0^{-2}M\leq C  \norm{u}_{L^\infty(B(x,3r))}^2 +Cs^4( s_0^{-2} M+  \norm{u}_{L^\infty(B(x,3r))}^2 +
	(\inf_{B(x,3r)} |u|)^{-3}
	s_0^{-4} M^2   ).
	\]
	This can be viewed as a quadratic inequality on $s^2s_0^{-2}M$. For $0\leq s\ll s_0$, the quadratic equation has a smaller root comparable to $\norm{u}_{L^\infty(B(x,3r))}^2  $ and a large root comparable to $(\inf_{B(x,3r)}|u|)^3 \gg \norm{u}_{L^\infty(B(x,3r))}^2$. 
 
    As we increase $s$ from zero to the order of magnitude of $s_0$, by continuity $s^2 s_0^{-2}M$ must be bounded by the smaller root, whence
	\[
	M\leq C\norm{u}_{L^\infty}^2,
	\]
	which implies the gradient estimate.
\end{proof}

\begin{remark}\emph{
	Intuitively, in the above proof the assumption  on $|u|$ is an analogue of the small energy assumption in the standard $\epsilon$-regularity result for harmonic maps.}
\end{remark}

\begin{remark}\emph{
	If we drop the assumption that $|u|$ stays large in the coordinate ball, then we expect the gradient bound is false by the discussion of Section \ref{sect:analyticsubtlety}. }
\end{remark}

\subsection{Local regularity}

The intuition of this Section is that if $|u|^2$ does not oscillate much within a coordinate ball, then on a smaller ball the Fueter section enjoys the same quantitative regularity that one would expect for an ordinary harmonic function.

We know the interior regularity in the presence of an $L^\infty$ bound on the coordinate ball, essentially by the previous works of Walpuski \cite{MR3718486} and Esfahani \cite{esfahani2023towards}.

\begin{lemma}\label{lem:localregularityAHcore}
	Suppose $u$ is a Fueter section on the coordinate ball $B(x,2r)$, such that $\norm{u}_{L^\infty} \leq K$ is bounded independent of $u$. Then for any $k\geq 1$, we have the uniform bound on $B(x,r)$ with constant depending on the bound on $\norm{u}_{L^\infty} $,
	\[
	|\nabla^k u|  \leq C(k,K) r^{-k}.
	\]

\end{lemma}

The following Proposition generalizes this phenomenon, by allowing $\norm{u}_{L^\infty}$ to be very large.

\begin{proposition}[Local regularity]\label{prop:localregularity}
	Suppose $u$ is a Fueter section on the coordinate ball $B(x, 3r)$, such that \[\frac{1}{2}|u|(x)\leq |u|\leq 2 |u|(x).\]  Then on the smaller ball $B(x, r)$, we have
	\[
	|\nabla^k u| \leq C(k) r^{-k} \max(1, \norm{u}_{L^\infty(B(x,3r)}).
	\]

\end{proposition}

\begin{proof}
	By rescaling the coordinate ball, which only decreases the effect of the domain Riemannian curvature, we can reduce to the case where $r=1$. By Lemma \ref{lem:localregularityAHcore}, we only need to consider the case that $\norm{u}_{L^\infty} \gg 1$, so 
	according to Lemma \ref{lem:localLinftygrad}, on a smaller ball, we have
	\begin{equation}
		|\nabla u|\leq C \norm{u}_{L^\infty}.
	\end{equation}
	By repeatedly differentiating the elliptic PDE (\ref{eqn:Fuetersecondorder}), and the fact that the Riemannian curvature effect on the Atiyah-Hitchin manifold satisfies $|\nabla^j Rm|=O(|u|^{-3-j})$, we can use the standard Bernstein method to deduce linear bounds on $|\nabla^k u|$ to higher orders $k\geq 1$ on slightly shrunken balls.

	Alternatively, one can argue slightly more geometrically as follows. We can scale down the Atiyah-Hitchin manifold by a large factor $\norm{u}_{L^\infty}$, and form a new bundle over the 3-manifold $M$ with fibers being the rescaled Atiyah-Hitchin manifold. Then $u$ can be viewed as a Fueter section $\tilde{u}$ inside the new bundle, and by the assumption $\frac{1}{2}|u|(x)\leq |u|\leq 2|u|(x)$ and the gradient bound on $u$, we know $\tilde{u}$ has magnitude uniformly comparable to one, and $C^1$-norm of order $O(1)$.
	
	On the other hand, the Riemannian curvature of the rescaled Atiyah-Hitchin manifold is of order $O( \norm{u}_{L^\infty}^{-1}) $ in the region $|\tilde{u}|\sim 1$. This means $\tilde{u}$ is a Fueter section into a target space of very small Riemannian curvature, which allows us to bootstrap to higher order estimates on slightly shrunken balls
	\[
	|\nabla^k \tilde{u}|\leq C(k),
	\]
	which is the same as $|\nabla^k u| \leq C(k) \norm{u}_{L^\infty}$.
\end{proof}

\subsection{Energy decay estimates on macroscopic scales}\label{sect:energydecay}

Recall that the energy monotonicity formula (\cf Proposition \ref{energymonotonicity}) gives an energy decay estimate on small geodesic balls,
\[
\int_{B(x,r)} |\nabla u|^2 \leq C r \int_M |\nabla u|^2 + Cr^2 \leq C r \norm{u}_{L^\infty}^2 +Cr^2. 
\]
The goal of this Section is to improve the exponent in the $r$-dependence.

\begin{proposition}[Energy decay on macroscopic balls] \label{prop:energydecay}
	Let $u$ be a Fueter section in $\Gamma(M, \mathfrak{X}_{AH})$ with $\norm{u}_{L^\infty}\gg 1$, and let $0<p<\sqrt{2/3}$. Then there is some small  constant $\gamma>0$ independent of $u$ and the geodesic balls $B(x,r)$, such that
	\[
	\int_{B(x,r)} |\nabla u|^2 \leq  C r^{1+\gamma} \norm{u}_{L^\infty}^2 + C(p) \norm{u}_{L^\infty}^{2-p} .
	\]
	
\end{proposition}

\begin{remark}\emph{
	If we can remove the $C(p) \norm{u}_{L^\infty}^{2-p}$ 
	error term, then the Campanato-Morrey embedding theorem would imply some uniform H\"older modulus of continuity on $ \pi(u)/\norm{u}_{L^\infty}$, which is false in general. 
	}
\end{remark}

The energy decay estimate is based on the following differential inequality. Let $v=\pi(u)$ denote the section of $T^*M/\Z_2$ obtained by projection.

\begin{lemma}\label{lem:energydecay}
	There are small constants $\gamma'>0$ and $r_0'>0$, such that whenever $0<r\leq r_0'$, and $0<p<\sqrt{2/3} $, then
	\[
	\int_{B(x,r)} |\nabla v|^2 \leq  (1-\gamma'/3) \int_{\partial B(x,r)} |\nabla v|^2 + C(p) \norm{u}_{L^\infty}^{2-p} + Cr^2 \norm{u}_{L^\infty}^2.
	\]
\end{lemma}

\begin{proof}
	Using a local orthonormal frame $\{e_i\}_1^3$ on $B(x,r)$ such that $\nabla e_i=0$ at $x$, we can trivialize the cotangent bundle of $M$, so $v$ can be regarded as a 2-valued function into $\R^3$, namely  $v: B(x,r)\to \mathcal{A}_2(\R^3)$ in the theory of Q-valued functions \cite{MR2663735}. Since $u$ is smooth, and the projection map is Lipschitz, it is clear that $v$ is in the Sobolev space $W^{1,2}$, and by construction, the two sheets of $v$ have an average value zero at every point on $B(x,r)$. The heuristic intuition below is that $v$ is close to a Dirichlet minimizer, so morally inherits the regularity of Dirichlet minimizers.

	De Lellis-Spadaro \cite[Prop. 3.10]{MR2663735} exhibited a competitor $h\in W^{1,2}(B(x,r), \mathcal{A}_2(\R^3))$ with the same boundary data as $v$ over $\partial B(x,r)$, and Dirichlet energy 
	\begin{equation}
		Dir(h, B(x,r))\leq (\frac{1}{3-2}- \gamma')r Dir (v, \partial B(x,r))= (1-\gamma') r Dir(v, \partial B(x,r)),
	\end{equation}
	for some small universal number $\gamma'>0$. 
 
    By symmetry, we can assume the two sheets of $h$ also have average zero at every point. 
	As a small caveat, here the Dirichlet energy is defined in the trivialisation on the chart, which deviates from the $L^2$-energy for sections of $T^*M/\Z_2$ by the Christoffel symbol term, and uses the Euclidean volume measure instead of the Riemannian volume measure.

	We compute in the local trivialisation with local coordinates $y_1,y_2,y_3$, by writing $h=[\![ H   ]\!]+ [\![ -H   ]\!]$. We have the identity
	\[
	\begin{split}
		Dir(h, B(x,r))= & \int_{B(x,r)} \sum_{i,j} |\partial_i H_j|^2 \\=& -\int_{B(x,r)} \sum_{cyc} dH_i\wedge dH_j \wedge dy_k + \int_{B(x,r)} |dH|^2 + |d^* H|^2
		\\
		= & -\int_{\partial B(x,r)} \sum_{cyc} H_i\wedge dH_j \wedge dy_k + \int_{B(x,r)} |dH|^2 + |d^* H|^2.
	\end{split}
	\]
	Here the first line uses a pointwise identity at every almost differentiable point of $h$, and the second line uses the Stokes formula. A similar identity holds for $v$. 
 
    Now $v$ and $h$ have the same boundary data, so the Stokes boundary terms coincide, whence 
	\[
	Dir(h, B(x,r))\geq  -\int_{\partial B(x,r)} \sum_{cyc} H_i\wedge dH_j \wedge dy_k = Dir(v, B(x,r))- \int_{B(x,r)} (|dv|^2+ |d^*v|^2).
	\]
	As a caveat, the Euclidean and the Riemannian versions of $\int_{B(x,r)} |dv|^2+ |d^*v|^2$ are slightly different, but their difference is absorbed by
	\begin{align*}
	    Cr Dir(v, B(x,r)) + Cr \norm{u}_{L^\infty} Dir(v, B(x,r))^{1/2} &+ Cr^2 \norm{u}_{L^\infty}^2 \\& \leq \frac{\gamma'}{2} Dir(v, B(x,r)) + Cr^2 \norm{u}_{L^\infty}^2,
	\end{align*}
	as long as $r$ is smaller than some fixed small constant. Thus
	\[
	Dir(h, B(x,r))\leq  (1-\gamma'/2) Dir(v, B(x,r))- \int_{B(x,r)} (|dv|_g^2+ |d_g^*v|_g^2) + C r^2 \norm{u}_{L^\infty}^2.
	\]
	Using Proposition \ref{prop:almostharmonic1form} to estimate $\int_{B(x,r)} (|dv|_g^2+ |d_g^*v|_g^2)$, 
	we deduce
	\[
	\begin{split}
		Dir(v, B(x,r)) & \leq Dir(h, B(x,r))+ C(p) \norm{u}_{L^\infty}^{2-p}
		\\
		& \leq   (1-\gamma'/2) r Dir(v, \partial B(x,r)) + C(p) \norm{u}_{L^\infty}^{2-p} + Cr^2 \norm{u}_{L^\infty}^2.
	\end{split}
	\]
	The deviation between the Euclidean Dirichlet energy 
	and the Riemannian $L^2$-energy comes from the Christoffel symbol term and the Riemannian volume element. For small enough $r$, we can absorb the $O(r)$-relative errors coming from these sources, to deduce
	\[
	\int_{B(x,r)} |\nabla v|^2 \leq  (1-\gamma'/3) \int_{\partial B(x,r)} |\nabla v|^2 + C(p) \norm{u}_{L^\infty}^{2-p} + Cr^2 \norm{u}_{L^\infty}^2
	\]
	as required.
\end{proof}

We now prove Proposition \ref{prop:energydecay}.

\begin{proof}
	We consider the one-variable function 
	\[
	h(r)= \int_{B(x,r)} |\nabla v|^2,\quad 0< r\leq r_0'.
	\]
	By the coarea formula, and Lemma \ref{lem:energydecay},
	\[
	h(r) \leq (1-\gamma'/3) r h'(r) +  C(p) \norm{u}_{L^\infty}^{2-p} + Cr^2 \norm{u}_{L^\infty}^2,
	\]
	or equivalently for $\gamma=  \frac{1}{1-\gamma'/3}-1$, and some modified constants,
	\[
	\frac{d}{dr} ( h(r) r^{-1-\gamma} ) \geq  - C(p) \norm{u}_{L^\infty}^{2-p} r^{-2-\gamma} - Cr^{-\gamma} \norm{u}_{L^\infty}^2.
	\]
	We also know that at the endpoint $r=r_0'$,
	\[
	h(r_0')\leq \int_M |\nabla v|^2 \leq C\norm{u}_{L^\infty}^2. 
	\]
	By integrating the differential inequality,
	\[
	h(r) r^{-1-\gamma} \leq C\norm{u}_{L^\infty}^2 r_0'^{-\gamma-1} + C(p) \norm{u}_{L^\infty}^{2-p} r^{-1-\gamma} +C r^{1-\gamma} \norm{u}_{L^\infty}^2. 
	\]
	After absorbing some terms, we obtain
	\begin{equation}\label{eqn:energydecayv}
		h(r)= \int_{B(x,r)} |\nabla v|^2 \leq C\norm{u}_{L^\infty}^2 r^{1+\gamma} + C(p) \norm{u}_{L^\infty}^{2-p}.
	\end{equation}
	Applying Proposition \ref{prop:almostharmonic1form}, we can improve this to an estimate on $\int_{B(x,r)} |\nabla u|^2$.  
\end{proof}

\subsection{Lipschitz approximation}

Let the $\Z_2$-valued 1-form $v=\pi(u)$ be the projection of a Fueter section $u$. We consider the renormalized version $v/\norm{u}_{L^\infty}$ on a coordinate ball $B(x_0,1)$.

\begin{proposition}\label{prop:Lipapprox}
	(Lipschitz approximation) Given any $\lambda\geq 1$, let 
	\[
	\Omega_\lambda= \{ x\in B(x_0,1): M(\frac{|Dv|}{\norm{u}_{L^\infty}})\leq \lambda  \},
	\]
	where $M$ is the maximal function operator. Then there is a 2-valued function $v_\lambda$ on $B(x_0,1)$ with pointwise average zero, such that $\text{Lip}(v_\lambda)\leq C\lambda$, and
	\[
	E_\lambda = \{  x\in M: \frac{v(x)}{\norm{u}_{L^\infty}}\neq v_\lambda(x) \} \subset B(x_0,1)\setminus \Omega_\lambda
	\]
	satisfies the measure estimate
	\[
	|E_\lambda|\leq C\lambda^{-2}  \norm{u}_{L^\infty}^{-2} \norm{v}_{W^{1,2}}^2 \leq C\lambda^{-2} ,
	\]
	and moreover $d_{W^{1,2}}(\frac{v}{\norm{u}_{L^\infty}},v_\lambda) \leq C d_{W^{1,2}}( \frac{v}{\norm{u}_{L^\infty}}
	,2[\![  0 ]\!]) \leq C$.
\end{proposition}

\begin{proof}
	Once we regard $v$ as a 2-valued function valued in $\mathcal{A}_2(B(x_0,1), \R^3)$, 
	this follows from the Lipschitz approximation theorem in the theory of multi-valued functions \cite[Prop. 4.4]{MR2663735}. The idea is that $v|_{\Omega_\lambda}$ has Lipschitz estimate, and then one uses a Whitney extension construction to extend $v|_{\Omega_\lambda}$ to $B(x_0,1)$. The additional property that the Lipschitz approximation $v_\lambda$ has pointwise average zero can be arranged in the Whitney extension step. 
\end{proof}

When $\norm{u}_{L^\infty}\gg 1$, we can choose $\lambda$ to be very large, so the Lipschitz approximation
$v_\lambda$ coincides with $v/\norm{u}_{L^\infty}$ except on a set with small measure. Moreover $v_\lambda$ satisfies quantitative H\"older continuity.

\begin{corollary}[Macroscopic H\"older estimate] \label{cor:macroscopicHolder}
	Fix $0<p<\sqrt{2/3}$, and let $0<\gamma< 1$ be as in Proposition \ref{prop:energydecay}, and suppose 
	\[
	1 \leq \lambda\leq   \norm{u}_{L^\infty}^{p \frac{2-\gamma}{2(1+ \gamma)} } .
	\]
	We have the uniform H\"older estimate on the Lipschitz approximations,
	\[
	\norm{ v_\lambda}_{C^{\gamma/2}} \leq C(p) .
	\]
\end{corollary}

\begin{proof}
	We claim that on all coordinate balls $B(x,r)$ with $r\lesssim 1$, the Campanato-Morrey estimate holds:
	\begin{equation}
		\int_{B(x,r)} |\nabla v_\lambda|^2 \leq  C(p) r^{1+\gamma}  .
	\end{equation}
	The H\"older estimate would then follow from the Campanato-Morrey embedding theorem \cite[Prop. 4.8]{MR2663735}.

	First we notice that for $r\leq \norm{u}_{L^\infty}^{ -p/(1+\gamma) }$, then by the Lipschitz bound on $v_\lambda$,
	\[
	\int_{B(x,r)} |\nabla v_\lambda|^2 \leq C\lambda^2 r^3 \leq C  r^{1+\gamma}. 
	\]
	Thus it suffices to assume $r\geq \norm{u}_{L^\infty}^{ -p/(1+\gamma) }$.

	If $\norm{u}_{L^\infty}^{ -p/(1+\gamma) } \leq s\lesssim 1$, then by the energy decay estimate,
	\[
	\int_{B(y,s)} |\nabla v|^2 \leq C\norm{u}_{L^\infty}^2 s^{1+\gamma} + C(p) \norm{u}_{L^\infty}^{2-p} \leq C(p) \norm{u}_{L^\infty}^2 s^{1+\gamma}.
	\]
	For any $y\in B(x,r)$ in the complement of $\Omega_\lambda$, the maximal function $M(\frac{|Dv|}{\norm{u}_{L^\infty}})\geq \lambda$, so there is some small radius $0< s \leq \norm{u}_{L^\infty}^{ -p/(1+\gamma) }$, with
	\[
	\norm{u}_{L^\infty}^{-2}\int_{B(y,s)} |\nabla v|^2 \geq  \lambda^2 s^3.
	\]
	In particular, $B(y,s)\subset B(x,2r)$. By Vitali covering, we can find a collection of balls $B(y_i, s_i)\subset B(x,2r)$ covering $B(x,r)\setminus \Omega_\lambda$, such that $B(y_i, s_i/5)$ are disjoint, and in particular
	\[
	\sum_i (s_i/5)^3 \leq C\lambda^{-2} \sum_i \int_{B(y_i,s_i)} \norm{u}_{L^\infty}^{-2} |\nabla v|^2 \leq C\lambda^{-2} \norm{u}_{L^\infty}^{-2}\int_{B(x,2r)} |\nabla u|^2   .
	\]
	Hence
	\[
	\begin{split}
		\int_{B(x,r)} |\nabla v_\lambda|^2 & \leq \int_{B(x,r)} \norm{u}_{L^\infty}^{-2}|\nabla v|^2 +\lambda^2 \text{Vol}( B(x,r) \cap \{ \norm{u}_{L^\infty}^{-1} v\neq v_\lambda\} )
		\\
		& \leq  \norm{u}_{L^\infty}^{-2}\int_{B(x,r)} |\nabla v|^2+ C\lambda^2\sum_i s_i^3 
		\\
		& \leq   C \norm{u}_{L^\infty}^{-2}\int_{B(x,2r)} |\nabla u|^2   \leq C(p) r^{1+\gamma}. 
	\end{split}
	\]
	Thus the claim is verified. 
\end{proof}

\begin{corollary}[Approximate $\Z_2$-harmonic 1-form] \label{cor:approxZ2harmonic1form}
	Suppose $\lambda\geq 1$, then the Lipschitz approximation satisfies
	\[
	\int_{B(x_0,1)} |dv_\lambda|+ |d^*_g v_\lambda| \leq 
	C(p)\norm{u}_{L^\infty}^{-p/2} + C\lambda^{-1} .
	\]
\end{corollary}

\begin{proof}
	By Cauchy-Schwarz and Proposition \ref{prop:Lipapprox}
	\[
	\begin{split}
		\int_{B(x_0,1)} |dv_\lambda|+ |d^*_g v_\lambda| &  \leq   \norm{u}_{L^\infty}^{-1} (\int_{B(x_0,1)} |dv|+ |d^*_g v|)  + C\lambda |E_\lambda| \\
		& 
		\leq C \norm{u}_{L^\infty}^{-1}( \int_M |dv|^2 + |d^*_g v|^2 )^{1/2} + C\lambda^{-1}
		\\
		& \leq C(p)\norm{u}_{L^\infty}^{-p/2} + C\lambda^{-1} .
	\end{split}
	\]
	Here the last line uses Proposition \ref{prop:almostharmonic1form}. 
\end{proof}

\section{Convergence to $\Z_2$-harmonic 1-form}\label{sect:convergence}

This section concludes the proof of Theorem \ref{mainthm}.

\subsection{Regularity of $\Z_2$-harmonic 1-form}\label{sect:Z2harmonic1form}

We recall

\begin{definition}
	We say $\mathcal{V}$ is a $\Z_2$-harmonic 1-form on the compact 3-manifold $M$, if $\mathcal{V}$ is a H\"older continuous section of $T^*M/\Z_2$, such that on the complement of the closed set $Z=\mathcal{V}^{-1}(0)$, the section $\mathcal{V}$ admits locally smooth lift to $T^*M$ satisfying $d\mathcal{V}=d^*\mathcal{V}=0$, and the total energy $\int_{M\setminus Z} |\nabla \mathcal{V}|^2<+\infty$.  
\end{definition}

Our definition of $\Z_2$-harmonic 1-form can be compared to the notion of $\Z_2$-harmonic spinor of Boyu Zhang \cite[Def. 1.1]{MR4596625} (which is based on the previous work of Taubes \cite{taubes2014zero}) as follows. On a 3-manifold, we can regard $\nu= T^*M\oplus \R$ as a Clifford bundle, via the standard Clifford action of $\R^3\simeq \text{Im}(\mathbb{H})$ on $\R^3\oplus \R\simeq \mathbb{H}$. If $\mathcal{V}$ is a $\Z_2$-harmonic 1-form in our sense, then $(\mathcal{V},0)\in \Gamma(M, \nu)$ is the 3-manifold version of a harmonic spinor in Zhang's sense. Moreover, since we impose H\"older continuity on $\mathcal{V}$, the 3-manifold version of \cite[Assumption 1.2]{MR4596625} is satisfied. This means that there is some small $\epsilon>0$, such that at each $x\in Z$ and for $r$ smaller than some fixed small number, 
\[
\int_{B(x,r)} |\mathcal{V}|^2 \leq  Cr^{3+\epsilon}.
\]
The regularity theorem of Zhang then implies:

\begin{proposition}\cite{taubes2014zero, MR4596625}
	If $\mathcal{V}$ is a $\Z_2$-harmonic 1-form on $M$, then the zero locus $Z=\mathcal{V}^{-1}(0)$ is a (possibly empty) 1-rectifiable subset, and has finite Minkowski content, and in particular, finite Hausdorff $\mathcal{H}^1$-measure. 
\end{proposition}

\subsection{Weak limit of $\Z_2$-valued 1-forms}

We are given a sequence of Fueter sections $u_j\in \Gamma(M, \mathfrak{X}_{AH})$ with $\norm{u_j}_{L^\infty}\to +\infty$, and we consider the sequence of renormalized $\Z_2$-valued 1-forms $v_j= \frac{ \pi(u_j)}{ \norm{u_j}_{L^\infty}}$. Inside coordinate charts, we can locally trivialise the cotangent bundle of $M$, so we can regard $v_j$ locally as 2-valued functions as in \cite{MR1777737} and \cite{MR2663735} (\cf Section \ref{sect:energydecay}).

By construction $v_j$ are normalized to have $L^\infty$-norm one. By the topological energy estimate
\[
\int_M |\nabla \pi(u)|^2\leq C\int_M |\nabla u|^2 \leq C\norm{u}_{L^\infty}^2,
\]
we see that the sequence $v_j$ have uniform $W^{1,2}$-estimate
\begin{equation}
	\int_M |\nabla v_j|^2\leq C.
\end{equation}
By the Rellich compactness theorem in the Sobolev space of multivalued functions (which reduces to standard Rellich compactness using Almgren's bi-Lipschitz embedding, see \cite[Chapter 2]{MR2663735}), we can extract a subsequence, still denoted as $v_j$, converging strongly in $L^2$-sense to a limiting 2-valued function $\mathcal{V}$ in $W^{1,2}$. Since the average of the two sheets of $v_j$ is zero at every point, so is the limit $\mathcal{V}$.  Moreover,
\[
\int_{B(x,r)} |\nabla v|^2     \leq \liminf_j \int_{B(x,r)} |\nabla v_j|^2.
\]

\begin{lemma}\label{lem:Holderlimit}
	The $L^2$-limit $\mathcal{V}$ is H\"older continuous.
\end{lemma}

\begin{proof}
	We apply the uniform estimate (\ref{eqn:energydecayv}) and take into account the normalization factor $\norm{u_j}_{L^\infty}^{-1}$, to see
	\[
	\int_{B(x,r)} |\nabla v_j|^2 \leq C r^{1+\gamma} + C(p) \norm{u}_{L^\infty}^{-p}.
	\]
	Passing to the limit, the assumption $\norm{u_j}_{L^\infty}\to +\infty$ implies that
	\[
	\int_{B(x,r)} |\nabla \mathcal{V}|^2 \leq C r^{1+\gamma}.
	\]
	The Campanato-Morrey embedding theorem for multi-valued functions (\cf \cite[Prop. 2.14]{MR2663735}) shows that $\mathcal{V}$ is in the H\"older space $C^{\gamma/2}$.
\end{proof}

\begin{proposition}
	The $L^2$-limit $\mathcal{V}$ is a $\Z_2$-harmonic 1-form.
\end{proposition}

\begin{proof}
	The intuition is that $v_j$ are approximate $\Z_2$-valued harmonic 1-forms in some $L^2$-integral sense, and we wish to pass to some weak limit. The problem is that we do not have direct control on the modulus of continuity for $v_j$, so it is difficult to directly lift $v_j$ continuously to a local 1-form. The main idea is to use the Lipschitz approximation $v_\lambda$ from Proposition \ref{prop:Lipapprox}.

	Fix $0<p<\sqrt{2/3}$, and let $0<\gamma< 1$ be as in Proposition \ref{prop:energydecay}, and let
	\[
	\lambda_j=  \norm{u_j}_{L^\infty}^{ p \frac{2-\gamma}{2(1+ \gamma)} } \to +\infty.
	\]
	We produce Lipschitz approximations $v_{j,\lambda_j}$ from the Fueter sections $u_j$ using Proposition \ref{prop:Lipapprox}. Clearly the difference between $v_j$ and $v_{j,\lambda_j}$ converges to zero in $L^2$, so $v_{j,\lambda_j}\to \mathcal{V}$ on the coordinate ball in the $L^2$-sense. Moreover by Corollary 
	\ref{cor:macroscopicHolder}, we know $v_{j,\lambda_j}$ have uniform H\"older estimate, so the convergence is in fact in $C^{\gamma/2}$.

	In particular, around each point where $\mathcal{V}$ is non-zero, we can find a definite neighborhood where $\mathcal{V}$ and $v_{j,\lambda_j}$ can be lifted to  H\"older continuous local 1-forms, and the lift of $v_{j,\lambda_j}$ converges to the lift of $\mathcal{V}$ strongly in $C^0$. By Corollary 
	\ref{cor:approxZ2harmonic1form}, we have
	\[
	\int_{B(x_0,1)} |dv_{j,\lambda_j}|+ |d^*_g v_{j,\lambda_j}| \leq 
	C(p)\norm{u_j}_{L^\infty}^{-p/2} + C\lambda_j^{-1} \to 0.
	\]
	By passing the weak formulation to the limit, we deduce that the local lift of $\mathcal{V}$ is  a harmonic 1-form, and a fortiori smooth.  
\end{proof}

\begin{corollary}[No energy loss] \label{cor:noenergyloss}
	We have the convergence of energy
	\[
	\int_M |\nabla v_j|^2 \to \int_M |\nabla \mathcal{V}|^2.
	\]
	Equivalently,
	\[
	\norm{u_j}_{L^\infty}^{-2} \int_M |\nabla u_j|^2 \to \int_M |\nabla \mathcal{V}|^2.
	\]
\end{corollary}

\begin{proof}
	Away from the zero locus of $\mathcal{V}$, the harmonic $\Z_2$-valued 1-form $\mathcal{V}$ has a smooth local lift, and by the Weitzenb\"ock formula
	\[
	\nabla^* \nabla \mathcal{V}= (dd^*+d^*d) \mathcal{V}- Ric_{ij} \mathcal{V}_i e_j^*= - Ric_{ij} \mathcal{V}_i e_j^*.
	\]
	Thus
	\[
	\frac{1}{2} \Lap |\mathcal{V}|^2= |\nabla \mathcal{V}|^2-
	(\mathcal{V}, \nabla^* \nabla \mathcal{V})= |\nabla \mathcal{V}|^2+ Ric_{ij}\mathcal{V}_i\mathcal{V}_j.
	\]
	We notice that both sides are defined independent of the choice of local lift.

	For any cutoff function $\chi$ which is supported away from the zero locus of $\mathcal{V}$, we have
	\[
	\int_M (\Lap \chi) |\mathcal{V}|^2= \int_M \chi \Lap |\mathcal{V}|^2 = 2\int_M  \chi( |\nabla \mathcal{V}|^2+ Ric_{ij}\mathcal{V}_i\mathcal{V}_j).
	\]
	Using the finiteness of the Minkowski content, we can choose $\chi$ so that $\Lap \chi$ is supported in the $r$-neighborhood of $\mathcal{V}^{-1}(0)$, and $\int_M |\Lap \chi| \leq C$. By the H\"older regularity, on the support of $\Lap \chi$, we have $|\mathcal{V}|\leq Cr^{\gamma/2}$. Thus
	\[
	\int_M |\Lap \chi| |\mathcal{V}|^2\leq Cr^{\gamma} \to 0,
	\]
	as $r\to 0$. Therefore we can remove the cutoff and deduce
	\begin{equation}
		\int_M   |\nabla \mathcal{V}|^2+ Ric(\mathcal{V},\mathcal{V}) =0.
	\end{equation}
	In contrast, by Corollary \ref{cor:L2gradient} and Corollary \ref{cor:L2gradientpowerlaw2},
	\[
	\int_M |\nabla v_j|^2+ Ric(v_j,v_j) \leq C \norm{u_j}_{L^\infty}^{-p} \to 0.
	\]
	By the strong convergence in $L^2$, we know
	\[
	\int_M Ric(v_j,v_j) \to \int_M Ric(\mathcal{V}, \mathcal{V}).
	\]
	Combining the above,
	\[
	\int_M |\nabla \mathcal{V}|^2 \geq \limsup \int_M |\nabla v_j|^2. 
	\]
	Contrasting this with Fatou's lemma, we deduce that there is no energy loss.
\end{proof}

\begin{corollary}
	The convergence $v_i\to \mathcal{V}$ is strong in $W^{1,2}$-topology for 2-valued functions.
\end{corollary}

\begin{proof}
	By Corollary \ref{cor:noenergyloss}, the energy cannot concentrate near sets of very small measure, so it suffices to prove $W^{1,2}$-convergence on compact subsets of $M\setminus \mathcal{V}^{-1}(0)$, where $\mathcal{V}$ admits smooth local lifts to 1-forms.

	Let \[\lambda_j= \norm{u_j}_{L^\infty}^{ p \frac{2-\gamma}{2(1+ \gamma)} } \to +\infty.\]
	We construct the Lipschitz approximations $v_{j,\lambda_j}$ using Proposition \ref{prop:Lipapprox} as before, which agrees with $v_j$ except on the set $ M\setminus \Omega_{\lambda_j}$. 
	Clearly for any $\lambda>0$,
	\[
	|M\setminus \Omega_\lambda| \leq C\lambda^{-2} \int_M |\nabla v_j|^2 \leq C\lambda^{-2}. 
	\]
	Again by Corollary \ref{cor:noenergyloss}, there cannot be energy concentration on sets with measure tending to zero, so 
	\[
	\int_{M\setminus \Omega_{\lambda_j/2}} |\nabla v_j|^2 \to 0.
	\]
	By the weak $(2,2)$-estimate for maximal functions,
	\[
	\lambda_j^2|M\setminus \Omega_{\lambda_j}| \leq C \int_{M\setminus \Omega_{\lambda_j/2}} |\nabla v_j|^2 \to 0.
	\]
	Since $v_j=v_{j,\lambda_j}$ in $\Omega_{\lambda_j}$, we see that 
	\begin{equation}
		d_{W^{1,2}} (v_j, v_{j,\lambda_j})\to 0. 
	\end{equation}
	Thus it suffices to prove that the Lipschitz approximations $v_{j,\lambda_j}\to \mathcal{V}$ strongly in $W^{1,2}$-topology.

	By the above, we have the energy convergence
	\[
	\int_M |\nabla v_{j,\lambda_j}|^2\to \int_M |\nabla \mathcal{V}|^2.
	\]
	What we gained is that $v_{j,\lambda_j}$ have uniform H\"older bounds, and converge in $C^0$-topology to $\mathcal{V}$. Thus $v_{j,\lambda_j}$ and $\mathcal{V}$ have continuous local lifts to 1-forms, and the statement reduces to the standard fact that for \emph{single-valued} functions, weak convergence and convergence in $W^{1,2}$-norm implies strong $W^{1,2}$-convergence.  
\end{proof}

This finishes the proof of the main Theorem \ref{mainthm}.

\section{Open questions}\label{sec:open-questions}

We discuss several directions of open questions concerning the compactness of Fueter sections, which is important for defining any enumerative invariants or Floer homology groups by counting Fueter sections of monopole bundles.

\begin{question}
	(Generic compactness)
	Suppose $g$ is a generic fixed Riemannian metric on $M$. Then is the moduli space of Fueter sections compact?
\end{question}

A positive answer to this question is important for the purpose of defining counting 3-manifold invariants. The examples on the product manifolds $\Sigma\times S^1$ produce non-compact moduli spaces, but the metric splits as a product, hence is highly non-generic. Since the linearised Fueter operator has index zero, in general, we expect that a gluing strategy would produce a sequence of Fueter sections with $L^\infty$-norm diverging to infinity when the underlying Riemannian metrics on the 3-fold varies in a \emph{1-parameter family}. For a fixed generic metric $g$, it seems unclear if there can be an infinite number of isolated Fueter sections.

\begin{question}
	(Codimension two smooth convergence)
	Given a sequence $u_j$ of Fueter sections into the Atiyah-Hitchin bundle on a compact and oriented 3-manifold, with $\norm{u_j}_{L^\infty}\to +\infty$,
	is there some Hausdorff codimension two locus $S\subset M$, such that the normalized sections $\pi(u_j)/\norm{u_j}_{L^\infty}$ converge in $C^\infty_{loc}$-topology away from $M\setminus S$?
\end{question}

If the Fueter sections stay in a uniformly bounded region of the target space, then the $L^2$-energy is uniformly bounded, and a standard consequence of $\epsilon$-regularity is that the Fueter sections will subsequentially converge in $C^\infty_{loc}$-topology, away from a codimension two singular set.

When the $L^\infty$-norms of the Fueter sections diverge to infinity, the $L^2$-energy also diverges, so the $\epsilon$-regularity argument breaks down. In the simplest examples of product type, we have not found a counterexample on a \emph{compact} oriented 3-manifold, but on non-compact 3-manifolds we notice that the failure of $C^0$-convergence can happen on increasingly dense subsets. This points towards the answer being false, and if the answer happens to be true then there must be a rather non-trivial global reason.

We also note that our main Theorem \ref{mainthm} together with the codimension two regularity of $\Z_2$-harmonic 1-forms \emph{do not} imply a positive answer to the question. This is because the convergence of $\pi(u_k)/\norm{u_k}_{L^\infty}$ to the limiting $\Z_2$-harmonic 1-form fails to be $C^0$ in general, so one cannot expect to transfer the regularity of the limit to the sequence.

\begin{question}(Small scale blow-up limit)
	Given a sequence of Fueter sections $u_j$ as above, and suppose the $\Z_2$-harmonic 1-form $v$ is the $L^p$-limit of $\pi(u_j)/\norm{u_j}_{L^\infty}$
    We assume that the zero locus $Z$ of $v$ is a graph with smooth edges in the 3-manifold $M$. Around any given point $p$ on an edge of $S$, we say a non-constant Fueter map $\tilde{u}: \R^3\to X_{AH}$ is a \emph{local blow-up limit} for the sequence $u_j$, if there is a sequence of central points $p_j\to p\in M$, and radius parameters $r_j\to 0$, 
	such that the Fueter sections $\tilde{u}_j$ defined by
	\[
	\tilde{u}_j(y) = u_j ( exp_g(p_j, r_j y)),\quad |y|\leq r_j^{-1},
	\]
	over the geodescic ball with rescaled metric $(B_g(p_j, 1), r_j^{-2}g)$, converge in $C^\infty_{loc}$ to the limiting Fueter section $\tilde{u}$. Is there always such a local blow-up limit? If yes, is it true that for generic $p\in S$, the local blow-up limit has to be translation-invariant along the $\R$ direction parallel to the tangent space of $S$?

\end{question}

This question may be compared with holomorphic sphere bubble extraction in the setting of compact hyperk\"ahler targets \cite[Thm. 1.9]{MR3718486}, which depends heavily on the energy monotonicity formula.
To answer the question in the positive, two inter-related ingredients seem highly desirable:
\begin{itemize}
	\item  We need some quantitative control on the growth rate of the $L^2$-energy on small geodesic balls on $(M,g)$.
	
	\item We need some monotonicity formula, whose coercivity effect guarantees that the Fueter sections $u_j$ are close to being translation-invariant on small scales.
\end{itemize}

In our context the total $L^2$-energy diverges to infinity, so the energy monotonicity formula does not provide adequate control, and the natural idea is to use the monotonicity of frequency function instead. On the other hand, the results about frequency function proved in this paper become degenerate at very small length scales, and in particular, it is unclear if the frequency is uniformly bounded at the length scale suitably for detecting the local blow-up limit. 
It seems a more clever argument or some better variant of the frequency function may be needed.

\begin{question}
	(Higher charge monopoles)
	What if we replace the Atiyah-Hitchin monopole moduli space with higher charge monopole moduli spaces $\text{Mon}^\circ_k(\R^3)$?    
\end{question}

More precisely, the $SO(3)$-action on $\R^3$ induces an action on the hyperk\"ahler manifold $\text{Mon}^\circ_k(\R^3)$, and we can consider the Fueter sections $u$ into the hyperk\"ahler bundle $\mathfrak{X}_k= P_{SO(3)}\times_{SO(3)} \text{Mon}^\circ_k(\R^3)$ (or the bundle formed from the finite covers of $\text{Mon}^\circ_k(\R^3)$).

The key difference is that the charge $k\geq 3$ centered monopole moduli spaces $\text{Mon}^\circ_k(\R^3)$ have a more complicated asymptote at infinity. This can be informally explained as follows. The mechanism for a sequence of monopoles in  $\text{Mon}^\circ_k(\R^3)$ to diverge to infinity, is that the monopole can separate into a superposition of `clusters' of smaller charge monopoles. (The rather complicated details can be found in \cite{MR4493580}.)

\begin{example} \emph{
	In the charge $k=2$ case, namely the Atiyah-Hitchin manifold, the only possibility is separation into two charge one monopoles, which must be centered around opposite points in $\R^3$ due to the requirement for the monopole to be centered.}
	
	\emph{
	In the charge $k=3$ case, the generic behaviour near infinity is for the charge three monopole to decompose into $3$ far separated single monopoles whose center of mass sits at the origin (`3=1+1+1'). However, in some non-generic region near the infinity, the charge three monopole can also decompose into a charge one monopole, and a charge two monopole (`3=1+2'). In general, the pattern of monopole separation is controlled by the partitions of $k$ into positive integers.}
\end{example}

This leads to two expectations:
\begin{itemize}
	\item (Inductive nature) The compactness problem for Fueter sections into the charge $k$ monopole moduli bundle, will contain the charge $j\leq k-1$ problems.
	
	\item (Generic region behaviour)  The Fueter section $u$ may land inside the generic region of $\text{Mon}^\circ_k(\R^3)$, corresponding to the complete separation into single monopoles ($k=1+1+\ldots +1$). In this region, the moduli space $\text{Mon}^\circ_k(\R^3)$ admits local coordinates 
	\[
	(\vec{x}_1,\ldots \vec{x}_k)\in (\R^3)^k/Sym_k,\quad |\vec{x}_i- \vec{x}_j|\gg 1,\quad \sum_1^k \vec{x}_i=0,
	\]
	encoding the center positions of the single monopoles modulo the permutation group $Sym_k$.
	Taking the associated bundle construction, there is a natural projection map of $u$ into the bundle
	\[
	P_{SO(3)}\times_{SO(3)} \{ (\vec{x}_i)_1^k\in  (\R^3)^k|\sum_1^k \vec{x}_i=0  \}/Sym_k\simeq (\bigoplus_1^{k-1} T^*M) /Sym_k. 
	\]
	Now given a sequence of Fueter sections $u_j$ diverging to infinity in the generic region of $\text{Mon}^\circ_k(\R^3)$, we obtain a sequence of sections of $
	(\bigoplus_1^{k-1} T^*M) /Sym_k
	$ by projection and rescaling, which we can normalize to have norm one. These are the analogue of $\Z_2$-valued 1-forms for higher $k$.

	We expect to extract a limit $v$ of these sections, such that $v$ is a continuous section of  $(\bigoplus_1^{k-1} T^*M) /Sym_k$, and away from all the subdiagonals (\ie  when $\vec{x}_i\neq \vec{x}_j $ for any $i\neq j$), $v$ lifts to a locally smooth section of $\bigoplus_1^{k-1} T^*M$, such that each direct summand is a harmonic 1-form. This limit $v$ should be viewed as the generalization of the $\Z_2$-harmonic 1-form.

	 As a caveat, in this discussion, we have assumed that the Fueter sections land in the generic region of the monopole moduli space. However, we expect that there can also be sequences of Fueter sections that diverge to infinity along the \emph{non-generic} directions.
\end{itemize}

\appendix

\section{Appendix: Frequency function}\label{appendix}

The goal of this appendix is to introduce a tentative definition of the frequency function for Fueter sections, and prove partial results on the monotonicity and boundedness of frequency.

\subsection{Frequency function}

Let $\phi: \R_{\geq 0}\to \R_{\geq 0}$ be a non-increasing smooth cutoff function, which is supported in the interval $[0,1]$, and equal to one on $[0,3/4]$. 
Let $f(u)=F(u)$ be a radially symmetric function on the Atiyah-Hitchin manifold, such that $F$ is strictly positive and non-decreasing, and for $|u|$ large enough it agrees with $|u|^2$, and globally \[\Lap f(u)\geq -C|u|^2,\](\cf Section \ref{sect:C0}).

Given a Fueter section $u$ on $M$, then inside geodesic balls with radius smaller than a fixed small constant, we consider an analogue of mollified Almgren frequency function:
\[
\begin{cases}
	H_\phi (x,r)= - \int_{ B(x,r) } f(u)(y) d(x,y)^{-1} \phi'( \frac{d(x,y)}{r} ) dy,
	\\
	D_\phi (x,r) = \int_{B(x,r) } |\nabla u|^2 \phi( \frac{d(x,y)}{r} ) dy,
	\\
	I_\phi (x,r) = \frac{ r D_\phi(x,r) } { H_\phi(x,r)  }.
\end{cases}
\]
and write $\nu$ as the  gradient vector field of the radial distance $d_x(y)=d(x,y)$.

For comparison, the mollified frequency function for any non-zero $\Z_2$-harmonic 1-form $\mathcal{V}$ is defined by
\[
\begin{cases}
	H_\phi^{\mathcal{V}} (x,r)= - \int_{ B(x,r) } |\mathcal{V}|^2 d(x,y)^{-1} \phi'( \frac{d(x,y)}{r} ) dy,
	\\
	D_\phi^{\mathcal{V}} (x,r) = \int_{B(x,r) } |\nabla \mathcal{V}|^2 \phi( \frac{d(x,y)}{r} ) dy,
	\\
	I_\phi^{\mathcal{V}} (x,r) = \frac{ r D_\phi(x,r) } { H_\phi(x,r)  }.
\end{cases}
\]
The unique continuation property says that $\mathcal{V}$ cannot vanish on any open subset, so $H_\phi^\mathcal{V}$ is never zero and $I_\phi^{\mathcal{V}}$ makes sense.

\begin{lemma}\label{lem:frequencyconvergence}
	Given a sequence of Fueter sections $u_j$ such that $\norm{u_j}_{L^\infty}\to +\infty$, and suppose that $\mathcal{V}$ is the strong $W^{1,2}$-limit of $\frac{u_j}{ \norm{u_j}_{L^\infty}}$ as in Theorem \ref{mainthm}. Then on any fixed geodesic ball
	\[
	\lim_j I_\phi(u_j)(x,r)= I_\phi^{\mathcal{V}} (x,r).
	\]

\end{lemma}

\begin{proof}
	By strong $L^2$-convergence, $ \norm{u_j}_{L^\infty}^{-2}  H_\phi(u_j)(x,r)$ converges to $H_\phi^{\mathcal{V}}(x,r)$, which is non-zero. By strong $W^{1,2}$-convergence, and Proposition \ref{prop:almostharmonic1form} which ensures that the energy in the circle component is negligible, we know that \[\norm{u_j}_{L^\infty}^{-2} D_\phi(u_j)(x,r)\to D_\phi^\mathcal{V} (x,r).\] This proves the convergence of the frequency. 
\end{proof}

\begin{remark}\emph{
	The mollification ensures that the quantities only depend on $W^{1,2}$ regularity, which is convenient for passing to $W^{1,2}$-limit. We note that the definition of $I_\phi$ (resp. $H_\phi, D_\phi$) satisfying the above Lemma is not very canonical; for instance, it depends on the choice of $f$, which only agrees with $|u|^2$ when $|u|$ is large enough. 
	Another natural but non-equivalent alternative is to define $D_\phi$ as
	\[
	\int_{B(x,r) } |\nabla \pi( u)|^2 \phi( \frac{d(x,y)}{r} ) dy.
	\]
	However, the computation of the $r$-derivative of $D_\phi(x,r)$ below does not seem to obviously carry over to this alternative definition.
	}
\end{remark}

\begin{lemma}\label{lem:L2nonvanishing}
	Suppose $\norm{u}_{L^\infty}$ is sufficiently large. Then on geodesic balls with radius comparable to one, we have a uniform $L^2$-non-vanishing estimate
	\[
	H_\phi(x,r) \geq C^{-1}\norm{u}_{L^\infty}^2.
	\]
\end{lemma}

\begin{proof}
	Suppose the contrary, we can find a sequence of Fueter sections with $\norm{u_j}_{L^\infty}\to +\infty$, and $x_j\to x, r_j\to r>0$, but 
	\[
	\norm{u_j}_{L^\infty}^{-2} H_\phi(u_j) (x_j, r_j) \to 0.
	\]
	By Theorem \ref{mainthm}, we can extract a subsequence such that $\frac{\pi(u_j)}{ \norm{u_j}_{L^\infty}}\to \mathcal{V}$ strongly in $L^2$, for some non-zero $\Z_2$-harmonic 1-form $\mathcal{V}$. Thus $\mathcal{V}$ vanishes on some open set $B(x, r/2)$, which contradicts the unique continuation property.
\end{proof}

\begin{lemma}
	The frequency of non-zero $\Z_2$-harmonic 1-forms on $M$ is uniformly bounded independent of the geodesic ball and $\mathcal{V}$. 
\end{lemma}

\begin{proof}
	The frequency $I_\phi^{\mathcal{V}}(x,r)+ Cr^2$ is known to be monotone increasing on small geodesic balls \cite[Cor. 4.7]{MR4596625}. Moreover, the total energy is bounded by
	\[
	\int_M |\nabla \mathcal{V}|^2 = \int_M -Ric(\mathcal{V}, \mathcal{V}) \leq C \int_M |\mathcal{V}|^2,
	\]
	so the frequency on geodesic balls with radius of order one is uniformly bounded, which by monotonicity passes to all smaller balls.
\end{proof}

Lemma \ref{lem:frequencyconvergence} then implies that on \emph{fixed geodesic balls}, the frequency of Fueter sections are \emph{uniformly bounded} when $\norm{u}_{L^\infty}$ is very large. On the other hand, this bound on frequency is a priori dependent on the small radius $r$.

 In the following we make some progress on the question of how small we can push the radius $r$, so that the frequency satisfies a uniform estimate.

\subsection{Monotonicity formula and boundedness}

\begin{lemma}\label{lem:Hphiderivative}
	The derivative of $H_\phi(x,r)$ is
	\[
	\frac{d}{dr} H_\phi(x,r)= r^{-1} \int -\nabla_\nu f(u) \phi'( \frac{d(x,y)}{r})dy= \int \Lap f(u) \phi( \frac{d(x,y)}{r}) dy. 
	\]
	In particular, the function $H_\phi(x,r) e^{Cr^2}$ is increasing for some constant $C$. 
\end{lemma}

\begin{proof}
	We compute by repeated integration by parts,
	\[
	\begin{split}
		\frac{d}{dr} H_\phi(x,r)=&  r^{-2} \int f(u)(y) \phi''( \frac{d(x,y)}{r}) dy
		\\
		=&  r^{-1} \int f(u)\nabla_\nu ( \phi'( \frac{d(x,y)}{r}) ) dy
		\\
		= &  r^{-1} \int -\nabla f(u) \cdot (\phi'( \frac{d(x,y)}{r}) \nu )dy
		\\
		= &  \int \Lap f(u) \phi( \frac{d(x,y)}{r}) dy
	\end{split}
	\]
	Here $\Lap f(u)$ was computed in Lemma \ref{lem:C0Lap}.

	We make several observations. First, by construction 
    \[\Lap f(u)\geq - C|u|^2 \geq -Cf(u),\] 
    so
	\[
	\frac{d}{dr} H_\phi(x,r) \geq -C \int  f(u) \phi( \frac{d(x,y)}{r}) dy .
	\]
	We then use the mean value inequality  to obtain an interior $L^\infty$-estimate for $f(u)$ in terms of $H_\phi$, and we can arrange that $\phi\leq -C\phi'$ close to the boundary of the support of $\phi$, so that
	\begin{equation}\label{eqn:meanvaluef(u)}
		\int f(u) \phi dy\leq  C\int f(u) \phi' dy.
	\end{equation}
	Thus
	\[
	\frac{d}{dr} H_\phi(x,r) \geq -C \int  f(u) \phi( \frac{d(x,y)}{r}) dy \geq -Cr H_\phi(x,r),
	\]
	so there is some constant $C$ such that $H_\phi(x,r)e^{Cr^2}$ is an increasing function. 
\end{proof}

\begin{lemma}\label{lem:Dphiderivative}
	We have
	\[
	\frac{d}{dr}(r D_\phi(x,r))= -2 \int  |\nabla_\nu u|^2  \frac{ d(x,y)}{r} \phi' + 
	O(  r^{3/2} D_\phi(x,r)^{1/2} H_\phi(x,r)^{1/2} + r^2 D_\phi (x,r)  ) .
	\]
\end{lemma}

\begin{proof}
	We observe that using (\ref{eqn:Fuetersecondorder}) and the curvature commutation rules, we can compute in an orthonormal frame $\{ e_i\}$ with $\nabla e_i=0$ at a point,
	\[
	\begin{split}
		\frac{1}{2} \nabla_\nu |\nabla u|^2= &  (\bar{\nabla}_\nu \nabla_i u , \nabla_i u)= (\bar{\nabla}_i \nabla_\nu u, \nabla_i u) +  (\nabla_{ [\nu, e_i]} u, \nabla_{e_i} u) + ( Rm(\nu, e_i) \cdot u, \nabla_i u) 
		\\
		=&   \nabla_i (\nabla_\nu u, \nabla_i u)-  (\Lap u, \nabla_\nu u) +  ( \nabla_{ [\nu, e_i]} u, \nabla_{e_i} u) +    O(|u| |\nabla u|)
		\\
		=&   \nabla_i (\nabla_\nu u, \nabla_i u) + \frac{1}{d(x,y)} (|\nabla_\nu u|^2 - |\nabla u|^2)+    O(|u| |\nabla u|+ d(x,y)|\nabla u|^2).
	\end{split}
	\]
	Here the big $O$ error terms come from the deviation of the domain from being Euclidean.

	\[
	\begin{split}
		 \frac{d}{dr} \int |\nabla u|^2 \phi(\frac{d(x,y)}{r}) dy
		 =&   -\frac{1}{r^2}  \int |\nabla u|^2  d(x,y) \phi'(\frac{d(x,y)}{r})
		\\ =& -\frac{1}{r}  \int |\nabla u|^2  d(x,y) \nabla_\nu( \phi(\frac{d(x,y)}{r}))
		\\
		 =& \frac{1}{r}  \int  \nabla_\nu (|\nabla u|^2 ) d(x,y) \phi(\frac{d(x,y)}{r}) +  \frac{1}{r}  \int  |\nabla u|^2  \phi(\frac{d(x,y)}{r})
		\\
		=&  \frac{2}{r}  \int  \nabla_i (\nabla_\nu u,\nabla_i u ) d(x,y) \phi +  
		\frac{2}{r}  \int (|\nabla_\nu u|^2 - |\nabla u|^2) \phi
		\\
		&  +  \frac{1}{r}  \int  |\nabla u|^2  \phi +O(  \int (|\nabla u| |u| + d(x,y)|\nabla u|^2) \phi )
		\\
		= &  \frac{-2}{r}  \int   (\nabla_\nu u,\nabla_i u ) \nabla_i(d(x,y) \phi )+  \frac{1}{r}  \int  (2|\nabla_\nu u|^2-|\nabla u|^2)  \phi \\& +O(  \int (|\nabla u| |u| + r|\nabla u|^2) \phi
		\\
		=& \frac{-2}{r}  \int  |\nabla_\nu u|^2  \frac{ d(x,y)}{r} \phi' - \frac{1}{r}  \int  |\nabla u|^2  \phi(\frac{d(x,y)}{r}) \\& + O(  \int (|\nabla u| |u|+r|\nabla u|^2) \phi).
	\end{split}
	\]
	Thus by rearranging,
	\[
	\frac{d}{dr}(r D_\phi(x,r))= -2 \int  |\nabla_\nu u|^2  \frac{ d(x,y)}{r} \phi' + O(  \int (r|\nabla u| |u|+r^2|\nabla u|^2) \phi)
	\]
	Using (\ref{eqn:meanvaluef(u)}) and Cauchy-Schwarz, we see
	\[
	\int r|\nabla u| |u| \phi \leq C r^{3/2} H_\phi(x,r)^{1/2} D_\phi(x,r)^{1/2},
	\]
	hence the result.
\end{proof}

The following Proposition morally says that the frequency is almost monotone, except when $H_\phi(x,r)$ is much smaller than some sublinear power of $\norm{u}_{L^\infty}^2$.

\begin{proposition}\label{prop:approximatemonotonicity}
	(Approximate monotonicity)
	Let $0<p<\sqrt{2/3}$ be a fixed exponent.
	Suppose $H_\phi(x,r) \geq \norm{u}_{L^\infty}^{2-p} \geq 1$, then
	\[
	\frac{d}{dr} I_\phi \geq -C(p) (rI_\phi+ 1).
	\]
	
\end{proposition}

\begin{proof}
	We compute using Lemmas \ref{lem:Hphiderivative} and \ref{lem:Dphiderivative},
	\begin{equation}\label{eqn:derivativeIphi}
		\begin{split}
			\frac{d}{dr} I_\phi(x,r) = &   \frac{d}{dr} (  \frac{ r D_\phi(x,r) } { H_\phi(x,r)  } )
			\\
			= &  H_\phi^{-1} \frac{d}{dr} (   r D_\phi(x,r) ) - rD_\phi(x,r) H_\phi^{-2} \frac{d}{dr} H_\phi(x,r) 
			\\
			= & -2 H_\phi^{-1} (\int  |\nabla_\nu u|^2  \frac{ d(x,y)}{r} \phi' + 
			O(  r^{3/2} D_\phi(x,r)^{1/2} H_\phi(x,r)^{1/2} + r^2 D_\phi (x,r)  )) 
			\\
			 & +  D_\phi(x,r) H_\phi^{-2}\int \nabla_\nu f(u) \phi'( \frac{d(x,y)}{r}) dy
			\\
			= & -2 H_\phi^{-1} \int  |\nabla_\nu u|^2  \frac{ d(x,y)}{r} \phi' \\ & +   D_\phi(x,r) H_\phi^{-2}\int \nabla_\nu f(u) \phi' dy +
			O(  r^{1/2} I_\phi^{1/2} + r I_\phi  )) .
		\end{split}
	\end{equation}
	Now using the Laplacian computation in Lemma \ref{lem:C0Lap}, followed by Proposition \ref{prop:almostharmonic1form} and Corollary \ref{cor:L2gradientpowerlaw2}, for given $0<p<\sqrt{2/3}$ and $\norm{u}_{L^\infty}$ large enough, we have
	\[
	\begin{split}
		D_\phi(x,r)
		=& \int_{B(x,r)} |\nabla \pi(u)|^2 \phi + O( \int_{B(x,r)} |\nabla u|^2 (1+|u|)^{-1} \phi + \int_{M\cap \{ |u|\geq 1\} } |\vartheta(\nabla u)|^2 )
		\\
		=&  \frac{1}{2}\int \Lap f(u) \phi + O( \int_{B(x,r)} |\nabla u|^2 (1+|u|)^{-1} \phi + \int_{M\cap \{ |u|\geq 1\} } |\vartheta(\nabla u)|^2 +\int |u|\phi )
		\\
		=& \frac{1}{2} \int_{B(x,r)} \Lap f(u) \phi+ O( C(p) \norm{u}_{L^\infty}^{2-p} )
		\\
		=& \frac{1}{2r} \int_{B(x,r)} -\nabla_\nu f(u) \phi'+ O( C(p) \norm{u}_{L^\infty}^{2-p} ).
	\end{split}
	\]
	We now consider the main term $\int_{B(x,r)} \nabla_\nu f(u) \phi'$. Since $f$ agrees with $|u|^2$ when $|u|$ is larger than a fixed constant, we see
	\[
	|\nabla_\nu f(u)|^2 \leq 4|\nabla_\nu  u|^2 f(u)(1+ C (1+ |u|)^{-1})
	\]
	where the second term on the RHS absorbs the error from smaller values of $|u|$, and the error from the deviation between the Euclidean metric on $\R^3$ and the Gibbons-Hawking asymptote of the Atiyah-Hitchin manifold. 
 
    By Cauchy-Schwarz, and using the fact that $\phi'\leq 0$, we obtain
	\[
	\begin{split}
		& ( \int_{B(x,r)} \nabla_\nu f(u) \phi')^2 
		\\
		&  \leq 4(\int_{B(x,r)} -|\nabla_\nu u|^2(y) \frac{d(x,y)}{r} \phi' dy  - C\int |\nabla u|^2 (1+ |u|)^{-1} \frac{d(x,y)}{r} \phi'  ) (\int -f(u)(y) r d(x,y)^{-1} \phi' dy  ) 
		\\
		& \leq 
		4r H_\phi (\int_{B(x,r)} -|\nabla_\nu u|^2(y) \frac{d(x,y)}{r} \phi' dy  + C(p)\norm{u}_{L^\infty}^{2-p} ) .
	\end{split}
	\]
	Thus
	\[
	\begin{split}
		 -D_\phi(x,r)& H_\phi^{-2}\int \nabla_\nu f(u) \phi'( \frac{d(x,y)}{r}) dy
		\\ 
		=& \frac{1}{2rH_\phi^2 } (\int_{B(x,r)} \nabla_\nu f(u) \phi' )^2 +O ( C(p) \norm{u}_{L^\infty}^{2-p}  H_\phi^{-2} | \int_{B(x,r)} \nabla_\nu f(u) \phi'|)
		\\
		=& \frac{1}{2rH_\phi^2 } (\int_{B(x,r)} \nabla_\nu f(u) \phi' )^2+ O( C(p) r \norm{u}_{L^\infty}^{2-p} H_\phi^{-2}D_\phi + C(p) r \norm{u}_{L^\infty}^{4-2p}H_\phi^{-2} )
		\\
		\leq& \frac{-2}{H_\phi } \int |\nabla_\nu u|^2(y) \frac{d(x,y)}{r} \phi' dy
		+ C(p) r \norm{u}_{L^\infty}^{2-p} H_\phi^{-1} I_\phi \\& + C(p) r \norm{u}_{L^\infty}^{4-2p}H_\phi^{-2}+ C(p) \norm{u}_{L^\infty}^{2-p} H_\phi^{-1}.
	\end{split}
	\]
	Plugging this into (\ref{eqn:derivativeIphi}), the main terms cancel, and we are left with
	\[
	\frac{d}{dr} I_\phi \geq - (C(p) r \norm{u}_{L^\infty}^{2-p} H_\phi^{-1} I_\phi + C(p) r \norm{u}_{L^\infty}^{4-2p}H_\phi^{-2}  + C(p) \norm{u}_{L^\infty}^{2-p} H_\phi^{-1}+ Cr^{1/2}I_\phi^{1/2}+ rI_\phi).
	\]
	Assuming that $H_\phi\geq \norm{u}_{L^\infty}^{2-p}$, this implies
	\[
	\frac{d}{dr} I_\phi \geq - C(p) (rI_\phi+ r^{1/2}I_\phi^{1/2}+ 1)\geq - C(p)(rI_\phi+ 1)
	\]
	as required.
\end{proof}

\begin{corollary}\label{cor:frequencyuniformbound1}
	Suppose $\norm{u}_{L^\infty}$ is sufficiently large, then as long as $H_\phi(x,r)  \geq \norm{u}_{L^\infty}^{2-p}$, we have 
	$
	I_\phi(x,r) \leq C(p),
	$
	where $C(p)$ is independent of the geodesic ball or the Fueter section $u$. 
\end{corollary}

\begin{proof}
	First, we recall that on balls of radius comparable to one, the frequency is uniformly bounded as a consequence of the frequency bound on $\Z_2$-harmonic 1-forms, Theorem \ref{mainthm} and the frequency convergence Lemma \ref{lem:frequencyconvergence}. The problem is to deal with $r\ll 1$.

	Since $H_\phi(x,r) e^{Cr^2}$ is monotone increasing in $r$, and by assumption $H_\phi(x,r) \geq \norm{u}_{L^\infty}^{2-p}$ for the given value of $x,r$, we see that for $r\leq s\lesssim 1$,
	\[
	H_\phi(x,s) \geq e^{-C(s^2-r^2)} H_\phi(x,r) \geq  e^{-C(s^2-r^2)} \norm{u}_{L^\infty}^{2-p} \geq C^{-1} \norm{u}_{L^\infty}^{2-p},
	\]
	where the constants do not depend  on $r,s$. This constant does not affect the proof of Proposition \ref{prop:approximatemonotonicity}, so
	\[
	\frac{d}{ds} I_\phi(x,s) \leq -C(p) (sI_\phi(x,s)+1),\quad r\leq s\lesssim 1.
	\]
	This differential inequality implies that there is some uniform constant depending on $p$, but not on $r,s$, such that
	\[
	I_\phi(x,s) \leq C(p),\quad r\leq s\lesssim 1,
	\]
	as required.
\end{proof}

\begin{proposition}\label{prop:frequencyuniformbound2}
	There is some small exponent $0<q<1$, such that whenever $u$ is a Fueter section with $\norm{u}_{L^\infty}$ large enough depending on $q$,  then for any radius $\norm{u}_{L^\infty}^{-q}\leq r\lesssim 1$, on the geodesic balls $B(x,r)$ the frequency is uniformly bounded with constants independent of $u, x,r$. 
\end{proposition}

\begin{proof}
	On geodesic balls with radius comparable to one, the frequency is uniformly bounded, and
	\begin{equation}
		H_\phi \geq C^{-1} \norm{u}_{L^\infty}^2.
	\end{equation}
	The main problem is what happens to very small radius. Given $x\in M$, we let $r$ be the infimum of the radius such that
	\(
	H_\phi(x,r) \geq  \norm{u}_{L^\infty}^{2-p}
	\).
	The task is to find a small upper bound for $r$. Without loss $r>0$.

	By Corollary \ref{cor:frequencyuniformbound1}, for any $r\leq s\lesssim 1$, we have 
	$
	I_\phi(x,s) \leq C(p)
	$,
	namely
	\[
	D_\phi(x,s) \leq C(p) s^{-1} H_\phi(x,s).
	\]
	By Lemma \ref{lem:Hphiderivative} and the Laplacian computation in Lemma \ref{lem:C0Lap},
	\begin{equation}
		\begin{split}
			|\frac{d}{ds}H_\phi(x,s) | & \leq  C (\int |\nabla u|^2 \phi + \int ( |u|^2 +1) \phi ) 
			\\
			& 
			\leq C D_\phi(x,r) + Cr H_\phi(x,r)
			\\
			&
			\leq C(p) s^{-1} H_\phi(x,s).
		\end{split}
	\end{equation}
	In the second line above, we used (\ref{eqn:meanvaluef(u)}).

	We fix a choice of exponent $0<p<\sqrt{2/3}$. By the lower bound on $H_\phi$ at radius comparable to one, and the differential inequality, we obtain
	\[
	\log H_\phi(x,s) \geq \log (C^{-1}\norm{u}_{L^\infty}^2) + C(p) \log s,\quad r\leq s\lesssim 1,
	\]
	namely
	\[
	H_\phi(x,s) \geq C^{-1} \norm{u}_{L^\infty}^2 s^{C(p)},\quad \text{for } r\leq s\lesssim 1.
	\]
	But by the choice of $r$,
	\[
	H_\phi(x,r)= \norm{u}_{L^\infty}^{2-p},
	\]
	so
	\[
	C^{-1} \norm{u}_{L^\infty}^2 r^{C(p)} \leq H_\phi(x,r)= \norm{u}_{L^\infty}^{2-p},
	\]
	which is equivalent to
	\[
	r\leq  (C\norm{u}_{L^\infty}^{-p})^{1/C(p)}. 
	\]
	Picking $0<q< p/C(p)$, and supposing that $\norm{u}_{L^\infty} $ is large enough, then this infimum radius satisfies
	\[
	r\leq  (C\norm{u}_{L^\infty}^{-p})^{1/C(p)} \leq \norm{u}_{L^\infty}^{-q}. 
	\]
	The frequency bound then follows from Corollary \ref{cor:frequencyuniformbound1}.
\end{proof}

\printbibliography[heading=bibintoc]{}

@book {MR934202,
    AUTHOR = {Atiyah, Michael and Hitchin, Nigel},
     TITLE = {The geometry and dynamics of magnetic monopoles},
    SERIES = {M. B. Porter Lectures},
 PUBLISHER = {Princeton University Press, Princeton, NJ},
      YEAR = {1988},
     PAGES = {viii+134},
      ISBN = {0-691-08480-7},
   MRCLASS = {53C80 (32L10 53C05 58E15 58F07 58G30 81E13)},
  MRNUMBER = {934202},
MRREVIEWER = {G. M. Khenkin},
       DOI = {10.1515/9781400859306},
       URL = {https://doi.org/10.1515/9781400859306},
}

@incollection {MR2893675,
    AUTHOR = {Donaldson, Simon and Segal, Ed},
     TITLE = {Gauge theory in higher dimensions, {II}},
 BOOKTITLE = {Surveys in differential geometry. {V}olume {XVI}. {G}eometry
              of special holonomy and related topics},
    SERIES = {Surv. Differ. Geom.},
    VOLUME = {16},
     PAGES = {1--41},
 PUBLISHER = {Int. Press, Somerville, MA},
      YEAR = {2011},
   MRCLASS = {53C07 (14J32 53C38 53D12 57R58 58D27)},
  MRNUMBER = {2893675},
MRREVIEWER = {Andrew Swann},
       DOI = {10.4310/SDG.2011.v16.n1.a1},
       URL = {https://doi.org/10.4310/SDG.2011.v16.n1.a1},
}

@article {MR2529942,
    AUTHOR = {Hohloch, Sonja and Noetzel, Gregor and Salamon, Dietmar A.},
     TITLE = {Hypercontact structures and {F}loer homology},
   JOURNAL = {Geom. Topol.},
  FJOURNAL = {Geometry \& Topology},
    VOLUME = {13},
      YEAR = {2009},
    NUMBER = {5},
     PAGES = {2543--2617},
      ISSN = {1465-3060},
   MRCLASS = {53D40 (53C26)},
  MRNUMBER = {2529942},
MRREVIEWER = {Hansj\"{o}rg Geiges},
       DOI = {10.2140/gt.2009.13.2543},
       URL = {https://doi.org/10.2140/gt.2009.13.2543},
}

@incollection {MR1708781,
    AUTHOR = {Taubes, Clifford Henry},
     TITLE = {Nonlinear generalizations of a {$3$}-manifold's {D}irac operator},
 BOOKTITLE = {Trends in mathematical physics ({K}noxville, {TN}, 1998)},
    SERIES = {AMS/IP Stud. Adv. Math.},
    VOLUME = {13},
     PAGES = {475--486},
 PUBLISHER = {Amer. Math. Soc., Providence, RI},
      YEAR = {1999},
   MRCLASS = {58J60 (53C26 53C27 57R57 58J05)},
  MRNUMBER = {1708781},
MRREVIEWER = {Yi-Jen Lee},
       DOI = {10.1090/amsip/013/37},
       URL = {https://doi.org/10.1090/amsip/013/37},
}

@article {MR3718486,
    AUTHOR = {Walpuski, Thomas},
     TITLE = {A compactness theorem for {F}ueter sections},
   JOURNAL = {Comment. Math. Helv.},
  FJOURNAL = {Commentarii Mathematici Helvetici. A Journal of the Swiss Mathematical Society},
    VOLUME = {92},
      YEAR = {2017},
    NUMBER = {4},
     PAGES = {751--776},
      ISSN = {0010-2571},
   MRCLASS = {58E20 (53C26)},
  MRNUMBER = {3718486},
MRREVIEWER = {Ling He},
       DOI = {10.4171/CMH/423},
       URL = {https://doi.org/10.4171/CMH/423},
}

@article{esfahani2023towards,
  title={Towards a monopole Fueter Floer homology I: a compactness theorem},
  author={Esfahani, Saman Habibi},
  journal={arXiv preprint arXiv:2305.09456},
  year={2023}
}

@book {MR4495257,
    AUTHOR = {Habibi Esfahani, Saman},
     TITLE = {Monopoles, {S}ingularities and {H}yperkahler {G}eometry},
      NOTE = {Thesis (Ph.D.)--State University of New York at Stony Brook},
 PUBLISHER = {ProQuest LLC, Ann Arbor, MI},
      YEAR = {2022},
     PAGES = {205},
      ISBN = {979-8351-45288-3},
   MRCLASS = {Thesis},
  MRNUMBER = {4495257},
       URL ={http://gateway.proquest.com/openurl?url_ver=Z39.88-2004&rft_val_fmt=info:ofi/fmt:kev:mtx:dissertation&res_dat=xri:pqm&rft_dat=xri:pqdiss:29327731},
}

@article {MR3941492,
    AUTHOR = {Bellettini, Costante and Tian, Gang},
     TITLE = {Compactness results for triholomorphic maps},
   JOURNAL = {J. Eur. Math. Soc. (JEMS)},
  FJOURNAL = {Journal of the European Mathematical Society (JEMS)},
    VOLUME = {21},
      YEAR = {2019},
    NUMBER = {5},
     PAGES = {1271--1317},
      ISSN = {1435-9855},
   MRCLASS = {58E20 (53C26)},
  MRNUMBER = {3941492},
MRREVIEWER = {Ling He},
       DOI = {10.4171/JEMS/860},
       URL = {https://doi.org/10.4171/JEMS/860},
}

@article {MR4753494,
    AUTHOR = {Lotay, Jason D. and Oliveira, Gon\c calo},
     TITLE = {Special {L}agrangians, {L}agrangian mean curvature flow and
              the {G}ibbons-{H}awking ansatz},
   JOURNAL = {J. Differential Geom.},
  FJOURNAL = {Journal of Differential Geometry},
    VOLUME = {126},
      YEAR = {2024},
    NUMBER = {3},
     PAGES = {1121--1184},
      ISSN = {0022-040X,1945-743X},
   MRCLASS = {53C26 (53D12)},
  MRNUMBER = {4753494},
       DOI = {10.4310/jdg/1717348872},
       URL = {https://doi.org/10.4310/jdg/1717348872},
}

@article {MR535151,
    AUTHOR = {Gibbons, G. W. and Pope, C. N.},
     TITLE = {The positive action conjecture and asymptotically {E}uclidean
              metrics in quantum gravity},
   JOURNAL = {Comm. Math. Phys.},
  FJOURNAL = {Communications in Mathematical Physics},
    VOLUME = {66},
      YEAR = {1979},
    NUMBER = {3},
     PAGES = {267--290},
      ISSN = {0010-3616,1432-0916},
   MRCLASS = {83C45 (53C05)},
  MRNUMBER = {535151},
MRREVIEWER = {Mauro\ Francaviglia},
       URL = {http://projecteuclid.org/euclid.cmp/1103905050},
}

@article {MR2869404,
    AUTHOR = {Imazato, Harunobu and Mizoguchi, Shun'ya and Yata, Masaya},
     TITLE = {Taub-{NUT} crystal},
   JOURNAL = {Internat. J. Modern Phys. A},
  FJOURNAL = {International Journal of Modern Physics A. Particles and
              Fields. Gravitation. Cosmology},
    VOLUME = {26},
      YEAR = {2011},
    NUMBER = {30-31},
     PAGES = {5143--5169},
      ISSN = {0217-751X,1793-656X},
   MRCLASS = {81T30},
  MRNUMBER = {2869404},
       DOI = {10.1142/S0217751X11054930},
       URL = {https://doi.org/10.1142/S0217751X11054930},
}

@article {MR1780369,
    AUTHOR = {Hanany, Amihay and Pioline, Boris},
     TITLE = {({A}nti-)instantons and the {A}tiyah-{H}itchin manifold},
   JOURNAL = {J. High Energy Phys.},
  FJOURNAL = {The Journal of High Energy Physics},
      YEAR = {2000},
    NUMBER = {7},
     PAGES = {Paper 1, 23},
      ISSN = {1126-6708,1029-8479},
   MRCLASS = {81T13 (53C07 53C26 53C80 81T30)},
  MRNUMBER = {1780369},
MRREVIEWER = {Anton\ N.\ Kapustin},
       DOI = {10.1088/1126-6708/2000/07/001},
       URL = {https://doi.org/10.1088/1126-6708/2000/07/001},
}

@article {MR4271388,
    AUTHOR = {Schroers, B. J. and Singer, M. A.},
     TITLE = {{$D_k$} gravitational instantons as superpositions of {A}tiyah-{H}itchin and {T}aub-{NUT} geometries},
   JOURNAL = {Q. J. Math.},
  FJOURNAL = {The Quarterly Journal of Mathematics},
    VOLUME = {72},
      YEAR = {2021},
    NUMBER = {1-2},
     PAGES = {277--337},
      ISSN = {0033-5606,1464-3847},
   MRCLASS = {53C26 (53C07)},
  MRNUMBER = {4271388},
       DOI = {10.1093/qmath/haab002},
       URL = {https://doi.org/10.1093/qmath/haab002},
}

@article {MR804459,
    AUTHOR = {Hurtubise, Jacques},
     TITLE = {Monopoles and rational maps: a note on a theorem of
              {D}onaldson},
   JOURNAL = {Comm. Math. Phys.},
  FJOURNAL = {Communications in Mathematical Physics},
    VOLUME = {100},
      YEAR = {1985},
    NUMBER = {2},
     PAGES = {191--196},
      ISSN = {0010-3616,1432-0916},
   MRCLASS = {53C80 (32L10 53C05 58C10 81E13)},
  MRNUMBER = {804459},
MRREVIEWER = {N.\ J.\ Hitchin},
       URL = {http://projecteuclid.org/euclid.cmp/1103943443},
}

@article {MR769355,
    AUTHOR = {Donaldson, S. K.},
     TITLE = {Nahm's equations and the classification of monopoles},
   JOURNAL = {Comm. Math. Phys.},
  FJOURNAL = {Communications in Mathematical Physics},
    VOLUME = {96},
      YEAR = {1984},
    NUMBER = {3},
     PAGES = {387--407},
      ISSN = {0010-3616,1432-0916},
   MRCLASS = {58E99 (32G13 57R99 81E13)},
  MRNUMBER = {769355},
MRREVIEWER = {Jacques\ Hurtubise},
       URL = {http://projecteuclid.org/euclid.cmp/1103941858},
}

@article{doan2022holomorphic,
  title={Holomorphic Floer theory and the Fueter equation},
  author={Doan, Aleksander and Rezchikov, Semon},
  journal={arXiv preprint arXiv:2210.12047},
  year={2022}
}

@article {MR2980921,
    AUTHOR = {Haydys, Andriy},
     TITLE = {Gauge theory, calibrated geometry and harmonic spinors},
   JOURNAL = {J. Lond. Math. Soc. (2)},
  FJOURNAL = {Journal of the London Mathematical Society. Second Series},
    VOLUME = {86},
      YEAR = {2012},
    NUMBER = {2},
     PAGES = {482--498},
      ISSN = {0024-6107,1469-7750},
   MRCLASS = {53C38 (53C07 53C29)},
  MRNUMBER = {2980921},
       DOI = {10.1112/jlms/jds008},
       URL = {https://doi.org/10.1112/jlms/jds008},
}

@article {MR4596625,
    AUTHOR = {Zhang, Boyu},
     TITLE = {Rectifiability and {M}inkowski bounds for the zero loci of
              {$Z/2$} harmonic spinors in dimension 4},
   JOURNAL = {Comm. Anal. Geom.},
  FJOURNAL = {Communications in Analysis and Geometry},
    VOLUME = {30},
      YEAR = {2022},
    NUMBER = {7},
     PAGES = {1633--1681},
      ISSN = {1019-8385,1944-9992},
   MRCLASS = {53C27},
  MRNUMBER = {4596625},
MRREVIEWER = {Volker\ Branding},
}

@article {MR4340726,
    AUTHOR = {Walpuski, Thomas and Zhang, Boyu},
     TITLE = {On the compactness problem for a family of generalized
              {S}eiberg-{W}itten equations in dimension 3},
   JOURNAL = {Duke Math. J.},
  FJOURNAL = {Duke Mathematical Journal},
    VOLUME = {170},
      YEAR = {2021},
    NUMBER = {17},
     PAGES = {3891--3934},
      ISSN = {0012-7094,1547-7398},
   MRCLASS = {53C07 (35J70 57K31)},
  MRNUMBER = {4340726},
MRREVIEWER = {Dave\ Auckly},
       DOI = {10.1215/00127094-2021-0005},
       URL = {https://doi.org/10.1215/00127094-2021-0005},
}

@incollection {MR1798140,
    AUTHOR = {Taubes, Clifford H.},
     TITLE = {{$\rm Gr=SW$}: counting curves and connections [1761081]},
 BOOKTITLE = {Seiberg {W}itten and {G}romov invariants for symplectic
              4-manifolds},
    SERIES = {First Int. Press Lect. Ser.},
    VOLUME = {2},
     PAGES = {275--401},
 PUBLISHER = {Int. Press, Somerville, MA},
      YEAR = {2000},
      ISBN = {1-57146-061-6},
   MRCLASS = {53D45 (32Q65 53D35 57R17 57R57)},
  MRNUMBER = {1798140},
}

@article{houghton1996monopole,
  title={Monopole Scattering with a Twist},
  author={Houghton, Conor and Sutcliffe, Paul},
  journal={arXiv preprint hep-th/9601148},
  year={1996}
}

@article {MR3948228,
    AUTHOR = {Foscolo, Lorenzo},
     TITLE = {A{LF} gravitational instantons and collapsing {R}icci-flat
              metrics on the {$K3$} surface},
   JOURNAL = {J. Differential Geom.},
  FJOURNAL = {Journal of Differential Geometry},
    VOLUME = {112},
      YEAR = {2019},
    NUMBER = {1},
     PAGES = {79--120},
      ISSN = {0022-040X,1945-743X},
   MRCLASS = {53C26 (53C07 53C21 53C25 53C42)},
  MRNUMBER = {3948228},
MRREVIEWER = {Andrew\ Swann},
       DOI = {10.4310/jdg/1557281007},
       URL = {https://doi.org/10.4310/jdg/1557281007},
}

@article {MR1392287,
    AUTHOR = {Bielawski, Roger},
     TITLE = {Monopoles, particles and rational functions},
   JOURNAL = {Ann. Global Anal. Geom.},
  FJOURNAL = {Annals of Global Analysis and Geometry},
    VOLUME = {14},
      YEAR = {1996},
    NUMBER = {2},
     PAGES = {123--145},
      ISSN = {0232-704X,1572-9060},
   MRCLASS = {53C07 (53C80 58D27 81T13)},
  MRNUMBER = {1392287},
MRREVIEWER = {Michael\ Murray},
       DOI = {10.1007/BF00127970},
       URL = {https://doi.org/10.1007/BF00127970},
}

@article{taubes2014zero,
  title={The zero loci of Z/2 harmonic spinors in dimension 2, 3 and 4},
  author={Taubes, Clifford Henry},
  journal={arXiv preprint arXiv:1407.6206},
  year={2014}
}

@article {MR4493580,
    AUTHOR = {Kottke, Chris and Singer, Michael},
     TITLE = {Partial compactification of monopoles and metric asymptotics},
   JOURNAL = {Mem. Amer. Math. Soc.},
  FJOURNAL = {Memoirs of the American Mathematical Society},
    VOLUME = {280},
      YEAR = {2022},
    NUMBER = {1383},
     PAGES = {vii+110},
      ISSN = {0065-9266,1947-6221},
      ISBN = {978-1-4704-5541-5; 978-1-4704-7283-2},
   MRCLASS = {53C07 (58J99 81T13)},
  MRNUMBER = {4493580},
MRREVIEWER = {\'Akos\ Nagy},
       DOI = {10.1090/memo/1383},
       URL = {https://doi.org/10.1090/memo/1383},
}

@article {MR2663735,
    AUTHOR = {De Lellis, Camillo and Spadaro, Emanuele Nunzio},
     TITLE = {{$Q$}-valued functions revisited},
   JOURNAL = {Mem. Amer. Math. Soc.},
  FJOURNAL = {Memoirs of the American Mathematical Society},
    VOLUME = {211},
      YEAR = {2011},
    NUMBER = {991},
     PAGES = {vi+79},
      ISSN = {0065-9266,1947-6221},
      ISBN = {978-0-8218-4914-9},
   MRCLASS = {49Q20 (35J50)},
  MRNUMBER = {2663735},
MRREVIEWER = {Michele\ Miranda},
       DOI = {10.1090/S0065-9266-10-00607-1},
       URL = {https://doi.org/10.1090/S0065-9266-10-00607-1},
}

@book {MR1777737,
    AUTHOR = {Almgren, Jr., Frederick J.},
     TITLE = {Almgren's big regularity paper},
    SERIES = {World Scientific Monograph Series in Mathematics},
    VOLUME = {1},
      NOTE = {$Q$-valued functions minimizing Dirichlet's integral and the
              regularity of area-minimizing rectifiable currents up to
              codimension 2,
              With a preface by Jean E.\ Taylor and Vladimir Scheffer},
 PUBLISHER = {World Scientific Publishing Co., Inc., River Edge, NJ},
      YEAR = {2000},
     PAGES = {xvi+955},
      ISBN = {981-02-4108-9},
   MRCLASS = {49-02 (35J20 49N60 49Q20 58E12)},
  MRNUMBER = {1777737},
MRREVIEWER = {Brian\ Cabell\ White},
}

@article {MR4122245,
    AUTHOR = {Doan, Aleksander and Walpuski, Thomas},
     TITLE = {Deformation theory of the blown-up {S}eiberg-{W}itten equation
              in dimension three},
   JOURNAL = {Selecta Math. (N.S.)},
  FJOURNAL = {Selecta Mathematica. New Series},
    VOLUME = {26},
      YEAR = {2020},
    NUMBER = {3},
     PAGES = {Paper No. 48, 48},
      ISSN = {1022-1824,1420-9020},
   MRCLASS = {53C07 (81T13)},
  MRNUMBER = {4122245},
MRREVIEWER = {Thomas\ Gibbs\ Leness},
       DOI = {10.1007/s00029-020-00574-6},
       URL = {https://doi.org/10.1007/s00029-020-00574-6},
}

@article {MR4255067,
    AUTHOR = {Doan, Aleksander and Walpuski, Thomas},
     TITLE = {On the existence of harmonic {$\rm Z_2$} spinors},
   JOURNAL = {J. Differential Geom.},
  FJOURNAL = {Journal of Differential Geometry},
    VOLUME = {117},
      YEAR = {2021},
    NUMBER = {3},
     PAGES = {395--449},
      ISSN = {0022-040X,1945-743X},
   MRCLASS = {53C27 (57K41)},
  MRNUMBER = {4255067},
MRREVIEWER = {Nicolas\ Ginoux},
       DOI = {10.4310/jdg/1615487003},
       URL = {https://doi.org/10.4310/jdg/1615487003},
}

@article {MR4085667,
    AUTHOR = {Doan, Aleksander and Walpuski, Thomas},
     TITLE = {On counting associative submanifolds and {S}eiberg-{W}itten
              monopoles},
   JOURNAL = {Pure Appl. Math. Q.},
  FJOURNAL = {Pure and Applied Mathematics Quarterly},
    VOLUME = {15},
      YEAR = {2019},
    NUMBER = {4},
     PAGES = {1047--1133},
      ISSN = {1558-8599,1558-8602},
   MRCLASS = {53C38 (14N35 57R57)},
  MRNUMBER = {4085667},
MRREVIEWER = {Kotaro\ Kawai},
       DOI = {10.4310/PAMQ.2019.v15.n4.a4},
       URL = {https://doi.org/10.4310/PAMQ.2019.v15.n4.a4},
}

@phdthesis{foscolo2013moduli,
  title={On moduli spaces of periodic monopoles and gravitational instantons},
  author={Foscolo, Lorenzo},
  year={2013},
  school={Imperial College London}
}

\noindent
\author{Department of Mathematics, Duke University, 120 Science Dr, Durham, NC, US,\\ 27708-0320} \\ E-mail address: \href{ mailto:Saman.HabibiEsfahani@msri.org}{saman.habibiesfahani@duke.edu}

\vspace{10pt}

\noindent
\author{Department of Pure Mathematics and Mathematical Statistics, Cambridge University, Wilberforce Road, Cambridge, UK
} \\ E-mail address: \href{ mailto:yl454@cam.ac.uk}{yl454@cam.ac.uk}

\end{document}